\documentclass{article}
\usepackage{amsmath,amssymb,color,graphicx,bbm,amsthm,verbatim}
\usepackage{subfig}
\usepackage{float}
\usepackage[font=small,labelfont=bf]{caption}
\usepackage{caption}
\usepackage{soul}%strikeout \st
\usepackage{cancel}
\usepackage{xcolor}
\usepackage{algorithm,algorithmic}
\usepackage{fullpage}
\usepackage{quoting}
\usepackage{url}
\usepackage[shortlabels]{enumitem}

\newcommand{\tr}[1]{\mathop{\mbox{Tr}}\left({#1}\right)}

\newcommand{\R}{{\mathbb{R}}}

\newcommand{\E}{\mathop{\mathbb{E}}}

\newcommand{\N}{\mathcal{N}}

\newcommand{\B}{\mathcal{B}}

\newcommand{\DD}{\mathbb{D}}

\newcommand{\CC}{\Lambda}

\newcommand{\cyan}[1]{USE MO MACRO}

\newcommand{\beq}{\begin{equation}}
\newcommand{\eeq}{\end{equation}}
\newcommand{\beqf}{\begin{flalign}}
\newcommand{\eeqf}{\end{flalign}}

\newtheorem{theorem}{Theorem}[section]

\newtheorem{lemma}      [theorem]   {Lemma}
\newtheorem{definition}   [theorem]       {Definition}
\newtheorem{assumption}   [theorem]      {Assumption}
\newtheorem{remark}    [theorem]      {Remark}
\newtheorem{proposition} [theorem]   {Proposition}
\newtheorem{fact}    [theorem]      {Fact}

\title{A $J$-Symmetric Quasi-Newton Method for Minimax Problems}
\author{Azam Asl\thanks{University of Chicago Booth School of Business.}
	\and 
	Haihao Lu\thanks{University of Chicago Booth School of Business.}
    \and
    Jinwen Yang\thanks{University of Chicago, Department of Statistics.}
}

\begin{document}
\maketitle
%\SL{Other options of the title:
%\begin{itemize}
%    \item A Quasi-Newton Method for Minimax Problems
%    \item A Practical Quasi-Newton Method for Minimax Problems
%\end{itemize}
%}

\begin{abstract}
Minimax problems have gained tremendous attentions across the optimization and machine learning community recently. In this paper, we introduce a new quasi-Newton method for the minimax problems, which we call $J$-symmetric quasi-Newton method. The method is obtained by exploiting the $J$-symmetric structure of the second-order derivative of the objective function in minimax problem. We show that the Hessian estimation (as well as its inverse) can be updated by a rank-2 operation, and it turns out that the update rule is a natural generalization of the classic Powell symmetric Broyden (PSB) method from minimization problems to minimax problems. In theory, we show that our proposed quasi-Newton algorithm enjoys local Q-superlinear convergence to a desirable solution under standard regularity conditions. Furthermore, we introduce a trust-region variant of the algorithm that enjoys global R-superlinear convergence. Finally, we present numerical experiments that verify our theory and show the effectiveness of our proposed algorithms compared to Broyden's method and the extragradient method on three classes of minimax problems.
\end{abstract}

\section{Introduction}\label{sec:intro}
Our problem of interest in this paper is the minimax problem (a.k.a. saddle-point problem)
\beq\label{minimax}
\min_{x \in \R^n}\max_{w \in \R^m} L(x,w) \ ,
\eeq
where $L(x, w)$ is a smooth objective in both $x$ and $w$, and we call $x$ the primal variable and $w$ the dual variable. Minimax problem is one of the most important classes of optimization problems, with a long research history and wide applications. The earliest motivation for minimax problems may come from the Lagrangian form of constrained optimization problems; see~\cite{bertsekas1982constrained} and the references therein. Another major application of minimax problem is zero sum games~\cite{osborne1994course}. More recently, minimax problem~\eqref{minimax} has regained significant attentions across the optimization and machine learning communities, mainly due to their applications in machine learning, such as generative adversarial networks (GANs)~\cite{GAN}, reinforcement learning \cite{reinforce}, robust training \cite{robust_learn}, image processing~\cite{chambolle2011first}, and applications in classic constrained optimization, such as linear programming~\cite{applegate2021practical}.

Here, we develop a quasi-Newton method for the minimax problem~\eqref{minimax}. Quasi-Newton method is a successful optimization method for minimization problems~\cite[Chapter 6]{NW06}. While Newton's method enjoys the fast local quadratic convergence, the iteration cost to access the Hessian and to solve the linear equation can be prohibitive when solving large instances. Instead, quasi-Newton method constructs an approximate Hessian (more often constructs an approximate inverse Hessian) and updates it with a low-rank operation at each iteration, which can significantly reduce the iteration cost. Under proper regularity conditions, one can show that the quasi-Newton method has local superlinear convergence {or} global linear convergence. Some famous quasi-Newton updates for minimization problems include BFGS formula \cite{Broyden_single,fletcher1970new,Goldfarb,shanno1970conditioning}, DFP formula \cite{Davidon,fletcher1963rapidly}, PSB formula \cite{POW_PSB}, etc.
% \red{Find some citations for these two formulas} \Asl{Done. } \red{These should be the original papers not the textbook}. \Asl{revised.}\
The low cost-per-iteration and superlinear eventual convergence make the quasi-Newton method a highly efficient algorithm. It is widely used in practice and is listed as one of the ten algorithms with the greatest impact on the development and practice of science and engineering in the early 21st century~\cite{higham_2016}.

Surprisingly, there has been very limited research on quasi-Newton methods for minimax problems. As a special case of nonlinear equations or as a special case of variational inequalities, one can adapt quasi-Newton methods for these problems to solve minimax problems \eqref{minimax}. In particular, Broyden's (``good'' or ``bad'') methods~\cite{broyden1965class} are quasi-Newton methods for solving generic nonlinear equations, and we can use them to solve the KKT system of \eqref{minimax}. In the 1990s, Burke and Qian proposed a variable metric proximal point method for monotone variational inequality~\cite{BurkePPM1, BurkePPM2}, where they effectively introduced a proximal point variant of quasi-Newton method and used Broyden's formula to update the second-order term. 
%\Asl{I can't tell what do you mean by both approaches here:} \red{PTAL} \Asl{looks good to me.}
Broyden's method and Burke and Qian's method target at a much larger class of problems, and \emph{do not} utilize the structure of minimax problems. In contrast, in this paper, we propose a new quasi-Newton method specialized for minimax problems that utilizes the structure of the second-order derivative of minimax problems. The utilization of such structures has the following advantages compared to existing quasi-Newton methods and first-order methods:
\begin{itemize}
    \item Many classic quasi-Newton methods, such as the BFGS formula and the DFP formula, target minimization problems and construct symmetric and positive definite approximations of Hessian. These methods do not directly work for minimax problems, where the second-order derivative is no longer symmetric.
    \item Broyden's formula targets at finding root of nonlinear equations and does not require any structure on the Jacobian estimation. While it is very general, it ignores the meaningful information of the Jacobian structure in minimax problems, and it is numerically unstable even when solving simple bilinear minimax problems (as shown in Section \ref{num_exp}). Furthermore, it is unclear how to properly initialize the Jacobian estimation of Broyden's method for minimax problems, which may lead to numerical issues.
    \item Compared with first-order methods, such as EGM, quasi-Newton method enjoys a local superlinear convergence rate and the convergence speed does not heavily rely on the condition number of the problem.
\end{itemize}

Throughout the paper, we assume the objective function $L(x,w)$ is third-order differentiable. For notational convenience, we denote $z=(x,w)\in\R^{m+n}$ as the primal-dual solution pair, $F(z)=[\nabla_x L(x,w), -\nabla_w L(x,w)]$ as the gradient (more precisely gradient for the primal and negative gradient for the dual) of $L(x,w)$. $F(z)$ is the cornerstone of first-order methods for minimax problems. For example, the gradient descent ascent (GDA) method has an iteration update $z_{k+1}=z_k - s F(z_k)$, the proximal point method (PPM) has an iteration update $z_{k+1}=z_k-sF(z_{k+1})$, and the extragradient method (EGM) has an iteration update $z'_{k}=z_k - s F(z_k), z_{k+1}=z_{k}-sF(z'_k)$. 

When turning to second-order methods, we denote
\beq\label{Jac}
\nabla F(z) = 
\left[
\begin{matrix}
	\nabla_{xx} L(x,w) & \nabla_{xw} L(x,w)\\
	-\nabla_{xw} L(x,w)^T & -\nabla_{ww} L(x,w)
\end{matrix}
\right] 
\eeq
as the Jacobian of $F(z)$. Then the standard
% The first-order Nash equilibrium  $z^*$ to \eqref{minimax} is a stationary point $z^*$ of problem \eqref{minimax}, i.e., $F(z^*)=0$. The standard 
Newton's method has an iteration update 
$$ z_{k+1} = z_k -\nabla F(z_k)^{-1}F(z_k) \ . $$
We here focus on quasi-Newton method with an iteration update
$$ z_{k+1} = z_k -B_k^{-1}F(z_k) \ , $$
where $B_k$ is an approximation of $\nabla F(z_k)$. A key observation is that $\nabla F(z)$ defined in \eqref{Jac} is symmetric on the main diagonal terms and skew-symmetric on the anti-diagonal terms. This type of matrix is called $J$-symmetric in the related literature~\cite{mackey}.  A $J$-symmetric matrix has many desirable numerical properties, see, for example, \cite[Theorems 3.6 and 3.7]{benzi2005numerical} and \cite[Lemma 1.1]{ben_golub}. $J$-symmetric matrix naturally appears and has been used in numerical analysis and applied mathematics. For example, \cite{ben_golub} introduces a $J$-symmetric system as a preconditioner for Krylov subspace methods for solving nonlinear equations. 
\cite{sidi} uses $J$-symmetric matrices as preconditioner when solving discrete Navier-Stokes equations in incompressible fluid mechanics.  

{ When the minimax problem is convex-concave, $\nabla F(z)$ is $J$-Symmetric. Similar to the fact that positive semidefinite Hessian is the cornerstone of BFGS method for solving a minimization problems, the $J$-symmetric structure is the cornerstone of our quasi-Newton update for solving minimax problems, and the utilization of the $J$-symmetric structure is the major novelty of our approach over existing literature.}

The major contributions of our work can be summarized as follows:
\begin{itemize}
	\item We introduce a new quasi-Newton update for minimax problems that comes from the $J$-symmetric structure of the Jacobian of the minimax objective. We show that we can efficiently update the Jacobian estimation as well as its inverse in our method via a rank-2 update. It turns out the update rule is a natural generalization of Powell's symmetric Broyden (PSB) update from minimization problems to minimax problems.
	\item We prove that the proposed unit-step quasi-Newton method enjoys local Q-superlinear convergence towards a stationary point of the minimax problem via the \textit{bounded deterioration} technique. Furthermore, we propose a trust-region variant of the proposed quasi-Newton method and prove its global R-superlinear convergence. The convergence results do not require the convexity-concavity of the objective function in the minimax problem.
	\item We present preliminary numerical experiments, which verifies our theory and showcases that our proposed methods are more stable and faster compared to Broyden's update when solving minimax problems. They also enjoy faster convergence compared to first-order methods such as EGM.
\end{itemize}

%###########################################################
\subsection{Applications of Minimax Problems}
We here briefly discuss three applications of minimax problems.

\textbf{(Linear equality-constrained convex optimization.)} Consider a constrained optimization problem of the form
\begin{align}
\begin{split}\label{quad_prog}
    \min_x & ~f(x) \\
	\mathrm{s.t.} & ~Ax=b \ .
\end{split}
\end{align}%\mathcal{L}
This type of problem is the subproblem in sequential quadratic programming~\cite[Chapter 6]{NW06} and arises in computational physics~\cite{benzi2005numerical}.
The Lagrangian of is $L(x,w) = f(x) + w^T(Ax-b)$, where $w$ is the Lagrange multiplier 
and thus \eqref{quad_prog} is equivalent to
\[
    \min_x\max_w L(x,w) = \min_x\max_w f(x) + w^T(Ax-b) \ .
\]

%##################################
%#################################
\textbf{(Inequality-constrained convex optimization.)} Consider a generic constrained convex optimization problem
\begin{align*}
	\min_x & ~ f(x) \\
	\mathrm{s.t.} & ~g(x)\ge 0 \ .% \mathrm{~~and~~~~}x\geq 0.
\end{align*}
Introducing the Lagrangian multiplier $w$ yields
\begin{equation}\label{ic}
    \min_x \max_{w\ge 0} f(x) - w^T g(x)\ .
\end{equation}
Notice that {the dual variables are constrained to be in the non-negative orthant}. We can instead consider a logarithmic-barrier formulation with barrier parameter $\mu$:
\begin{equation}\label{ic-mu}
    \min_x \max_{w} L(x,w;\mu):= f(x) - w^T g(x) +\mu \sum_i \log w_i\ .
\end{equation}
{\eqref{ic-mu} with parameter $\mu$ can be viewed as the central path of the problem \eqref{ic}. It recovers \eqref{ic} as $\mu\rightarrow 0$. One can potentially apply the interior-point method (IPM) to solve \eqref{ic}, which follows from the central path \eqref{ic-mu} by Newton'm ethod. Here we solve \eqref{ic-mu} for a fixed $\mu$ using J-symmetric quasi-Newton algorithm.
}The solution to the above minimax problem identifies an optimal solution to the original minimization problem when $\mu\rightarrow 0$. Indeed, as long as $\mu$ is chosen properly, it provides an approximate solution.

\textbf{(Generative Adversarial Network.)}
Generative Adversarial Network (GAN) \cite{GAN} is a recent development in machine learning, which has many applications in image processing such as producing realistic images \cite{prog_GAN}, quality super-resolution \cite{sup_res} and image-to-image translation \cite{img_tran}.  A GAN is a  minimax problem of the form \eqref{minimax} which is the equilibrium condition of a zero-sum two-player game. The two players are the generator (parameterized by $G$) and the discriminator (parameterized by $D$):
\begin{equation}\label{eq:gan}
    \min_G\max_D \E_{s\sim p}[\log D(s)] +   \E_{e\sim q}[\log( 1- D(G(e)))] \ ,
\end{equation}
where $p$ is the data distribution and $q$ is the latent distribution.  The generator produces a sample, and the discriminator decides whether they are real or fake data. The goal is to learn the best generator which can produce realistic data \cite{var_GAN}. Notice that $G$ and $D$ are usually represented as parameters of neural networks, thus \eqref{eq:gan} is a nonconvex nonconcave minimax problem.

%%%%%%%%%%%%%%%%%%%%%%%%%%%%%%%%%%%%%%%%%%%%%%%%%%%%%%%%%%%%%%
\subsection{Related Literature}
\textbf{Minimax optimization.} Minimax optimization \eqref{minimax} has long history and wide applications. The early work on minimax optimization focuses on a more general problem, monotone variational inequalities. The two classical algorithms for monotone variational inequality/minimax problems are perhaps proximal point method (PPM) proposed by Rockafellar~\cite{rock} and extragradient method (EGM) proposed by Korpelevich~\cite{korp} in 1970s. Later, Nemirovski~\cite{nemirovski2004prox} proposes the mirror prox algorithm, which generalizes EGM with Bregman divergence and builds up the connection between EGM and PPM. 

Motivated by machine learning applications, there is a renewed recent interest in developing efficient first-order algorithms for minimax problems.~\cite{daskalakis2018training} studies an Optimistic Gradient Descent Ascent (OGDA) with applications in GAN. \cite{mokhtari2020unified} presents an interesting observation that OGDA approximates PPM on bilinear problems. ~\cite{SeanODE_recent} proposes a high-resolution ODE framework that can characterize different primal-dual algorithms. \cite{grimmer2020landscape} studies the landscape of PPM and presents examples showing that classic algorithms such as PPM, EGM, gradient descent ascent, and alternating gradient descent ascent may converge to a limit circle on a simple nonconvex-nonconcave example. See \cite{grimmer2020landscape} for a thorough literature review on the recent development of minimax problems. Compared to these first-order methods, our focus is on quasi-Newton methods, and our theoretical results do not rely on the convexity of the objective.

\textbf{Quasi-Newton methods.} Quasi-Newton methods are alternatives to the classical Newton's method. 
% Instead of computing the Newton's direction by solving linear equations using the Jacobian, quasi-Newton methods often formulate an approximate inverse Jacobian and use a low-rank operation to update the inverse. Quasi-Newton methods are  computationally more efficient than Newton's method, thus can solve much larger instances. Due to their high impact on minimization problems, quasi-Newton methods have appeared in the updated list of ``10 algorithms with the greatest influence on the development and practice of science and engineering in the 20th century'' given by Nick Higham in 2016~\cite{higham_2016}.
The first quasi-Newton method was developed by W.C. Davidon in 1959 and was later published in \cite{Davidon} in 1991.  Instead of computing the inverse Hessian at every iteration, Davidon's method obtains a good approximation of it using gradient differences. 
Soon after, Fletcher and Powell realized the efficiency of this method. They studied and popularized Davidon's original formula and established its convergence for convex quadratic functions \cite{DFP_conv}.  This method became known as DFP method.
 BFGS \cite{Goldfarb} is perhaps the  most popular quasi-Newton method \cite{NW06}. It was discovered by Broyden, Fletcher, Goldfarb and Shanno independently in the 1970s. Soon after, Broyden, Dennis and Moré proved the first local and superlinear convergence results for BFGS, DFP as well as other quasi-Newton methods \cite{localB}. 
% In addition to assuming  $G_*$ is  a symmetric positive definite matrix he assumed that the hessian $G$ is lipschitz continuous in a neighbourhood of $z^*$.
Later, Powell \cite{POW76b} presented the first global convergence result of BFGS with an inexact Armijo-Wolfe line search for a general class of smooth convex optimization problems.  \cite{global_conv} extended Powell's result
to a broader class of quasi-Newton methods.

\textbf{Quasi-Newton methods for minimax problems.} While minimax problems and quasi-Newton methods are both well studied individually, there are fairly limited works on quasi-Newton methods for minimax problems. Notice that one can solve minimax problems by finding a root of a corresponding nonlinear equation, thus one can use the classical Broyden's (good and bad) algorithms~\cite{broyden1965class,Broyden_single} for minimax problems. Another line of early research is to use proximal quasi-Newton methods for monotone variational inequalities proposed in~\cite{fukushima, BurkePPM1, BurkePPM2} to solve convex-concave minimax problems. However, both Broyden's methods and the proximal quasi-Newton methods target at a more general class of problems, without considering the special structure of the minimax problems. As a result, these algorithms may not always be stable, even when solving simple bilinear minimax problems, as we see in our numerical experiments. More recently, \cite{VIP_BFGS, imp_gd} proposed different quasi-Newton methods for minimax problems. However, neither of them shows the convergence rate of their algorithms. In contrast to these works, we introduce a new quasi-Newton method for the minimax problem and present its local/global superlinear rate.

\textbf{Trust-region method.} 
Trust-region method is another classic algorithm in numerical optimization. It first defines a region around the current best solution, and then creates a quadratic model that can approximate the objective function in the region and takes a step by solving a subproblem based on this quadratic model. Quasi-Newton methods are often used together with trust-region method~\cite{NW06}. Unlike a line-search method, which picks the direction first and then looks for an acceptable stepsize along that direction, a trust-region method first picks the stepsize  and then looks  for an acceptable direction within that region.  

%%%%%dogleg method
There are different methods to solve the trust-region subproblem. The simplest way is to move along the negative gradient direction to a point within the trust-region. The resulting solution is called Cauchy point. Although the Cauchy point is cheap to calculate, it may perform poorly in some cases. A famous approach to avoid this issue is the dogleg method. The dogleg method was originally introduced by Powell as hybrid method in \cite{hybrid}.
The dogleg point refers to a point on the boundary of the trust-region that is a linear combination of the Cauchy point and the minimizer of the quadratic model, and it is used only when the Cauchy point is strictly inside the trust-region and the minimizer of the quadratic model is strictly outside the trust-region. 
See~\cite{NW06} for more details on the trust-region method.

%##################################
\subsection{Notations}
Throughout this paper, the norm $\| \cdot \|$ denotes the $\ell_2$ norm for a vector or the operator norm (i.e., the $\ell_{2,2}$ norm) for a matrix, unless specified. The norm $\|\cdot\|_F$ refers to the Frobenius norm for a matrix.
As a common notation in quasi-Newton method, $s_k$ denotes the potential step at iteration $k$. When $s_k$ is a sufficient decrease step  and we accept it, we have $s_k= z_{k+1}-z_k$. 
Otherwise, we reject it (equivalently, we take a null step and set $z_{k+1}=z_k$). We use $y_k= F(z_k+s_k) - F(z_k)$ to denote the gradient difference between two consecutive points. We use $J\in\R^{(n+m)\times (n+m)}$ to represent the following block diagonal square matrix:
$$
J = \left[
\begin{matrix}
	I_{n\times n} & 0 \\ 
	0 & -I_{m\times m} 
\end{matrix}
\right] \ .
$$
%\eeq
% It is easy to check that we have 
% %\beq
% $J^2=I$. %\label{J2I} 
% %\eeq  
% Additionally, for any vector $q \in \R^{n+m}$ we have 
% $\| Jq \| = \| q \|$.
%#####################################################################
%#####################################################################
\section{$J$-symmetric  Update}\label{step}
In this section we present our $J$-symmetric update for minimax problems. The major idea is to construct the estimated Jacobian by utilizing the $J$-symmetric structure in $\nabla F(z)$. 
We begin by introducing the following notations for notational convenience:
\begin{equation*}
	D(z) =  \nabla_{xx} L(z) \ ,
	C(z) = -\nabla_{ww} L(z) \ ,
	A(z) = \nabla_{xw} L(z)^T \ .
\end{equation*}
Then the Jacobian  defined in \eqref{Jac} can be rewritten as
%\beq\label{G}
$$\nabla F(z) = \left[
\begin{matrix}
	D(z) & A^T(z)\\
	-A(z) & C(z)
\end{matrix}
\right] \ , $$
where the main diagonal terms are symmetric and the main off-diagonal terms are anti-symmetric. This structure is called $J$-symmetric~\cite{ben_golub,benzi2005numerical}.
Recall that matrix $J=\left[
\begin{matrix}
	I_{n\times n} & 0 \\ 
	0 & -I_{m\times m} 
\end{matrix}
\right] \ .$ 
It is easy to check it holds for a $J$-symmetric matrix $M$ that
$$ M = JM^TJ  \text{ and } JM = M^TJ \ . $$
% Recall that Newton method has an iteration update
% $$ z_{k+1} = z_k -\nabla F(z_k)^{-1}F(z_k) \ . $$
The general scheme of the quasi-Newton method consists of iteration updates of the following form
\beq\label{seq}
z_{k+1} = z_k -B_k^{-1}F(z_k) \ ,
\eeq
where $B_k$ denotes the approximation to the current Jacobian $\nabla F(z_k)$, and we hope to obtain a better and better approximation over time. In particular, we seek update rules from $B_k$ to $B_{k+1}$ such that:
\begin{enumerate}[(a)]
	\item $B_{k+1}$ is a good approximation to  $\nabla F(z_{k+1})$. %\SL{Have the GBSS structure as a separate point.} 
	\item $B_{k+1}$ is a $J$-symmetric matrix.
	\item $B_{k+1}$ is not too far away from $B_k$.
	\item There is an efficient way for computing $B_{k+1}$ from $B_k$ by a low rank update. %For now a quadratic cost per iteration is acceptable.
\end{enumerate}
A common requirement to satisfy (a) is that $B_{k+1}$ should satisfy the secant condition
\beq \label{sec}
y_k = B_{k+1}s_k \  .
\eeq
The secant condition imposes only $n+m$ constraints on $B_{k+1}$ and even after taking into consideration the required $J$-symmetric structure, we are still left with many degrees of freedom to pick $B_{k+1}$. 
%Similar to BFGS method, \asl{BFGS uses a different norm, namely the norm defined by the average Hessian. We first wanted to use that norm but that did not seem to easy. }
In addition, we select $B_{k+1}$ such that it is the closest matrix to $B_k$ in Frobenius norm. In summary, $B_{k+1}$ is given by solving the following minimization problem:
%\begin{gathered}
\begin{align}
	\begin{split}
		\min_B ~ &  \frac{1}{2}\| B-B_k \|_F^2 \\ 
		\mathrm{s.t.}~ & Bs_k -y_k=0      \\
		& D = D^T,~~~ C = C^T ~~\mathrm{~and~~~}  
		B = 
		\left[
		\begin{matrix}
			D & A^T\\
			-A &C
		\end{matrix}
		\right]   \ .
	\end{split}
	\label{Bopt}
\end{align}
%\end{gathered}
The last line imposes the $J$-symmetric structure on $B$. Notice that the constraint set is a convex set and the objective is strongly convex, thus \eqref{Bopt} is a convex optimization problem with a unique solution, and furthermore:
\begin{proposition}
	The unique solution to the constrained optimization problem \eqref{Bopt} is given by
	\beq\label{Ebar3} 
	B_{k+1} = B_k + \dfrac{Js_k(y_k-B_ks_k)^TJ}{s_k^Ts_k} + \dfrac{(y_k-B_ks_k)s_k^T}{s_k^Ts_k} - \dfrac{(Js_k)^T(y_k-B_ks_k)Js_ks_k^T}{(s_k^Ts_k)^2} \ ,
	\eeq
	which is a rank-2 update. 
\end{proposition}
\begin{proof}
	Define 
	%\beq\label{E}
	$E = B-B_k$ and 
	%\eeq
	%\beq\label{r}
	$r = y_k-B_ks_k$. 
	%\eeq
  It is easy to see then the minimization problem \eqref{Bopt} is equivalent to the following after changing variables
	\begin{align}
		\min_E ~ &  \frac{1}{2}\| E \|_F^2  \label{obj}\\  
		\mathrm{s.t. }~ & Es_k -r =0  \label{secant}\\
		& {D' = D'^T, ~~~ C' = C'^T ~~\mathrm{~and~~~}   \label{anti-symm}
		E = 
		\left[
		\begin{matrix}
			D' & A^T\\
			-A &C'
		\end{matrix}
		\right] \ . }  
	\end{align}
	$J$-symmetry constraint \eqref{anti-symm} is equivalent to that $E+E^T$ is a block diagonal matrix, and $E-E^T$ is a block anti-diagonal matrix. We dualize these two constraints and let $\Gamma_A, \Gamma_D\in\R^{(m+n)\times(m+n)}$ be the Lagrange multipliers of the condition involving $E+E^T$ and $E-E^T$, respectively. Then $\Gamma_A$ is a block anti-diagonal matrix and $\Gamma_D$ is a block diagonal matrix. Let $\lambda\in \R^{m+n}$ be the Lagrange multiplier corresponding to the secant condition \eqref{secant}.
	Then, the Lagrangian can be written as:
	\[
	\Phi(E;~\lambda,\Gamma_D,\Gamma_A) = \frac{1}{2} \tr{EE^T} + \lambda^T (Es_k-r) + \tr{\Gamma_D(E-E^T)}  + \tr{\Gamma_A(E+E^T)} \ .
	\]
	%where the vector $\lambda\in \R^{m+n}$ is the Lagrange multiplier corresponding to the secant condition \eqref{secant}. 
	Since $ \lambda^T(Es_k-r) = \tr{(Es_k-r)\lambda^T}$, then
	%\beq \label{lag}
	$$ \Phi(E;~\lambda,\Gamma_D,\Gamma_A) = \frac{1}{2} \tr{EE^T} + \tr{(Es_k-r)\lambda^T}+ \tr{\Gamma_D(E-E^T)}  + \tr{\Gamma_A(E+E^T)} \ .  $$
	%\eeq
	The KKT condition requires $\partial \Phi /\partial E = 0$, whereby
	\beq \label{E1}
	E = - \Big( \lambda s_k^T + \Gamma_D^T - \Gamma_D + \Gamma_A^T + \Gamma_A \Big) \ .
	\eeq
	Furthermore, we decompose $\lambda s_k^T $ as following
	\beq\label{lam_s}
	\lambda =  
	\begin{bmatrix}
		\lambda_1 \\
		\lambda_2
	\end{bmatrix} \ ,~~
	s_k  = 
	\begin{bmatrix}
		s_1\\
		s_2
	\end{bmatrix}  ,  ~~ 
	\lambda s_k^T = 
	\left[
	\begin{matrix}
		\lambda_1s_1^T & \lambda_1s_2^T\\
		\lambda_2s_1^T & \lambda_2s_2^T
	\end{matrix}
	\right] \ ,
	\eeq
	where $\lambda_1,~s_1 \in \R^n$ and $ \lambda_2,~ s_2 \in  \R^m$.
	Note that the $J$-symmetry constraint \eqref{anti-symm} requires the diagonal blocks in \eqref{E1} are symmetric, so
	\[
	\left[
	\begin{matrix}
		\lambda_1s_1^T & 0\\
		0 & \lambda_2s_2^T
	\end{matrix}
	\right]  
	+ \Gamma_D^T - \Gamma_D   = 
	\left[
	\begin{matrix}
		s_1\lambda_1^T & 0\\
		0 & s_2\lambda_2^T
	\end{matrix}
	\right]  
	+ \Gamma_D - \Gamma_D^T  \ ,
	\]
	thus
	\beq\label{symm}
	\Gamma_D^T - \Gamma_D = -\dfrac{1}{2}
	\left[
	\begin{matrix}
		\lambda_1s_1^T - s_1\lambda_1^T& 0\\
		0 & \lambda_2s_2^T - s_2\lambda_2^T
	\end{matrix}
	\right] \ .
	\eeq
	Furthermore, the block anti-diagonal matrix in \eqref{E1} is skew-symmetric, so
	\[
	\left[
	\begin{matrix}
		0 & \lambda_1s_2^T \\
		\lambda_2s_1^T & 0
	\end{matrix}
	\right]  
	+ \Gamma_A^T + \Gamma_A = - \Big(
	\left[
	\begin{matrix}
		0 & s_1\lambda_2^T \\
		s_2\lambda_1^T & 0
	\end{matrix}
	\right]  
	+ \Gamma_A + \Gamma_A^T \Big) \ ,
	\]
	thus
	\beq\label{skew}
	\Gamma_A^T + \Gamma_A = -\dfrac{1}{2}\left[
	\begin{matrix}
		0 & \lambda_1s_2^T+s_1\lambda_2^T \\
		\lambda_2s_1^T+s_2\lambda_1^T & 0
	\end{matrix}
	\right] \ . 
	\eeq
	Substituting \eqref{symm} and \eqref{skew} back into \eqref{E1} and noticing \eqref{lam_s}, we obtain: 
	\beq\label{E2}
	E = -\dfrac{1}{2}
	\left[
	\begin{matrix}
		\lambda_1s_1^T + s_1\lambda_1^T & \lambda_1s_2^T-s_1\lambda_2^T \\
		\lambda_2s_1^T-s_2\lambda_1^T & \lambda_2s_2^T + s_2\lambda_2^T
	\end{matrix}
	\right] = -\dfrac{1}{2} (\lambda s_k^T + Js_k\lambda^TJ) \ .
	\eeq
	The rest of the proof is to compute the multiplier $\lambda$. Substituting \eqref{E2} into the secant condition \eqref{secant}, we obtain
	\[
	(\lambda s_k^T + Js_k\lambda^TJ)s_k = -2r \ ,
	\]
	and since both $s_k^Ts_k$ and $\lambda^TJs_k$ are scalars, it holds that
	\beq\label{lam1}
	\lambda  = -\dfrac{1}{s_k^Ts_k}\Big(2r + (s_k^TJ\lambda)Js_k \Big) \ .
	\eeq
	Multiplying both sides  with $s_k^TJ$, we arrive at
	\[
	s_k^TJ\lambda  = -\dfrac{2s_k^TJr + (s_k^TJ\lambda)s_k^TJJs_k }{s_k^Ts_k} \ ,
	\]
	which can be further simplified to $
	s_k^TJ\lambda  = -s_k^TJr/(s_k^Ts_k)
	$ by using $J^2=I$.
	Substituting this into \eqref{lam1}, we obtain
	%\beq\label{lam}
	\[\lambda  =  \dfrac{s_k^TJr}{(s_k^Ts_k)^2}Js_k -\dfrac{2}{s_k^Ts_k}r \ .
	\]%\eeq
	Now, by substituting $\lambda$ into \eqref{E2} we obtain
	\begin{align*}
		E = -\dfrac{1}{2} \Big(\dfrac{s_k^TJr}{(s_k^Ts_k)^2}Js_ks_k^T -\dfrac{2}{s_k^Ts_k}rs_k^T  
		+\dfrac{s_k^TJr}{(s_k^Ts_k)^2}Js_ks_k^TJJ -\dfrac{2}{s_k^Ts_k}Js_kr^TJ  \Big) \ .
	\end{align*}
	By noticing $J^2=I$, we conclude that the unique solution of the problem \eqref{obj}-\eqref{anti-symm} is given by:
	%\beq\label{E3}
	\[
	E =  \dfrac{1}{s_k^Ts_k}rs_k^T+\dfrac{1}{s_k^Ts_k}Js_kr^TJ -\dfrac{s_k^TJr}{(s_k^Ts_k)^2}Js_ks_k^T \ .
	\]
	%Since the optimization problem \eqref{obj}-\eqref{anti-symm} is equivalent to 
	Finally by substituting $E=B_{k+1}-B_k$ into this equation we obtain the unique solution of problem \eqref{Bopt} as:% \eqref{E3}
	%\begin{align}\label{Er2}
	\[
	B_{k+1} = B_k + \dfrac{1}{s_k^Ts_k} (r-\dfrac{s_k^TJr}{s_k^Ts_k}Js_k)s_k^T +\dfrac{Js_kr^TJ}{s_k^Ts_k} \ .
	\]
	This equations  reveals the update is a rank-2 update.  By changing the order and plugging in $r = y_k-B_ks_k$ we arrive at \eqref{Ebar3}.
\end{proof}
%####################################################
Next, we show that the inverse update of \eqref{Ebar3} can also be obtained by a low rank update via Sherman-Woodbury identity.
\begin{proposition}
	%	r = F_{k+1}
	Let $r= y_k-B_ks_k$, $H_{k} = B_k^{-1}$ and  $H_{k+1} = B_{k+1}^{-1}$. The inverse update of \eqref{Ebar3} is 
	\begin{align}\label{inv_up}
         	H_{k+1}= Q^{-1} -  \frac{Q^{-1} Js_k(Jr)^T Q^{-1}}{s_k^Ts_k+ (Jr)^TQ^{-1}Js_k} \ ,  \mathrm{~~~where~~~~~~~~~}
             Q^{-1}=  H_k - \frac{H_k JP_kJrs_k^T H_k}{s_k^Ts_k+ s_k^TH_kJP_kJr} \ ,  
	\end{align}
	and
	\beq\label{p2}
	P_k = I - \dfrac{s_ks_k^T}{s_k^Ts_k} \ .
	\eeq
\end{proposition}
\begin{proof}
	Define $a =  r- s_k^TJrJs_k/(s_k^Ts_k) = JP_kJr$ and  $Q = B_k + as_k^T/(s_k^Ts_k)$,  then from  \eqref{Ebar3} we have that
	\begin{align*}
		B_{k+1} = Q +  \frac{Js_k(Jr)^T}{s_k^Ts_k} \ . 
	\end{align*}
	From one application of Sherman-Woodbury to $Q$, we obtain $Q^{-1}$ and from another application to $B_{k+1}$ we obtain $H_{k+1}$ in \eqref{inv_up}.
	%%%.
\end{proof}
% \teal{ 
	% New addition [November 19, 2021]:\\
	% \textbf{Remark.} 
	% If  $D$ and $C$ are symmetric positive definite matrices, then $E$ as defined in \eqref{obj}-\eqref{anti-symm} is invertible and its inverse also has a  $J$-symmetric structure with symmetric positive definite diagonal matrices.  To see this note that 
	% \begin{align}\nonumber
		% 	E  = -\left[
		% 	\begin{matrix}
			% 		-D & 	A^T\\
			% 		A & C
			% 	\end{matrix}
		% 	\right]J \  .   
		% \end{align}
	% The symmetric matrix in the right is a {\it quasi-definite} matrix and is invertible \cite{vanderbei}. Furthermore, its inverse is also   {\it quasi-definite} \cite[Thm. 1]{vanderbei}. Recalling that $J^{-1} = J$ we deduce that $E^{-1}$ exists and has a  $J$-symmetric structure with symmetric positive definite diagonal matrices.  Hence we can write the matrix optimization problem \eqref{obj}-\eqref{anti-symm}  in terms of  the inverse of $E$ and from the symmetry of $y_k$ and $s_k$ get the direct update for $H_{k+1}$ as following:
	% \beq\nonumber%\label{E3_inv} 
	% H_{k+1} = H_k + \dfrac{Jy_k(s_k-H_ky_k)^TJ}{y_k^Ty_k} + \dfrac{(s_k-H_ky_k)y_k^T}{y_k^Ty_k} - \dfrac{(Jy_k)^T(s_k-H_ky_k)Jy_ky_k^T}{(y_k^Ty_k)^2} \ .
	% \eeq
	% }
Algorithm \ref{alg0} describes the basic $J$-symmetric quasi-Newton method. We initialize with a solution $z_0$ and  an inverse Jacobian estimation $H_0$. For every iteration, we calculate the direction $s_k$, update the iterates, compute the difference in $F(z)$, and finally update the inverse Jacobian  estimation via \eqref{inv_up}. The algorithm is similar to any quasi-Newton method, and the key is the inverse Jacobian  update rule \eqref{inv_up}. We will present the local Q-superlinear convergence of Algorithm \ref{alg0} in the next section, and present the global R-superlinear convergence of a variant of Algorithm \ref{alg0} in Section \ref{global}.
\begin{algorithm}
	\caption{Unit-step $J$-symmetric Quasi-Newton Algorithm (J-symm)} 
	\begin{algorithmic}[1]
		\STATE Initialize with solution $z_0\in\R^{m+n}$ and inverse Jacobian  estimation $H_0\in\R^{(m+n)\times (m+n)}$
		\FOR {$k=1,2,3,\ldots,$}
		\STATE $s_k = -H_kF(z_k)$
		\STATE $z_{k+1} = z_k + s_k$
		\STATE $y_k = F(z_{k+1}) - F(z_k)$
		\STATE update $H_{k+1}$ via \eqref{inv_up}
		\ENDFOR
	\end{algorithmic}
	\label{alg0}
\end{algorithm}
%###########################################
Next, Algorithm \ref{algls} presents a simple 
%\deltextb{Armijo backtracking } 
%%\Asl{the formal definition of (Armijo) backtracking line search (and I checked both Boyd and Nocedal books) in convex optimization tries to satisfy$ f(x) -f(x+ts)  \ge  c_1t \nabla f(x)^Ts$  with backtracking $t$. so I think we should not call our line search a backtracking line search. A better description is a line search with backtracking approach}
line-search version of the above algorithm. More specifically, after computing the $J$-symmetric direction $s_k=-H_k F(z_k)$ as in Algorithm \ref{alg0}, we test how much improvement we can obtain by taking the step. If we see enough improvement, we take this step, and otherwise we halve the stepsize $t_k$, as one does in a backtracking line-search. Notice that we start with $t_k=1$ at each iteration, thus Algorithm \ref{algls} recovers Algorithm \ref{alg0} if we see sufficient improvements every iteration with $t_k=1$. Unfortunately, we do not have theoretical guarantees on this line search scheme, but numerical experiments in Section \ref{num_exp} showcases the benefits of the line-search scheme over other schemes.
\begin{algorithm}[H]
	\caption{$J$-symmetric  Quasi-Newton Algorithm with Line Search (J-symm-LS)} 
	\begin{algorithmic}[1]
		\STATE Initialize with solution $z_0\in\R^{m+n}$,  inverse Jacobian estimation $H_0\in\R^{(m+n)\times (m+n)}$ and linear-search parameter $c_1\in(0,1/2)$.
		\FOR {$k=1,2,3,\ldots,$}
		\STATE $s_k = -H_kF(z_k)$
		%\STATE set $t_k$ by passing $c_1$, $ z_k$ and $p_k$ to Algorithm \ref{alg1}
		\STATE $t_k=1$
		\WHILE {$\| F(z_k) \| - \| F(z_k+t_k s_k) \| < c_1 \| F(z_k) \|$} 
		\STATE $t_k = t_k/2$
		\ENDWHILE 
		\STATE $z_{k+1} = z_k + t_k s_k$
		\STATE $y_k = F(z_{k+1}) - F(z_k)$
		\STATE update $H_{k+1}$ via \eqref{inv_up}
		\ENDFOR
	\end{algorithmic}
	\label{algls}
\end{algorithm}

In the end of this section, we discuss the connections between our method and Powell symmetric Broyden (PSB) update, and comment on the instability of Broyden's update. 

The traditional minimization problem can be viewed as a special case of minimax problem \eqref{minimax} when the dual dimension is eliminated (namely $m=0$). In such a case, our quasi-Newton update  \eqref{Ebar3} recovers Powell symmetric Broyden update
\beq\tag{PSB}\label{PSB}
B_{k+1} = B_k + \dfrac{s_k(y_k-B_ks_k)^T}{s_k^Ts_k} + \dfrac{(y_k-B_ks_k)s_k^T}{s_k^Ts_k} - \dfrac{s_k^T(y_k-B_ks_k)s_ks_k^T}{(s_k^Ts_k)^2} \ . 
\eeq
Indeed, PSB update is known to be the unique minimizer of:
\begin{align*}
	\min_B ~ &  \dfrac{1}{2}\| B-B_k \|_F^2 \\ 
	\mathrm{s.t. }~ & Bs_k -y_k=0 \ , \\
	& B = B^T \ .
\end{align*}
Therefore, the $J$-symmetric update \eqref{Ebar3} is a direct generalization of PSB update \eqref{PSB}.

Next, let us look at Broyden's update for minimization problems. Broyden rank-1 update
\beq \tag{Broyden}
B_{k+1} = B_k + \dfrac{(y_k-B_ks_k)s_k^T}{s_k^Ts_k} \ ,
\eeq
also known as good Broyden update,
is the unique minimizer of the above minimization problem without the symmetry constraint
\begin{align*}
	\min_B ~ &  \dfrac{1}{2}\| B-B_k \|_F^2 \\ 
	\mathrm{s.t.~} & Bs_k -y_k=0 \ .
\end{align*}
The inverse Hessian estimation in Broyden's update can be written as %\eqref{broyden} 
%\beq \label{broy_good}
$$  H_{k+1} = H_k + \frac{(s_k-H_ky_k)s_k^TH_k}{s_k^TH_ky_k} \ . $$
%\eeq
{Notice that Broyden's update can be numerically unstable, because there is no guarantee that the denominator $s_k^TH_ky_k$ is far away from $0$. Moreover, if the Hessian at the optimal solution is not full-rank, a small perturbation of Hessian matrix would make iterates oscillate. 
Indeed, avoiding such numerical instability is a major task in the historical development of quasi-Newton methods. According to a survey by Dennis and Moré \cite{DM_survey}, the motivation which led to the derivation of PSB update and in fact later on to a whole new class of quasi-Newton methods using Powell's technique, was due to the fact that Symmetric Rank-1 (SR1) update has a similar numerical instability issue.  A similar issue could happen in BFGS formula, where the update is 
% Whereas in an update like BFGS %\SL{comments on BFGS update as well?}
\beq\tag{BFGS}
{B}_{k+1} =  {B_k} 
-\dfrac{ {B_k} {s_k} {s_k}^T {B_k}}{ {s_k}^T {B_k} {s_k}} + 	
\dfrac{ {y_k} {y_k}^T}{ {y_k}^T {s_k}} \ .
\eeq
The advantage of BFGS versus Broyden's method is that one can guarantee the denominator  $s_k^Ty_k >0$ by imposing the Wolfe condition \cite{AO20g}. In fact our experiment with Broyden method  shows that it can be unstable even when  applied to simple bilinear problems. Additionally, another major drawback of Broyden method is that unlike BFGS, it is not self-correcting. $B_k$ in Broyden method depends on each $B_j$ with $j\leq k$, and it might carry along irrelevant information for a long time~\cite{DM_survey}. Similar to BFGS, the J-symmetric update \eqref{Ebar3} always have a non-negative denominator in the update rule, which guarantees the stability of the update.}
% \beq\tag{SR1}
% B_{k+1} = B_k + \dfrac{(y_k-B_ks_k)(y_k-B_ks_k)^T}{(y_k-B_ks_k)^Ts_k}  \ ,
% \eeq
% is numerically unstable since $s_k$ could be orthogonal to $y_k-B_ks_k$. 
% A similar issue exists with the inverse Broyden update \eqref{broy_good} 
% and therefore, one needs to make sure $s_k^TH_ky_k$ is numerically nonzero before applying the update. 

%and hence nonsingular, which in turn is guaranteed.

\section{Local Q-Superlinear Convergence of the $J$-symmetric Update}\label{local}
In this section we present the local Q-superlinear convergence of Algorithm \ref{alg0}. In the local setting, we assume that the initial solution $z_0$ and the initial estimate of Jacobian $B_0$ is chosen from a close neighborhood of $z^*$ and $\nabla F(z^*)$,  respectively. Here, we assume
\begin{assumption}\label{assump_local} {(Assumptions for Local Superlinear Convergence)}
	\begin{enumerate}[(a)]
		\item There exists a minimax solution $z^*$ such that $F(z^*) =0$, and $\nabla F(z^*) $ is invertible {with $ \gamma =\|\nabla F^{-1}(z^*) \|$. }
		\item  There exist a nonzero open ball of radius $\epsilon$ centered at $z^*$, $B_{\epsilon}(z^*):=\{z|\|z-z^*\| < \epsilon\}$, such that for any $z \in B_{\epsilon}(z^*)$, it holds that:
		\beq\label{lip_zstar}
		\| \nabla F(z) - \nabla F(z^*) \| \leq \CC \| z-z^* \| \ .
		\eeq
	\end{enumerate}	
\end{assumption}
In the local convergence, we consider a ball $B_{\epsilon}(z^*)$ around a minimax solution $z^*$. Assumption \ref{assump_local} (a) assumes the non-singularity of $\nabla F(z^*)$, and (b) assumes Lipschitz continuity of $\nabla F(z^*)$ inside $B_{\epsilon}(z^*)$. These assumptions are quite weak and only require the Jacobian of the solution $\nabla F(z^*)$ to be invertible and Lipschitz continuous in a neighborhood. The local superlinear convergence of Algorithm \ref{alg0} is formalized in the next theorem:
\begin{theorem}\label{loc_lin_conv}
    Consider Algorithm \ref{alg0} for solving minimax problem \eqref{minimax}. Suppose there exists an optimal minimax solution $z^*$ that satisfies Assumption \ref{assump_local}. Then for any given $0<r<1$, there exist positive constants $\bar\epsilon$ and $\delta$ such that for any $z_0\in \{\| z_0 -z^* \| < \bar\epsilon\}$ and $B_0\in\{\|B_0 -\nabla F(z^*) \|_F < \delta\}$, it holds that: \\
	\textbf{(a)}. The sequence $\{ z_k  \}$ generated by Algorithm \ref{alg0}
	is well defined and converges to $z^*$, and $\{ \| B_k \| \}$ and  $\{ \| B_k^{-1} \| \}$ are uniformly bounded for any $k \geq 0$. Additionally,
	%\textbf{(b)}.
	\beq\label{lin_conv}
	\| z_{k+1} -z^* \|	 \leq  r\| z_k -z^* \| \ .
	\eeq 
	 \\
	\textbf{(b)}. The iterates $\{z_k\}$ enjoy Q-superlinear convergence towards $z^*$. 
\end{theorem}
Our local analysis in Theorem \ref{loc_lin_conv} is based on the \textit{bounded deterioration} technique and is similar to the analysis given in \cite{localB}. To establish Theorem \ref{loc_lin_conv}, we first present two lemmas, which are used in the proof.

\begin{lemma}\label{my_lemma}
	Suppose Assumption \ref{assump_local} holds. Then, it holds for any small enough $\epsilon>0$ and $u,v \in B_{\epsilon}(z^*)$: 
	\begin{enumerate}[(a)]
		\item 
		\beq\label{y_Az}
		\|F(v) - F(u) - \nabla F(z^*)(v-u)\| \leq \CC \max \{ \| v -z^* \|, \| u -z^* \| \}\| v-u  \| \ .
		\eeq
		\item 
		There exists $\rho>0$ such that
		\beq\label{rho}
		\dfrac{\| v-u  \|}{\rho} \leq \|F(v) - F(u)\| \leq \rho\| v-u  \| \ .
		\eeq
	\end{enumerate}
\end{lemma}
\begin{proof}
	(a). Denote $T(z) = F(z) -  \nabla F(z^*)z$, then $T(z)$ is  differentiable by noticing $F(z)$ is  differentiable and $\nabla T(z) = \nabla F(z) -  \nabla F(z^*)$. By Taylor expansion at $u$, we obtain
	$$T(v) = T(u)  +\int_0^1 \nabla T\Big(u + t(v-u)\Big) (v-u)dt \ , $$
	thus, 
	$$\|T(v) - T(u) \| \leq \sup_{0\leq t \leq 1} \| \nabla T\big(u + t(v-u)\big) \|\|v-u\| \ .$$
	Substituting $T(z)$ to the above inequality, we obtain
	\begin{align*}
	\|F(v) - F(u) - \nabla F(z^*)(v-u) \| &\leq \sup_{0\leq t \leq 1} \| \nabla F\big(u + t(v-u)\big) - \nabla F(z^*) \|\|v-u\|\\
	&\leq \sup_{0\leq t \leq 1}  \CC  \| u + t(v-u) -z^* \|\|v-u\|\\
	&= \CC\max \{ \| v -z^* \|, \| u -z^* \| \}\|v-u\| \ ,
	\end{align*} 
	where the second inequality uses \eqref{lip_zstar}. 
	
	(b). 
	It follows from \eqref{y_Az} by triangle inequality that
	\begin{align*}
	\|F(v) - F(u) \| &\leq \CC \max \{ \| v -z^* \|, \| u -z^* \| \}\| v-u  \| + \|\nabla F(z^*)(v-u)\| \leq  \big(\CC\epsilon + \| \nabla F(z^*) \|\big)\|v-u\| \ .
	\end{align*}

	Furthermore, let $\sigma$ be the smallest singular value of $\nabla F(z^*)$, then $\sigma>0$ as $\nabla F(z^*)$ is full rank, whereby it holds for any $u, v$ that	
	$$ \sigma\|v-u\| \leq \|\nabla F(z^*)(v-u) \| \ .$$ Therefore, it follows from \eqref{y_Az} that
	\begin{align*}
	\|F(v) - F(u) \| &\ge \|\nabla F(z^*)(v-u)\|  -\CC \max \{ \| v -z^* \|, \| u -z^* \| \}\| v-u  \| \ge (\sigma-\CC\epsilon )\| v-u  \| \ .
% 	&\ge \sigma\|v-u\| -\CC \max \{ \| v -z^* \|, \| u -z^* \| \}\| v-u  \| \\ 
% 	&\ge (\sigma-\CC\epsilon )\| v-u  \| .
	\end{align*}
	Now suppose $\epsilon<\sigma/\CC$ and setting  
	$~\rho = \max \left\{ 1/(\sigma-\CC\epsilon)~,~  \CC\epsilon + \| \nabla F(z^*) \|  \right\}$, we arrive at \eqref{rho}.
\end{proof}
%%%%%%%%%%%%%%

The following lemma presents an equivalent representation of \eqref{Ebar3} that we will use later.
%%%%%%%%%%%%%%%%%%%%%%%%%%%%%%%%%%%%%%%%%%%%%%%%%%%%%%%%%%%%%%%%%%%%%
%%%%%%%%%%%%%%%%%%%%%%%%%%%%%%%%%%%%%%%%%%%%%%%%%%%%%%%%%%%%%%%%%%%%%
\begin{lemma}\label{local_Mbar}
Consider the $B_k$ update rule \eqref{Ebar3}. Then it holds that
\begin{align}\label{Bk1}
B_{k+1} &= JP_kJB_kP_k 
+ \dfrac{y_ks_k^T}{s_k^Ts_k} + \dfrac{Js_ky_k^TJ}{s_k^Ts_k}P_k \ ,
\end{align}
where $P_k$ is the projection matrix defined in \eqref{p2}. Furthermore, it holds that
\begin{align}\label{Mbarnorm}
\| B_{k+1}-\nabla F(z^*) \|_F \leq  \sqrt{(1-\theta_{1,k}^2)(1-\theta_{2,k}^2)}\| B_{k}-\nabla F(z^*) \|_F + (1+\sqrt{n+m-1})\dfrac{\|y_k-\nabla F(z^*) s_k\|}{\|s_k\|} \ ,
\end{align}
where
\beq \label{theta2}
\theta_{1,k} = \dfrac{\| JP_kJ(B_{k}-\nabla F(z^*))s_k  \| }{\|s_k\| \| JP_kJ (B_{k}-\nabla F(z^*)) \|_F } \ ,	~~~\mathrm{and}~~~~~~~
\theta_{2,k} = \dfrac{\| (B_{k}-\nabla F(z^*))^TJs_k \| }{\|s_k\| \| B_{k}-\nabla F(z^*) \|_F } \ .
\eeq
\end{lemma}
\begin{proof}
First note that $P_k$ is the projection matrix onto the $m+n-1$ dimension subspace which is perpendicular to $s_k$, thus
\beq\label{normp2}
\| P_k\| =1 \ , ~~~~
\| P_k\|_F = \sqrt{m+n-1} \ .
\eeq

Let $O$ be any $J$-symmetric matrix with proper size and let $M = B_k -O$ and  $\bar M = B_{k+1} -O$, then we claim that the following holds:
\begin{align}\label{Mbar}
\bar M &= JP_kJMP_k
+ \dfrac{(y_k- Os_k)s_k^T}{s_k^Ts_k} + \dfrac{Js_k(y_k-Os_k)^TJ}{s_k^Ts_k}P_k \ .
\end{align}
This is because from \eqref{Ebar3} we have
\small
\begin{align*}
B_{k+1} -O = ~&B_k-O + \dfrac{Js_k(y_k-B_ks_k +Os_k-Os_k)^TJ}{s_k^Ts_k} + \dfrac{(y_k-B_ks_k+Os_k-Os_k)s_k^T}{s_k^Ts_k} -\\
& \dfrac{(Js_k)^T(y_k-B_ks_k+Os_k-Os_k)Js_ks_k^T}{(s_k^Ts_k)^2}  \ .
\end{align*}
\normalsize
Substituting $B_{k+1} -O = \bar M$ and $B_k-O = M$ we obtain
\small
\begin{align*}
\bar M =~ & M -\dfrac{Js_ks_k^TM^TJ}{s_k^Ts_k} + \dfrac{Js_k(y_k-Os_k)^TJ}{s_k^Ts_k} 
- \dfrac{Ms_ks_k^T}{s_k^Ts_k} + \dfrac{(y_k- Os_k)s_k^T}{s_k^Ts_k} + \dfrac{(Js_k)^TMs_kJs_ks_k^T}{(s_k^Ts_k)^2} -\dfrac{(Js_k)^T(y_k-Os_k)Js_ks_k^T}{(s_k^Ts_k)^2}\\
%%%%%%%%
= ~& M - \dfrac{Js_ks_k^TM^TJ}{s_k^Ts_k}  - \dfrac{Ms_ks_k^T}{s_k^Ts_k} + \dfrac{(Js_k)^TMs_kJs_ks_k^T}{(s_k^Ts_k)^2} +  \dfrac{Js_k(y_k-Os_k)^TJ}{s_k^Ts_k}  + \dfrac{(y_k- Os_k)s_k^T}{s_k^Ts_k}  - \dfrac{(Js_k)^T(y_k-Os_k)Js_ks_k^T}{(s_k^Ts_k)^2}\\
%%%%%%%%
= ~& M - \dfrac{Js_ks_k^TJM}{s_k^Ts_k}  - \dfrac{Ms_ks_k^T}{s_k^Ts_k} + \dfrac{Js_k(Js_k)^TMs_ks_k^T}{(s_k^Ts_k)^2}+  \dfrac{Js_k(y_k-Os_k)^TJ}{s_k^Ts_k}  + \dfrac{(y_k- Os_k)s_k^T}{s_k^Ts_k}  - \dfrac{(Js_k)^T(y_k-Os_k)Js_ks_k^T}{(s_k^Ts_k)^2} \\
%%%%%%%%
=~ & JP_kJMP_k + \dfrac{(y_k- Os_k)s_k^T}{s_k^Ts_k} + \dfrac{Js_k(y_k-Os_k)^TJ}{s_k^Ts_k}P_k\ , 
\end{align*}
\normalsize
where the second equality comes from rearrangement. The third equality uses the fact that $M$ is $J$-symmetric, thus we have $M^TJ = JM$, and the fact that $(Js_k)^TMs_k$ is a scalar, thus we have $(Js_k)^TMs_kJs_ks_k^T = Js_k(Js_k)^TMs_ks_k^T$.
The last equality uses $J^2=I$ and \eqref{p2} thus
$ JP_kJ = I - J(s_ks_k^T)J/(s_k^Ts_k)$,
and therefore the sum of the first four terms on the third line is exactly $JP_kJMP_k$.
Additionally, in the final term on the same line, $(Js_k)^T(y_k-Os_k)$ is a scalar, so we can use
$(Js_k)^T(y_k-Os_k)Js_ks_k^T = Js_k(y_k-Os_k)^TJs_ks_k^T.$
By factoring out $\big(Js_k(y_k-Os_k)^TJ\big)/(s_k^Ts_k)$ from this term and the fifth term and recalling \eqref{p2}  we arrive at \eqref{Mbar}. 

Using \eqref{Mbar} and setting $O$ equal to zero, we obtain \eqref{Bk1} and therefore conclude that the update rule \eqref{Ebar3} is equivalent to \eqref{Bk1}. 

	%%%
	%proof of the second part
	%%%	
	To show \eqref{Mbarnorm},  we start by bounding the first term in \eqref{Mbar} as following: %\eqref{Mbar00}:
	\small
	\begin{align*}\label{Mn1} 
		\|JP_kJMP_k\|_F^2 &= \Big\|  JP_kJM\big( I - \dfrac{s_ks_k^T}{s_k^Ts_k}\big)\Big\|_F^2\\
		&= \| JP_kJM \|_F^2 -2\dfrac{\| JP_kJMs_k  \|^2}{s_k^Ts_k} + \dfrac{\| JP_kJMs_k  \|^2}{s_k^Ts_k}\\
		&= \Big(1 - \dfrac{\| JP_kJMs_k  \|^2}{\|s_k\|^2\| JP_kJM \|_F^2}\Big) \| JP_kJM \|_F^2 \ ,
	\end{align*}
	\normalsize
	%By Cauchy-Schwartz 
	where the first equality uses \eqref{p2} and 
	the second equality follows directly from the 
	definition of the Frobenius norm. 
	%(see also \cite[Lemma 4]{Broyden_single} in the appendix). 
	%application of \eqref{fer_exp}. 
	Furthermore, 
	%we know from \eqref{Fl2} that $\theta_{1,k}, \theta_{2,k}\in [0,1]$.Thus, 
	\small
	\begin{align*}
		\| JP_kJM \|_F^2 =\|M^TJP_kJ\|_F^2 &= \Big\|  M^T\big( I - \dfrac{Js_k(Js_k)^T}{s_k^Ts_k}\big)\Big\|_F^2\\
		&= \| M^T \|_F^2 -2\dfrac{(Js_k)^TMM^TJs_k}{s_k^Ts_k} + \dfrac{\| M^TJs_k  \|^2}{s_k^Ts_k}\\
		&=\| M \|_F^2 - \dfrac{\| M^TJs_k  \|^2}{\| s_k \|^2}\\
		&= \Big(1 - \dfrac{\| M^TJs_k \|^2}{\|s_k\|^2\| M \|_F^2}\Big) \| M \|_F^2 \ ,
	\end{align*}
	\normalsize
	where the second equality uses  $ JP_kJ = I - J(s_ks_k^T)J/(s_k^Ts_k)$ and 
	the third equality follows directly from the 
	definition of the Frobenius norm.  
	%(see also \cite[Lemma 4]{Broyden_single} in the appendix).
	%uses \eqref{fer_exp}, 
	%and the last line we have used $\|s_k\|=\|Js_k \| $.  
	%Cauchy-Schwartz inequality 
	For the remainder of this proof we set $O = \nabla F(z^*)$, so we obtain
	$$	\theta_{1,k} = \dfrac{\| JP_kJMs_k  \| }{\|s_k\| \| JP_kJ M \|_F } \  \mathrm{~~and,~~~~~~~~~~~}
	\theta_{2,k} = \dfrac{\| M^TJs_k \| }{\|s_k\| \| M \|_F }  \ .$$
	%From the definition of the induced $l_2$  operator norm
	By Cauchy-Schwarz inequality we know  $ \| JP_kJMs_k  \| /(\|s_k\|\| JP_kJ M \|) \leq 1$ and since induced $l_2$  norm is less than Frobenius norm, we conclude: $0 < \theta_{1,k}  \leq 1$. Similarly, we obtain $ \| M^TJs_k \|/(\|Js_k\|\| M^T \|_F)  \leq  1$ and since$\|Js_k \|=\|s_k\|$ and $\|M^T \|_F=\|M \|_F$, we conclude $0 < \theta_{2,k}  \leq 1$.  Hence, we can safely take square root from both sides and arrive at:
	\beq\label{t1}
	\|JP_kJMP_k\|_F = \sqrt{(1-\theta^2_{1,k})(1-\theta^2_{2,k})}\| M \|_F \ .
	\eeq
	We obtain the following equality for the norm of the second term (recall $O = \nabla F(z^*)$) in \eqref{Mbar}:
	\small
	\beq\label{t2}
	\left\| \dfrac{\big(y_k- \nabla F(z^*)s_k\big)s_k^T}{s_k^Ts_k} \right\|_F   =   \dfrac{\sqrt{\mathrm{Tr}\Big( \big(y_k- \nabla F(z^*)s_k\big)s_k^T  s_k\big(y_k- \nabla F(z^*)s_k\big)^T  \Big)}}{\|s_k\|^2} = \dfrac{\|y_k-\nabla F(z^*)s_k\|}{\|s_k\|} \ .
	\eeq
	\normalsize
	Finally from the application of the inequality $\| AB \|_F \le \| A \|\| B \|_F$ (see Lemma \ref{F2ineq} in the appendix),
	%\eqref{Fl2}  
	to the third term in \eqref{Mbar} we obtain:
	\small
	\begin{align}\label{t3}
		\Big\| \dfrac{Js_k\big(y_k-\nabla F(z^*)s_k\big)^TJ}{s_k^Ts_k}P_k \Big\|_F &\leq  \|P_k\|_F \dfrac{\|Js_k\big(y_k-\nabla F(z^*)s_k\big)^TJ\|}{\|s_k\|^2} &    \nonumber \\
		&\leq \|P_k\|_F\dfrac{\|Js_k\|\|\big(y_k-\nabla F(z^*)s_k\big)^TJ\|}{\|s_k\|^2} = 
		\sqrt{n+m-1}\dfrac{\|y_k-\nabla F(z^*)s_k\|}{\|s_k\|} \ ,
	\end{align}
	\normalsize
	where we use the fact that $\|Jq\| = \|q\|$ for any vector $q$ of the appropriate size and \eqref{normp2} in the final equality.
	Combining \eqref{t1},\eqref{t2} and \eqref{t3}, and then substituting $\bar M = B_{k+1} - \nabla F(z^*) $ and $M = B_k - \nabla F(z^*)$, we obtain \eqref{Mbarnorm}.
\end{proof}

%%%%%%%%%%%%%%%%%%%%%%%%%%%%%%%%%%%%%
%%%%%%%%%%%%%%%%%%%%%%%%%%%%%%%%%%%%%
\begin{proposition}\label{bd_det_lem} 
	Suppose Assumption \ref{assump_local} holds. Recall that $ \gamma =\|\nabla F^{-1}(z^*) \|$. 
	For any given $z_k \in B_{\epsilon}(z^*)$ and  invertible $J$-symmetric matrix $B_k$ such that 
	$\| B_k^{-1}  \|< 2\gamma$,
	we have $z_{k+1} \in B_{\epsilon}(z^*)$, where $z_{k+1}$ is obtained from \eqref{seq}.  %$z_{k+1} = z_k -B_k^{-1}F(z_k)$.
	Moreover, if  $B_{k+1}$  is obtained from 
	\eqref{Ebar3}, 
	%\eqref{Bk1}
	we have 
	\begin{align}\label{bd_detrio_strong}
		\| B_{k+1} -\nabla F(z^*)  \|_F ~ \leq & \sqrt{(1-\theta_{1,k}^2)(1-\theta_{2,k}^2)} \| B_k - \nabla F(z^*)\|_F  \\
		& +\CC(1+\sqrt{n+m-1})\max \{ \| z_{k+1} -z^* \|, \| z_k -z^* \| \} \nonumber \ .  
	\end{align} 
\end{proposition}
\begin{proof}  
	%Applying \eqref{Fl2} to 
	Starting from \eqref{seq} we have
	$$\|z_{k+1}-z_k  \| = \| B_k^{-1}F(z_k) \| \leq \| B_k^{-1}\| \|F(z_k) \| \leq 2 \gamma \|F(z_k) \| \ . $$ 
	Since $F(z^*) = 0$ it follows from \eqref{rho}  that $ \|F(z_k) \| \leq \rho \|z_k-z^*\| ,$ so, $ \|z_{k+1}-z_k\|
	\leq 2 \rho\gamma \|z_k-z^*\| .$
	By further restricting $z_k$ such that  
	\[\| z_k -z^* \| < \min \Big\{ \epsilon/2~,~ \frac{\epsilon/2}{2 \rho \gamma} \Big\}\ ,\]
	we obtain $\| z_{k+1} - z^* \| \leq  \|z_{k+1}-z_k \| + \|z_k-z^*\| < \epsilon, $
	and therefore it holds that $z_{k+1} \in B_{\epsilon}(z^*) .$
	Applying \eqref{y_Az}, we obtain
	\begin{align*}
		\|F(z_{k+1}) - F(z_k) - \nabla F(z^*)(z_{k+1}-z_k) \| &\leq \CC\max \{ \| z_{k+1} -z^* \|, \| z_k -z^* \|\} \| z_{k+1}-z_k  \|\\
		\frac{\|y_k - \nabla F(z^*)s_k \|}{ \| s_k  \|} &\leq \CC\max \{ \| z_{k+1} -z^* \|, \| z_k -z^* \|\} \ .
	\end{align*}
	%where the second inequality comes from the definition of $y_k$ and $s_k$. 
	We arrive at \eqref{bd_detrio_strong} by substituting the above inequality into \eqref{Mbarnorm}.% with  $O = \nabla F(z^*)$.
\end{proof}
%%%%%%%%%%%%%%%%%%%%%%%%%%%%%%%%%%%%%%%%%%%%%%%%%%%%%%%%%
%.    Proof of Thm 3.2 (a)
%%%%%%%%%%%%%%%%%%%%%%%%%%%%%%%%%%%%%%%%%%%%%%%%%%%%%%%%%
Now we are ready to prove Theorem \ref{loc_lin_conv}:
\begin{proof}[Proof of Theorem \ref{loc_lin_conv}]
	Set
	\beq\label{delta}
	\delta = \dfrac{r}{\gamma(r+1)\Big( \frac{1-r}{1+\sqrt{m+n-1}} +2\Big)} \ ,
	\eeq
	and
	\beq\label{epsilon}
	\bar \epsilon = \min\Big\{\dfrac{(1-r)\delta}{\CC(1+\sqrt{m+n-1})}~,~ \epsilon\Big\} \ ,
	\eeq
	then it holds that:
	\beq\label{req}
	%%to see this easily substitute $b = \frac{(1-r)}{(1+\sqrt{m+n-1})}$ into $\delta$ and $\bar \epsilon$.
	\gamma(r+1)\big( \CC\bar\epsilon + 2\delta \big) \leq r \ .
	\eeq
	We prove part \textbf{(a)} by induction. We begin with $k=0$. \\
	From $\|B_0 -\nabla F(z^*) \|_F < \delta$  we know $\|B_0 -\nabla F(z^*) \| < \delta < 2\delta$, and recall $ \|\nabla F^{-1}(z^*) \| =\gamma $. Notice that  \eqref{req} implies  $\gamma2\delta < r/(r+1)  < 1$.  So,
	we can apply Banach Perturbation Lemma (see Lemma \ref{banach} in the appendix) to the matrices $\nabla F(z^*)$ and $B_0$, and obtain
	\beq\label{unbd0}
	\| B_0^{-1} \| \leq \frac{\gamma}{1-r/(1+r)} = \gamma(r+1) \ .
	\eeq
	To prove \eqref{lin_conv}, recall that $F(z^*) =0,$ and since  $\|z_0-z^*  \| < \bar \epsilon $, Lemma \ref{my_lemma} applies.  From \eqref{seq} we have
	\small
	\begin{align*}
		\| z_1 - z^* \| &= \|   B_0^{-1}F(z_0) -(z_0-z^* )\|  \\
		&=   \| B_0^{-1}F(z_0) - B_0^{-1}\nabla F(z^*)(z_0-z^*) + B_0^{-1}\nabla F(z^*)(z_0-z^*) -(z_0-z^*) \| \\
		&\leq   \| B_0^{-1} \|\Big(\|F(z_0) -F(z^*)- \nabla F(z^*)(z_0-z^*) \| +  \|\nabla F(z^*) -B_0 \| \|z_0-z^*  \|\Big) \\
		& \leq \gamma(r+1)( \CC\bar\epsilon + 2\delta ) \|z_0-z^*  \| \ ,
	\end{align*}
	\normalsize
	where in the final inequality we use \eqref{unbd0} and \eqref{y_Az}. By applying \eqref{req} to this inequality we obtain
	\beq\label{lin_conv1}
	\| z_1 - z^* \| \leq r \|z_0-z^* \| \ .
	\eeq
	This implies $\| z_1 - z^* \| <  \bar \epsilon \leq \epsilon$ and hence $z_1 \in B_{\epsilon}(z^*)$.
	Now we prove the claims for $k=K$, assuming \eqref{unbd0} and \eqref{lin_conv1} 
	hold for $k=0,\hdots, K-1.$ Notice that \eqref{unbd0} implies $\| B_{K-1}^{-1} \| \leq 2\gamma$ and hence we can apply Proposition \ref{bd_det_lem} and by \eqref{bd_detrio_strong}  
	together with $\| z_K - z^* \| \leq r \|z_{K-1}-z^* \|$ conclude that
	\[
	\| B_K -\nabla F(z^*)  \|_F  ~\leq ~\| B_{K-1} - \nabla F(z^*)\|_F  + \CC\big(1+\sqrt{n+m-1}\big)\|z_{K-1}-z^* \| \ .
	\]
	Summing up from $k=0$ to $k=K-1$ we have
	\begin{align}\label{bddet}
		\| B_K -\nabla F(z^*)  \|_F ~\leq \| B_0 - \nabla F(z^*)\|_F  + \CC\big(1+\sqrt{n+m-1}\big)\bar\epsilon\dfrac{1-r^K}{1-r} \ .
	\end{align} 
	From \eqref{epsilon} we have $\CC\big(1+\sqrt{n+m-1}\big)\bar\epsilon/(1-r) \leq \delta$ and recalling $\|B_0 -\nabla F(z^*) \|_F < \delta$, we conclude: 
	\[
	\| B_K -\nabla F(z^*)  \|_F ~ < 2\delta \ .
	\]
	Using this inequality and $ \gamma =\|\nabla F^{-1}(z^*) \|$, via Banach Perturbation Lemma and with the same exact  proof as we did for $k=0$, we conclude
	\beq\label{unbd}
	\| B_K^{-1} \| \leq \gamma(r+1) \ .
	\eeq
	Now let us prove \eqref{lin_conv} for $k=K$. Notice
	\small 
	\begin{align*}
		\| z_{K+1} - z^* \| &= \|   B_K^{-1}F(z_K) -(z_K-z^* )\|  \\
		&\leq   \| B_K^{-1} \|\Big(\|F(z_K) -F(z^*)- \nabla F(z^*)(z_K-z^*) \| +  \|\nabla F(z^*) -B_K \| \|z_K-z^*  \|\Big) \\
		& \leq \gamma(r+1) \Big( \CC\bar\epsilon + 2\delta \Big) \|z_K-z^*  \| \ .
	\end{align*}
	\normalsize
	Thus, we get $\| z_{K+1} - z^* \| \leq r \|z_K-z^* \| ,$ which completes the proof of part \textbf{(a)} by induction.
	
%%%%%%%%%%%%%%%%%%%%%%%%%%%%%%%%%%%%%%%%%%%%%%%%%%%%%%%%%
%.    Proof of Thm 3.2 (b)
%%%%%%%%%%%%%%%%%%%%%%%%%%%%%%%%%%%%%%%%%%%%%%%%%%%%%%%%%
 Next we move to part \textbf{(b)} to show \eqref{seq} is Q-superlinearly convergent.
	As a result of part \textbf{(a)}, Proposition \ref{bd_det_lem} applies for all $k$. 
    In \eqref{bd_detrio_strong} define $\bar \theta_k =  (\theta_{1,k}^2+\theta_{2,k}^2)/2 $ and since $
	\sqrt{(1-\theta_{1,k}^2)(1-\theta_{2,k}^2)} \leq 1-\bar \theta_k  \  $, together with \eqref{lin_conv}, we deduce% \eqref{bd_detrio_strong} as 
	\begin{align*}
	\bar \theta_k \| B_k - \nabla F(z^*)\|_F \leq \| B_k - \nabla F(z^*)\|_F - 	\| B_{k+1} -\nabla F(z^*)  \|_F   +\CC(1+\sqrt{n+m-1})\| z_k -z^* \| \  . 
	\end{align*} %\max \{ \| z_{k+1} -z^* \|, \| z_k -z^* \| \}
	Summing up for $k=0,\hdots, \infty$ we obtain $\sum_{k=0}^{\infty}\bar \theta_k \| B_k - \nabla F(z^*)\|_F $ in the L.H.S. and since we know that the R.H.S. is bounded above (see \eqref{bddet}) {we conclude	
	$$
	\lim_{k\to \infty} \bar \theta_k \| B_k - \nabla F(z^*)\|_F = \frac{1}{2}\lim_{k\to \infty}  (\theta_{1,k}^2+\theta_{2,k}^2) \| B_k - \nabla F(z^*)\|_F  =0  \  .
	$$
	%Hence we must have  $ \lim_{k\to \infty} \bar \theta_k = 0 $.
	Since both $\theta_{1,k}$ and $\theta_{2,k} $ are positive, we conclude:
    $\lim_{k\to \infty} \theta_{1,k}^2 \| B_k - \nabla F(z^*)\|_F =0$ and $\lim_{k\to \infty} \theta_{2,k}^2 \| B_k - \nabla F(z^*)\|_F  =0$. 
	Substituting  $\theta_{2,k}$ from \eqref{theta2} followed by replacing $\left(B_k - \nabla F(z^*)\right)^TJ = J\left(B_k - \nabla F(z^*)\right)$ 
	%\eqref{BJBT} 
	we obtain
% 	we have
% 		\[\lim_{k\to \infty} \dfrac{\| M^TJs_k \|}{\|s_k\|\| M \|_F} = 0 \ , \]
% 	where $M = B_k - \nabla F(z^*)$, 
%	 	Due to  $M^TJ = JM$ we get recalling \eqref{normJs} we get
	\[
	\lim_{k\to \infty} \dfrac{\Big\| J\big(B_k - \nabla F(z^*)\big)s_k \Big\|^2}{\|s_k\|^2\| B_k-\nabla F(z^*) \|_F} = 0 \ .
	\]
    From Theorem \ref{loc_lin_conv} (a), $\| B_k\|$ are uniformly bounded and thus there exists a constant $c>0$ such that
    \[
        \frac{1}{\| B_k-\nabla F(z^*) \|_F}>\frac{1}{c} \ .
    \]
    Hence it holds that 
    \[
        \lim_{k\to \infty} \dfrac{\Big\| J\big(B_k - \nabla F(z^*)\big)s_k \Big\|}{\|s_k\|} = 0 \ .
    \]
	Recalling that $\|Jq\| = \|q\|$ for any vector $q \in \R^{n+m}$, we conclude
	$$
	\lim_{k\to \infty} \dfrac{\Big\| \big(B_k - \nabla F(z^*)\big)s_k \Big\|}{\|s_k\|} = 0 \ .
	$$
 }
	This is Dennis-Moré characterization  identity for Q-superlinear convergence (see Theorem \ref{Qcharac} in the appendix) and therefore the proof is finished. % \eqref{sup_char}
\end{proof}

% %%Uncomment to see the idea of the proof for the line-search J-symmetric alg.
% \begin{comment}
% \Asl{New addition Dec 10th, 2021:}\\
% \teal{
% Although Theorem \ref{loc_lin_conv} is stated for Algorithm \ref{alg0}, which is a unit-step quasi-Newton method (i.e. stepsize $t_k=1$ for $ k=0,1,\hdots$  ), it is possible to give a proof of this theorem for a line-search version of this algorithm. See  Algorithm \ref{alg2} in section  \S \ref{step}. Assuming that $t_k \leq 1$ is sufficiently large so that  in addition to  \eqref{req},
% $$
% 	t_{min}\gamma(r+1)( \CC\bar\epsilon + 2\delta ) + (1-t_{min} )\leq r
% $$
% with $t_{min} = \min_{0\leq k} ~ t_k$, also holds true, the  proof of \eqref{lin_conv}  for $k=0$  in part \textbf{(a)} could be modified in the 
% following way:
% \small
% \begin{align*}
% 	\| z_1 - z^* \| & \leq t_0\|   B_0^{-1}F(z_0) -(z_0-z^* )\| + (1-t_0)\|z_0-z^*  \| \\
% 	&=   t_0\| B_0^{-1}F(z_0) - B_0^{-1}\nabla F(z^*)(z_0-z^*) + B_0^{-1}\nabla F(z^*)(z_0-z^*) -(z_0-z^*) \| + (1-t_0)\|z_0-z^*  \| \\
% 	&\leq   t_0\| B_0^{-1} \|\Big(\|F(z_0) -F(z^*)- \nabla F(z^*)(z_0-z^*) \| +  \|\nabla F(z^*) -B_0 \| \|z_0-z^*  \|\Big) + (1-t_0)\|z_0-z^*  \| \\
% 	&\leq  \Big(t_0\gamma(r+1)( \CC\bar\epsilon + 2\delta ) + (1-t_0) \Big) \|z_0-z^*  \| \ .
% \end{align*}
% \normalsize
% Since $t_{min}\leq t_0$ we again conclude  $ \| z_1 - z^* \| \leq r \|z_0-z^* \| \ . $ The poof of \eqref{lin_conv}  for $k=K$ changes similarly and
% the rest of the proof of  Theorem \ref{loc_lin_conv} remains unchanged.
% }
% \end{comment}

\section{A Globally Convergent $J$-symmetric quasi-Newton Method}\label{global}
	The previous section establishes the local superlinear convergence of the $J$-Symmetric quasi-Newton method. In this section, we present a trust-region $J$-symmetric quasi-Newton method (Algorithm \ref{tr_alg}), and show  its global superlinear convergence guarantees. 
	To present our algorithm, we first introduce a merit function minimization problem:
	\beq\label{obj_fun}
	\min_{z \in \R^{n+m}} \frac{1}{2} \| F(z) \|^2  \ .
	\eeq
	{Note that \eqref{obj_fun} is generally non-convex, directly apply conventional quasi-Newton on this minimization problem can get stuck at local minimas. Despite the non-convexity,}
    it is straight-forward to see that the global minimizers to \eqref{obj_fun} are exactly the same as the saddle points to \eqref{minimax}. Furthermore, we define $m_k(s)$ as the quadratic model of the merit function at $z_k$:
	\beq\label{mk}
	m_k(s) := \frac{1}{2}\| F(z_k) \|^2+ g_k^T s  + \frac{1}{2}s^TB_k^TB_ks \ , %~~~~ \mathrm{such~ that~} \| s \| \leq \Delta_k,
	\eeq
	where 
	$$g_k=\nabla F(z_k)^T F(z_k) \ ,$$ 
	is the gradient of the merit function and $B_k$ is the estimation of the {Jacobian $\nabla F(z_k)$} (see the below update rule \eqref{Ebar3_beta}). Then, $m_k(s)$ is an approximated second order expansion of the merit function around $z_k$. Here we would like to highlight that (i) while the calculation of $g_k$ involves $\nabla F(z_k)$, it can be performed efficiently by using fast Hessian-vector product for many applications~\cite{pearlmutter1994fast,schraudolph2002fast}; (ii) similar to other quasi-Newton methods, we can store $B_k^{-1}$ in memory and update $B_{k+1}^{-1}$ by a low-rank operation. As a result, calculating the minimizer of the quadratic model \eqref{mk} only involves matrix-vector multiplication, in contrast to Newton's method which involves solving linear equations. In other words, it has the same order of cost-per-iteration as a first-order method in general.

	Algorithm \ref{tr_alg} presents our Trust-region $J$-Symmetric Algorithm. We initialize with solution $z_0$, Jacobian estimation $B_0$,  trust-region radius upper bound $R_0$, initial trust-region radius $\Delta_0 \in (0,~R_0]$, and valid step (sufficient decrease) parameter 
	$\zeta \in (0, 10^{-3})$.
	%tolerance  $\varepsilon$ and maximum number of iterations $\bar m$.  
	In the $k$-th iteration of the algorithm, there are three potential valid steps (i) the quasi-Newton step $p^B_k$, (ii) the Cauchy point step $p_k^C$, and (iii) the dogleg step $p_k^D$, as defined below:
	
	%%%%%%%%%%%%%%%%%%%%%%%%%%%%%%%%%%%  
	% Here we summarize our strategy for choosing the direction $s_k$:
	
	\textbf{(Quasi-Newton step $p^B_k$)}. 
	The quasi-Newton point $p^B_k$ is defined as the global minimizer of $m_k(s)$, namely
	\beq\label{eq:pBk}
	p^B_k = -B_k^{-1}(B_k^{-1})^Tg_k \ .
	\eeq 
	\textbf{(Cauchy point step $p^C_k$)}. The Cauchy point is defined as the minimizer of $m_k(s)$ over the trust-region along  the negative gradient direction:
	$$
	p^C_k := -\tau_k(\Delta_k/\|g_k\|) g_k\ ,
	$$
	where $\tau_k := \arg\min_{ 0\le\tau\le 1}~m_k \Big(\tau \Delta_k g_k/ \|g_k\|\Big) \ $.
	The Cauchy point 
	has the following closed-form solution 
	\beq\label{eq:pCk}
	p^C_k= -\min\Big\{\|g_k\|^2/({g_k}^TB_k^TB_kg_k)~,~\Delta_k/\|g_k\| \Big\}g_k \ .
	\eeq
	\textbf{(Dogleg step $p^D_k$).}  %$\|p_k^B\|>\Delta_k$, and the 
	When the Cauchy point is strictly inside the trust-region ($\|p^C_k\|<\Delta_k$) and the Quasi-Newton step $p_k^B$ is strictly outside the trust-region ($\|p^B_k\|>\Delta_k$), that is when $p^C_k= - \Big(\|g_k\|^2/({g_k}^TB_k^TB_kg_k)\Big)g_k$,
	we then define the dogleg point as
	\beq\label{dogleg}
	p^D_k = p^C_k + \alpha (p^B_k -p^C_k) \ , 
	\eeq
	where $\alpha \in (0,1)$ is the unique solution that satisfies $\|p^C_k + \alpha (p^B_k -p^C_k)\| = \Delta_k \ $. 
	%%%%%%%%%%%%%%%%%%%%%%%%%%%%%%%%%%% 
		
	To update the iterate solution, we first calculate the quasi-Newton step $p_k^B$. If $p_k^B$ is inside the trust-region, we take the quasi-Newton step. Otherwise, we calculate the Cauchy point step $p_k^C$. If the Cauchy point step is on the boundary of the trust-region, we take the Cauchy point step, otherwise, we compute and take the dogleg step $p_k^D$. In summary, we set the step $s_k$ as
	\begin{equation}\label{eq:steps}
		s_k=\left\{ \begin{array}{cl} \vspace{0.2cm}
			p_k^B     &  \text{ if } \|p_k^B\|\le\Delta_k \ ,\\ \vspace{0.2cm}
			p_k^C     &  \text{ if } \|p_k^B\|>\Delta_k \text{ and } \|p_k^C\|=\Delta_k \ ,\\ 
			p_k^D & \text{ if } \|p_k^B\|>\Delta_k \text{ and } \|p_k^C\|<\Delta_k \ .
		\end{array}   
		\right.
	\end{equation}
	Then it is obvious that the step $s_k$ is always within the trust-region, namely, $\|s_k\|\le \Delta_k$.
		
	%In Algorithm \ref{tr_alg} is similar to what is given by the general framework in \cite[Alg. 4.1]{NW06}. Our method  bears similarity to the dogleg method but it uses a different quadratic model, that is it uses $B_k^TB_k$ to construct the subproblem. 
	\begin{algorithm}[ht]
		\caption{$J$-symmetric Quasi-Newton Method with Trust-region (J-symm-Tr)}
		\label{tr_alg}
		\begin{algorithmic}[1]
			\STATE Initialize with solution $z_0\in\R^{m+n}$, Jacobian estimation $B_0\in\R^{(m+n)\times (m+n)}$, maximum allowed trust-region radius $R_0>0$, initial trust-region radius $\Delta_0 \in (0,~R_0]$, parameter $\hat{\beta}=0.9$, sufficient decrease threshold 
			$\zeta \in (0, 10^{-3})$ and
			%and $\zeta_1 \in (\zeta, 0.5)$, 
			iteration counter $k=0$.
			\FOR {$k=1,2,3,\ldots,$}
			\STATE compute $p_k^B$ via \eqref{eq:pBk}
			\IF {$\|p^B_k \| \leq \Delta_k $ }
			\STATE $s_k = p^B_k $
			\ELSE 
			\STATE compute $p^C_k$ via \eqref{eq:pCk}
			\IF {$\|p^C_k\|=\Delta_k$}
			\STATE $s_k=p_k^C$
			\ELSE
			\STATE compute $p_k^D$ via \eqref{dogleg}
			\STATE set $s_k=p_k^D$
			\ENDIF
			\ENDIF
			\\ 
			\STATE evaluate $ \rho_k$ from \eqref{pred}  
			\IF {$\rho_k \leq 0.5 $ }
			\STATE $\Delta_{k+1} = \Delta_k/2$
			\ELSE
			\STATE $\Delta_{k+1} = \min\{2 \Delta_k~,~R_0\}$
			\ENDIF
			\IF{$\rho_k \geq \zeta$}
			\STATE $z_{k+1} = z_k +s_k$
			%\STATE $N = \|s_k\|$
			\ELSE
			%\STATE $s_k = Ns_k/\|s_k\|$
			\STATE $z_{k+1} = z_k$
			\ENDIF	
			\STATE $y_k = F(z_k+s_k) - F(z_k)$
			\STATE update $B_{k+1}$ via \eqref{Ebar3_beta} with $\beta_k$ uniformly randomly chosen from $[1-\hat{\beta},1+\hat{\beta}]$
			\STATE $k = k+1$	
			\ENDFOR 
		\end{algorithmic}
	\end{algorithm}

	Next, we compute the ratio between the actual decay and the predicted decay of the merit function $\rho_k$ as
	\beq\label{pred}
	\rho_k := \dfrac{\| F(z_k) \|^2/2 - \| F(z_k+s_k) \|^2/2}{m_k(0) - m_k(s_k)} \ .
	\eeq
	If the ratio $\rho_k$ is reasonably large (i.e., $\rho_k >  0.5$), the step $s_k$ provides sufficient decay on the merit function, and we safely
	%call this iteration a successful iteration and 
	expand the trust-region radius (recall $R_0$ is the maximal trust-region radius specified by the user): 
	$$
	\Delta_{k+1} = \min\{2 \Delta_k, R_0\} \ ,
	$$
	otherwise, we reduce the trust-region radius:
	$$
	\Delta_{k+1} =  \Delta_k/2 \ .
	$$
	Moreover, if $\rho_k$ is not too small (i.e., $\rho_k\ge \zeta\in (0,10^{-3})$), we update the iterate solution by accepting the step $z_{k+1}=z_k+s_k$, and call it a \emph{valid step}; otherwise we reject the update and take a \emph{null step} by setting $z_{k+1}=z_k$ \ . 

	Finally, we update the Jacobian estimation $B_{k+1}$ by a slightly modified version of \eqref{Ebar3} in order to guarantee the non-singularity of $B_{k+1}$:
	\beq\label{Ebar3_beta} 
	B_{k+1} = B_k + \beta_k\dfrac{Js_k(y_k-B_ks_k)^TJ+ (y_k-B_ks_k)s_k^T}{s_k^Ts_k}  - \beta_k^2\dfrac{(Js_k)^T(y_k-B_ks_k)Js_ks_k^T}{(s_k^Ts_k)^2}\ ,
	\eeq
	where for any given $\hat\beta \in (0,1)$ we pick $1-\hat\beta\leq\beta_k\leq1+\hat\beta$ such that  $B_{k+1} $ is nonsingular for any $k$. Indeed, suppose $B_k$ is nonsingular, then there only exists finite number of $\beta_k$ such that $B_{k+1}$ is singular, thus $B_{k+1}$ is nonsingular with probability $1$ if we randomly pick $\beta_k$ uniformly from the range $[1-\hat\beta,1+\hat\beta]$. This strategy dates back to Powell \cite{POW_PSB}.
		
	In the rest of this section, we present the global convergence and local superlinear convergence of Algorithm \ref{tr_alg}. 
	First, we define the level set of the merit function as
	$S = \{  z~|~ \|F(z)\|^2/2\leq \|F(z_0)\|^2/2 \}$,
	and the extended level set as
	%\beq\label{R0}
	$$S(R_0) := \{ z+s~|~ \|s\| <R_0 \mathrm{~for~some~}z \in S \} \ .$$
	%\eeq
	
	The following assumptions are needed to develop the global convergence results of Algorithm \ref{tr_alg}:

	\begin{assumption}\label{assump_global} {(Assumptions for Global Convergence)}
	%\begin{assumption}\label{assump_local} {(Assumptions for Local Superlinear Convergence)}
		%\vspace{-0.2cm}
	    \begin{enumerate}[(a)]	
		% \item $F(z) $ is twice continuously differentiable in $\R^{m+n}$. 
		\item 
		%There exist a nonzero open ball of radius $\epsilon$ centered at $z^*$, say $\B_{\epsilon}(z^*)$, wherein
		For any $R_0>0$, $F(z)$ and $\nabla F(z)$ are Lipschitz continuous in $S(R_0)$  namely, there exist constants $\gamma_1$ and $\gamma_2$ such that it holds for any $z, z+s\in S(R_0)$ that
		$$\|F(z) -F(z+s) \| \leq \gamma_1\|s \| \text{\  \  and \ \  } \|\nabla F(z) -\nabla F(z+s) \| \leq \gamma_2\|s \| \ .$$
		\item There exists at least one $z^*$ such that $F(z^*) =0$. Furthermore, $\nabla F(z^*)$ is invertible for all saddle point $z^*$, and there exists $\gamma$ such that $ \gamma \geq \|\nabla F^{-1}(z^*) \|$ . 
		\item  The sequence of vectors $\{ s_k\}$ is uniformly linearly independent \footnote{see Definition \ref{ulidef} in the appendix  for a formal definition of uniform linear independence. }. 
	\end{enumerate} 
	\end{assumption}
	We here examine Assumption \ref{assump_global}. Part (a) impose regularity conditions on the function $L$ (or equivalently on the function $F$). Since $F(z) $ is twice continuously differentiable and if the level set $S$ is bounded, then (a) automatically holds. Part (b) assumes the existence of (at least one) saddle point $z^*$, and furthermore, the saddle point $z^*$ is non-degenerate (i.e., $\nabla F(z^*)$ is invertible). Part (c) implies that every $n+m$ consecutive steps in the sequence $\{s_k/\|s_k\|\}$ span the entire $\R^{n+m}$. The non-degenerate assumption (b) and the uniformly linearly independent assumption (c) are the classic assumptions for obtaining the global convergence of a quasi-Newton method for a minimization problem. As an example see \cite[Theorem 6.2]{NW06} which requires such conditions in order for SR1 update to generate a good Hessian approximation. We here extend them to minimax problems. 
		
	Our main theoretical results are presented in the following two theorems:
	%%%%%%%%%%%%%%%%%%%%%%%%%%%%%%%%%%%%%%%%%%%%%%%%%%%%%%%%%%%%%%%%%%%%%%
	\begin{theorem}\label{glob_conv_thm}
		Consider Algorithm \ref{tr_alg} to solve the minimax problem \eqref{minimax}. 
		%Suppose there exists an optimal minimax solution $z^*$ that satisfy Assumption \ref{assump_global} and in general 
		Under Assumption \ref{assump_global}, it holds that the sequence $\{ g_k\}$ generated by 
		Algorithm \ref{tr_alg} converges to 0, that is
		\beq\label{glim}
		\lim_{k \to \infty} \| g_k\| = 0 \ .
		\eeq
	\end{theorem}
	Theorem \ref{glob_conv_thm} states that under Assumption \ref{assump_global} Algorithm \ref{tr_alg} generate iterates such that the gradient $g_k$ converges to $0$. As a direct consequence of Theorem \ref{glob_conv_thm}, we know that
	if $\nabla F(z_k)$ is nonsingular and 
	bounded, then $F(z_k)$ converges to $0$ by noticing $g_k = \nabla F(z_k)^TF(z_k)$. Furthermore, if the saddle point solution $z^*$ is unique, Theorem \ref{glob_conv_thm} implies $z_k\rightarrow z^*$. Similar arguments appear in Powell's hybrid algorithm for minimization problem~\cite{hybrid}.
	%%%%%%%%%%%%%%%%%%%%%%%%%%%%%%%%%%%%%%%%%%%%%%%%%%%%%%%%%%%%%%%%%%%%%%
	
	Now we assume $\{z_k\}$ converges to a stationary solution, then the next theorem states that (1) $\{ B_k \}$ must converge to $\nabla F(z^*)$, namely $B_k$ eventually provides a good approximation of the Jacobian; (2) states that Algorithm \ref{tr_alg} is R-superlinearly convergent, which showcases the global convergence property of Algorithm \ref{tr_alg}.

    % Moreover, the null step subsequence can be modified to converge to 0. When $s_k$ is a null step, we can reset $\|s_k\|$ to be equal to the previous valid step norm. This modification does not change any of our results and the only extra cost is to recompute $F(z_k+s_k)$ which is needed to compute $y_k$. For simplicity, we drop this modification and instead assume $\{s_k\} \to 0$.}
    %When $s_k$ is a null step with a  norm greater than previous valid step norm, we reset   $\|s_k\|$ to be equal to that norm. The only extra cost is to recompute $F(z_k+s_k)$ which is needed to compute $y_k$. This modification does not change any of our results and for simplicity, we drop this modification and instead assume $\{s_k\} \to 0$.}
	
	% Theorem \ref{B_conv} assumes that $\{s_k\}$ converges to zero. This is a reasonable assumption since $z_k \to z^*$ and \textcolor{blue}{thus} $\|z_{k+1}-z_{k}\| \to 0$, so we know that the valid step subsequence of $\{s_k\}$ converges to 0. To guarantee $\{s_k\} \to 0$ we can modify Algorithm \ref{tr_alg}  such that when $s_k$ is a null step with a  norm greater than previous valid step norm, we reset   $\|s_k\|$ to be equal to that norm. This modification does not change any of our results and the only extra cost is to recompute $F(z_k+s_k)$ which is needed to compute $y_k$. In order to keep Algorithm \ref{tr_alg} simple, we drop this modification and instead assume $\{s_k\} \to 0$.
		%
	\begin{theorem}\label{B_conv}
		Consider Algorithm \ref{tr_alg} to solve the minimax problem \eqref{minimax}.
		Suppose Assumption \ref{assump_global} holds, $\{z_k\}$ converges to a saddle point $z^*$ such that $F(z^*)=0$ and $\{s_k\}$ converges to zero, then it holds that 
		
		1. $\{ B_k \}$ converges to $\nabla F(z^*)$ \ .
		
		2. Algorithm \ref{tr_alg} is R-superlinearly convergent to $z^*$.
	\end{theorem}
		
	%%%%%%%%%%%%%%%%%%%%%%%%%%%%%%%%%%%%%%%%%%%%%%%%%%%%%%%%%%%%%%%%%%%%
	%%%%%%%%%%%%%%%%%%%%%%%%%%%%%%%%%%%%%%%%%%%%%%%%%%%%%%%%%%%%%%%%%%%%%%
% 	\begin{remark}
	    {We comment that we assume $\{s_k\}$ converges to zero in Theorem \ref{B_conv} just to simplify the proof. Actually, this does not impose any additional assumptions. The reason is that suppose all steps are valid steps, then $z_k \to z^*$ implies $\|s_k\|=\|z_{k+1}-z_{k}\| \to 0$. Otherwise, suppose there is any null step $k$, we can instead rescale $s_k$ so that its norm is the previous valid step norm, and the proof of Theorem \ref{B_conv} keeps valid for a scaler change on $s_k$ in the null steps.}
	   % is required only for simplicity of the statement and the analysis following Theorem \ref{glob_conv_thm}. 
	   % , without loss of generality,  we can assume that $\{s_k\}$ converges to zero. Here is the reason. Suppose all steps are valid steps, then $\{s_k\}$ converges to 0 because $z_k \to z^*$ and thus $\|s_k\|=\|z_{k+1}-z_{k}\| \to 0$. If any null step is taken, we can instead set $\|s_k\|$ to be the previous valid step norm, which does not change the algorithm or the analysis. Then the previous argument still holds.
% 	\end{remark}
	        
% 	\end{comment}
		
	In the remainder of this section, we present proofs for the above two theorems. We start with presenting three simple facts:
	\begin{fact}
		As a direct consequence of Assumption \ref{assump_global} we have
		$\|\nabla F(z)\| \leq \gamma_1$ and $\|\nabla^2 F(z)\| \leq \gamma_2$ .
	\end{fact}
	\begin{fact}\label{factM_0}
		Let $D_0 = \|z_0-z^*\|$. Then for any $z\in S(R_0)$, $\|  F(z) \|$ is upper-bounded as
		$$ %\leq \gamma_1\|z-z^*\|
		\|  F(z) \|  \leq \gamma_1 (D_0+R_0) \ .
		$$
% 		%%Uncomment to see the proof:
% 		\begin{comment}
\vspace{-0.3in}
    \begin{proof}
				Since $z\in S(R_0)$, there exist $\bar z \in S$ such that $\|z-\bar z \| <R_0$. Then, $\|F(\bar z)\| \le \|F(z_0)\|$. So
				$$
				\|  F(z) \| \le \|  F(z) - F(\bar z)\| + \|  F(\bar z) \| \le  \gamma_1\|z-\bar z\| +  \|  F(z_0) \| \leq \gamma_1 (R_0+D_0) \ ,
				$$
    the second inequality comes from Assumption \ref{assump_global} (a).
			\end{proof}

% 		\end{comment}
	\end{fact}
	%%%%%%%%%%%%%%%%%%%%%%%%%%%%%%%%%%%%%%%%%
	\begin{fact}\label{lip_g_lemma}
		Denote $\mu = \gamma_2\gamma_1 (D_0+R_0)  + \gamma_1^2$. Then it holds for any $z\in S$,and $\|s\| \le R_0$ that
		\beq\label{g_lip}
		\|\nabla F(z)^TF(z) - \nabla F(z+s)^TF(z+s)\| \leq \mu \| s\| \ .
		\eeq
	\end{fact}
	\begin{proof}
		It holds that
		%\small
		%\textbf{Proof of Lemma \ref{lip_g_lemma}:}\\
		\begin{align*}
			%\| g_k(z_k) -g_k(z_k+s) \| &= 
			&\|\nabla F(z)^TF(z) - \nabla F(z+s)^TF(z+s)\| \\
			&=
			\left\|\Big(\nabla F(z) - \nabla F(z+s)\Big)^TF(z)+ \nabla F(z+s)^T\Big(F(z)-F(z+s)\Big) \right\|\\
			&\leq  \gamma_2\|F(z)\| \|s\| + \gamma_1^2 \|s\|\\
			&\leq  \Big( \gamma_2\gamma_1 (D_0+R_0)  + \gamma_1^2 \Big) \|s\| \ ,
		\end{align*}
		the first equality comes from adding and subtracting $\nabla F(z+s)^TF(z)$, the following inequality comes from  the Lipschitz-continuity of $F(z)$ and $\nabla F(z)$,
		%\eqref{lip_F}, \eqref{lip_nabla} and \eqref{nabla_bd}, 
		and finally the last inequality uses
		Fact \ref{factM_0}.
		% \eqref{M_0}.
	\end{proof}
	%%%%%%%%%%%%%%%%%%%%%%%%%%%%%%%%%%%%%%%%%%%%%%%%%%%%%%%%%%%%%%%%%%%%%
	The proof of theorem \ref{glob_conv_thm} heavily relies on the following two propositions. Proposition \ref{B_bd} shows that $\|B_k\|$ is always upper-bounded. Proposition \ref{cauchy_armijo} shows that $m_k(s)$ has sufficient decay in Algorithm \ref{tr_alg}.
	\begin{proposition}\label{B_bd}
		Suppose Assumption \ref{assump_global} holds. Then there exists $\nu_2$ such that it holds for any $k\ge 0$ that $ \|B_k\|\le \nu_2$ \ .
	\end{proposition}
	\begin{proposition}
		\label{cauchy_armijo}
		Algorithm \ref{tr_alg} generates steps $s_k$ such that 	for all $k$ we have:
		%\beq\label{c_a}
		$$m_k(0) - m_k(s_k) \geq \frac{\|g_k\|}{2} \min \Big\{ \Delta_k ~,~ \frac{\|g_k\|}{\nu_2^2} \Big\} \ .$$
		%\eeq
	\end{proposition}
	\begin{remark}
		As a direct consequence of Proposition \ref{cauchy_armijo}, Algorithm \ref{tr_alg} is a nonincreasing algorithm in $\|F(z_k)\|$, namely, $\|F(z_{k+1})\|\le \|F_{k}\|$ for all iterate $k$. This is because (i) if the $k$-th step is a null step, then $z_{k+1}=z_k$ thus it is a nonincreasing step; (ii) if the $k$-th step is a valid step, then $$\|F(z_k)\|^2/2-\|F(z_{k+1})\|^2/2\ge \zeta \big(m_k(0)-m_k(s_k)\big)\ge 0 \ ,$$
		where the last inequality is from Proposition \ref{cauchy_armijo}. Furthermore, if $g_k\not=0$, then Proposition \ref{cauchy_armijo} shows that a valid step of Algorithm \ref{tr_alg} provides sufficient decay in the merit function $\|F(z)\|^2/2$. This observation is the cornerstone of the convergence results of Algorithm \ref{tr_alg}.
	\end{remark}
	%%%%%%%%%%%%%%%%%%%%%%%%%%%%%%%%%%%%%%%%%%%%%%%%%%%%%%%%%%%%%%%%%%%%%%%%%%%%%%%
	To show Proposition \ref{B_bd} and Proposition \ref{cauchy_armijo}, we first establish two simple lemmas to better understand the update rule of $B_k$.
	\begin{lemma}\label{global_Mbar1}
		Let 
		\beq\label{Q}
		Q_k= I - \beta_k\frac{s_ks_k^T}{s_k^Ts_k} \ .
		\eeq
		Then it holds that 
		$\| Q_k\| \leq 1$. Furthermore, under Assumption \ref{assump_global}, there exists a constant $\theta \in (0,1)$ and an index $K$ such that for $k\geq K$ we have:
		%$\|Q_{k}\| \leq \alpha$ for any $k\geq K $.
		$$ \Big\|\prod_{j=k+1}^{k+n+m} Q_j \Big\| \leq \theta \ . $$
	\end{lemma} 
	\begin{proof}
		Notice that it holds for any vector  $v\in\R^{m+n}$ that
		$\|Q_kv \|^2= \| v\|^2 - \beta_k(2-\beta_k)(v^Ts_k)^2/\|s_k\|^2$, and since $0 < 1-\hat\beta \leq \beta_k \leq 1+\hat\beta < 2$, then,  $\|Q_kv \| \leq \| v\|$. Hence,
		%	\beq\label{normQ}
		$\| Q_k\| \leq 1$.
		Furthermore, the existence of such $K$ and $\theta\in(0,1)$ is from 
		Theorem \ref{uliequval} in the appendix, 
		%\cite[Theorem 5.3.]{globalB} (see the appendix) 
		following the uniform linear independence assumption of $\{s_k/\|s_k\|\}$ and $|1-\beta_k| \leq \hat\beta$.
	\end{proof}	
	%%%%%%%%%%%%%%%%%%%%%%%%%%%%%%%%%%%%%%%%%%%%%%%%%%%%%%%%%%%%%%%%%%%%%%%%%%%%%%%
	\begin{lemma}\label{global_Mbar2}
		For any $J$-symmetric matrix $O\in\R^{(n+m)\times(n+m)}$, let $M_k = B_k -O$ and  $M_{k+1} = B_{k+1} -O$, then it holds that
		\begin{align}\label{Ebar_beta}
			M_{k+1} &= JQ_kJM_kQ_k 
			+ \beta_k\frac{(y_k- Os_k)s_k^T}{s_k^Ts_k} + \beta_k\frac{Js_k(y_k-Os_k)^TJ}{s_k^Ts_k}Q_k \ ,
		\end{align}
		and in particular, we obtain the following equivalent representation of \eqref{Ebar3_beta} by letting $O=0$:
		\begin{align}\label{BP_beta}
			B_{k+1} &= JQ_kJB_kQ_k 
			+ \beta_k\frac{y_ks_k^T}{s_k^Ts_k} + \beta_k\frac{Js_ky_k^TJ}{s_k^Ts_k}Q_k \ ,
		\end{align}
		where $Q_k$ is defined in \eqref{Q}.
	\end{lemma}
	\begin{proof}
		Notice that we can write  \eqref{Ebar3_beta} as		
		\small
		\begin{align*}
			B_{k+1} -O = ~&B_k-O + \dfrac{Js_k(y_k-B_ks_k +Os_k-Os_k)^TJ}{(s_k^Ts_k)/\beta_k} + \dfrac{(y_k-B_ks_k+Os_k-Os_k)s_k^T}{(s_k^Ts_k)/\beta_k} -\\
			& \dfrac{(Js_k)^T(y_k-B_ks_k+Os_k-Os_k)Js_ks_k^T}{\big((s_k^Ts_k)/\beta_k\big)^2}\ .
		\end{align*}
		The rest of the  proof follows the same steps of  Lemma \ref{local_Mbar} by replacing $(s_k^Ts_k)/\beta_k$
		with $s_k^T s_k$.	
	\end{proof}
		
	Now we are ready to prove Proposition \ref{B_bd} and Proposition \ref{cauchy_armijo}.
	%%%%%%%%%%%%%%%%%%%%%%%%%%%%%%%%%%%%%%%%%
	\begin{proof}[Proof of Proposition \ref{B_bd}]
		First, notice that $\|y_k\|=\|F(z_k+s_k)-F(z_k)\| \leq \gamma_1\|s_k\|$, where we utilize the fact that $F$ is $\gamma_1$-Lipschitz continuous in $S(R_0)$, $z_{k}\in S$ and $\|s_k\|\le R_0$. By	expanding $B_{j+1}$ using \eqref{BP_beta} for $j=k+n+m$ we obtain
		\begin{align*}
			\|B_{j+1} \| &=  \left\|JQ_jJ B_j Q_j  + \beta_j \dfrac{y_js_j^T}{s^T_js_j} +  \beta_j \dfrac{Js_jy_j^TJ}{s^T_js_j} Q_j \right\| \\
			& \leq  \|JQ_jJ\|\| B_jQ_j\|  + \beta_j \dfrac{\|y_j\|}{\|s_j\|}  + \beta_j \dfrac{\|Jy_j\|}{\|s_j\|}  
			\leq \|B_jQ_j\| + 4\gamma_1 \ ,
		\end{align*}
		where  the first inequality uses Cauchy-Schwarz followed by $\|Js_j\| =  \|s_j\|$ and $\|Q_j \|\leq 1$, the second inequality uses  $\|J\|= 1$,  $ \|Q_j \|\leq 1$, $\beta_j < 2$, $\|Jy_j\| =  \|y_j\|$, and   $\|y_j\| \leq \gamma_1 \|s_j\|$. 
		Expanding $B_j$ in the R.H.S. of the inequality $\|B_{j+1} \| \leq \|B_jQ_j\| + 4\gamma_1 $ recursively for $n+m-1$ times, using \eqref{BP_beta} and in the same way as we did for $B_{j+1}$,  
		%and applying $\|Q_{j} \|\leq 1$ each time 
		we obtain:
		\begin{align*}
			\|B_{k+n+m+1} \|  \leq \|B_{k+1}Q_{k+1}\hdots Q_{k+n+m-1}Q_{k+n+m}\| + 4(n+m)\gamma_1 \ .
		\end{align*}
		It follows from Lemma \ref{global_Mbar1} that there exists a constant $\theta\in (0,1)$ and index $K$ such that $ \|\prod_{j=k+1}^{k+n+m} Q_j \|\le \theta$ for any $k\ge K$, thus
		\begin{align*}
			\|B_{k+n+m+1} \| 	 \leq \theta  \|B_{k+1}\| + 4(n+m)\gamma_1 \ .
		\end{align*}
		We now apply Lemma \ref{bound_seq} in the appendix  
		%[Lemma 5.5.]\cite{globalB} (see the appendix) 
		and conclude since $4(n+m)\gamma_1$ is upper-bounded, then so is  $\{ \|B_k \|\} $ for $k \ge K+n+m+1$ .  Thus there exists $\nu_2$ such that 
		$\{ \|B_k \|\} \leq \nu_2$ for all $k$.%k=0,1,\hdots$.
	\end{proof}
	%%%%%%%%%%%%%%%%%%%%%%%%%%%%%%%%%%%%%%%%%
	\begin{proof}[Proof of Proposition \ref{cauchy_armijo}]
		We prove the lemma by the following two steps:\\
		(a) $m_k(p^C_k) \geq m_k(s_k),$\\
		(b) $m_k(0) - m_k(p^C_k) \geq \frac{\|g_k\|}{2} \min \{ \Delta_k ~,~ \frac{\|g_k\|}{\nu_2^2} \}.
		$\\	
		To show part (a), it follows from \eqref{eq:steps} that the step $s_k$ takes values from $\{p_k^B, p_k^C, p_k^D\}$. Notice that $p_k^B$ is the global minimizer of $m_k(s)$, thus $m_k(p^C_k)  \geq m_k(p_k^B)$. Then we just need to show that $m_k(p_k^D)\le m_k(p_k^C)$ under the condition that $\|p_k^B\|>\Delta_k$ and $\|p_k^C\|<\Delta_k$, in which case we have $p_k^C=-\|g_k\|^2/({g_k}^TB_k^TB_kg_k) g_k$. Recall that  $p^D_k= p^C_k + \alpha (p^B_k -p^C_k)$, thus we just need to show that $h(\alpha) := m_k\big(p^C_k + \alpha (p^B_k -p^C_k)\big)$ is monotonically nonincreasing in $\alpha$, by noticing $h(0)=m_k(p^C_k)$. This is because $h(\alpha)$ is differentiable with derivative
		\begin{align*}
			h'(\alpha) &= g_k^T(p^B_k -p^C_k) + (p^B_k -p^C_k)^TB_k^TB_kp^C_k + \alpha (p^B_k -p^C_k)^TB_k^TB_k(p^B_k -p^C_k)\\
			&= (p^B_k -p^C_k)^T\Big(g_k+ B_k^TB_kp^B_k -B_k^TB_kp^B_k + B_k^TB_kp^C_k + \alpha B_k^TB_k(p^B_k -p^C_k)\Big)\\
			&= (p^B_k -p^C_k)^T\Big(g_k+ B_k^TB_kp^B_k  -(1 - \alpha) B_k^TB_k(p^B_k -p^C_k)\Big)\\
			&= (p^B_k -p^C_k)^T\Big(0 -(1 - \alpha) B_k^TB_k(p^B_k -p^C_k)\Big)\\
			&\leq 0 \ ,
		\end{align*}
		where the first equality comes from substituting $s_k = p^C_k + \alpha (p^B_k -p^C_k)$ into the definition of $m_k$ in \eqref{mk}, the  second equality and the third equality come from  rearrangement, the fourth equality is a result of the definition of $p^B_k = -B_k^{-1}(B_k^{-1})^Tg_k $, and finally
		the inequality comes from noticing that $B_k^TB_k$ is positive definite and $0<\alpha<1$. This shows part (a).
		
		To show part (b), recall the formulation of Cauchy point $p_k^C$ \eqref{eq:pCk}. 
		If $p^C_k = -{\Delta_k}g_k/{\|g_k\|}$, then it must hold that
		\begin{equation}\label{eq:headphone}
			\Delta_k/\|g_k\|\le\|g_k\|^2/({g_k}^TB_k^TB_kg_k) \ .
		\end{equation}
		Therefore, 
		\begin{align*}
			m_k(0) -m_k(p^C_k) &= \Delta_k\|g_k\| - \frac{1}{2} \frac{\Delta_k^2}{\|g_k\|^2} g_k^TB_k^TB_kg_k\\
			&=	 \Delta_k\|g_k\| - \frac{1}{2} \Delta_k \|g_k\|\frac{\Delta_kg_k^TB_k^TB_kg_k}{\|g_k\|^3}\\
			&\geq \frac{1}{2} \|g_k\| \Delta_k \ ,
		\end{align*}
		where the inequality is from \eqref{eq:headphone}.

		Otherwise, we have 	$p^C_k= -\|g_k\|^2/({g_k}^TB_k^TB_kg_k)g_k$, thus
		\begin{align*}
			m_k(0) -m_k(p^C_k) &= \frac{\|g_k\|^2}{{g_k}^TB_k^TB_kg_k} \|g_k\|^2 - \frac{1}{2} \Big(\frac{\|g_k\|^2}{{g_k}^TB_k^TB_kg_k}\Big)^2 g_k^TB_k^TB_kg_k\\
			&=	 \frac{\|g_k\|^4}{2{g_k}^TB_k^TB_kg_k}\geq \frac{\|g_k\|^2}{2\|B_k\|^2}\geq \frac{\|g_k\|^2}{2\nu_2^2} \ ,
		\end{align*}
		where the last inequality uses Proposition \ref{B_bd}. 
		
		Putting the last two inequalities together we conclude the claim in part (b).% of \eqref{c_a}.
	\end{proof}
	Next, we present the proof of Theorem \ref{glob_conv_thm}. The proof of Theorem \ref{glob_conv_thm} is inspired by \cite[Theorems 4.5-4.6]{NW06}. 
	\begin{proof}[Proof of Theorem \ref{glob_conv_thm}]
		We prove the theorem in two steps:\\
		\textbf{(a). } First we show $\lim\inf_{k \to \infty} \| g_k\| = 0$.\\
		\textbf{(b). } Next we prove $\lim_{k \to \infty} \| g_k\| = 0$. \\
	%%%%%%%%%%%%%%%%%%%%%%%%%%%%%%%%%%%%%%%%%%%%%%%%%%%%%%%%%%%%%%%%%%%%%%%%%%%
	%                proof of  Thm 4.2 part (a)
	%%%%%%%%%%%%%%%%%%%%%%%%%%%%%%%%%%%%%%%%%%%%%%%%%%%%%%%%%%%%%%%%%%%%%%%%%%%
		We first prove \textbf{(a).} by contradiction. Suppose that there exists  $\epsilon>0$ and a positive index $K$ such that $\|g_k\| \geq \epsilon$ for all $k\geq K$.
		It follows from the mean value theorem  on the one-dimensional function $f(t)=\|F(z_k+ts_k)\|^2/2$ that
		$$\|F(z_k+s_k)\|^2/2= \|F(z_k)\|^2/2 + F(z_k+ts_k)^T\nabla F(z_k+ts_k)s_k, $$ 
		for some $0 < t < 1$. By the definition of $m_k(s_k)$, we obtain 
		\begin{align}\label{eq:mouse}
			\begin{split}
				\big|m_k(s_k)  -\|F(z_k+s_k)\|^2/2 \big| &=  \left|F(z_k)^T\nabla F(z_k)s_k + (s_k^TB_k^TB_ks_k)/2-F(z_k+ts_k)^T\nabla F(z_k+ts_k)s_k  \right| \\
				&\leq  (\mu  + \nu_2^2/2)  \|s_k\|^2 \leq (\mu  + \nu_2^2/2)  \Delta_k^2 \ ,
			\end{split}
		\end{align}
		where the first inequality comes from Fact \ref{lip_g_lemma}
		%the application of \eqref{g_lip} 
		and Proposition \ref{B_bd}, and the second inequality uses $s_k\le \Delta_k$.
		%and  \eqref{eqB_bd}.\\
		%Recall that by definition $m_k(0) = \|F(z_k)\|^2/2$. Hence from the definition of $\rho_k$ we obtain:
		Thus, it holds for any $k\ge K$ that
		\begin{equation}\label{eq:keyboard}
			|\rho_k -1 |= \frac{\big|m_k(s_k)  -\|F(z_k+s_k)\|^2/2 \big|}{|m_k(0) - m_k(s_k)|} \leq 
			\frac{ (\mu  + \nu_2^2/2) \Delta_k^2 }{\frac{\epsilon}{2} \min \Big\{ \Delta_k ~,~ \frac{\epsilon}{\nu_2^2} \Big\}}\ , 
		\end{equation}
		where the inequality comes from the Proposition \ref{cauchy_armijo} and \eqref{eq:mouse} by noticing $\|g_k\|\ge \epsilon$.
		
		Define $\bar\Delta :=\min\big\{  \frac{\epsilon/2}{\mu  + \nu_2^2/2} , R_0 \big\}$ and then $\bar\Delta < \frac{\epsilon}{\nu_2^2}$. We here first show by induction that the trust-region radius is lower bounded: 
		\beq\label{Delta_claim}
		\Delta_k \geq \min\left\{ \Delta_K~,~ \frac{\bar\Delta}{4} \right\} \mathrm{~for~any~} k\ge K \ .
		\eeq
		Apparently, \eqref{Delta_claim} holds for $k=K$. Now suppose \eqref{Delta_claim} holds for $k$. If $\Delta_k < \bar\Delta/2$, (so then $\Delta_k =  \min \big\{ \Delta_k ~,~ \frac{\epsilon}{\nu_2^2} \big\}$),
		it follows from \eqref{eq:keyboard} that
		$$  |\rho_k -1 | \leq \frac{ (\mu  + \nu_2^2/2) \Delta_k^2 }{\frac{\epsilon}{2} \Delta_k} \le \frac{\Delta_k}{\bar \Delta} < \frac{1}{2}\ .$$
		Therefore, we have $\rho_k > 1/2$, and as a result, the trust-region increases in the next iteration: $\Delta_{k+1}=\min\{2\Delta_k, R_0\}\ge \Delta_{k}$, thus \eqref{Delta_claim} holds for $k+1$ by induction. Otherwise, we have $\Delta_k \geq \bar\Delta/2$, and it follows from the fact that the trust-region  in one iteration can only contract by a factor of $2$ that $\Delta_{k+1}\ge \bar\Delta/4$, thus \eqref{Delta_claim} holds for $k+1$. Combining the above two cases, we prove \eqref{Delta_claim} by induction.
			
    	Next, if we have an infinite increasing subsequence $\{k^i\}\subseteq\{K, K+1, K+2...\}$ such that $\rho_{k^i} > 1/2$, then we deduce from \eqref{pred} that
    	\begin{align*}
    		\|F(z_{k^i})\|^2/2- \|F(z_{k^i+1})\|^2/2  \ge \rho_{k^i} \left(m_{k^i}(0)-m_{k^i}(s_{k^i})\right) \ge  \frac{\epsilon}{4} \min \left\{ \Delta_{k^i} ~,~ \frac{\epsilon}{\nu_2^2} \right\}
    		\ge \frac{\epsilon}{4} \min \left\{ \Delta_K~,~ \bar\Delta/4 ~,~ \frac{\epsilon}{\nu_2^2} \right\}\ ,
    	\end{align*}
    	where the second inequality uses Proposition \ref{cauchy_armijo} and $\|g_{k^i}\|\ge \epsilon$, and the last inequality is from \eqref{Delta_claim}.
    	Therefore, noticing $\|F(z_k)\|^2/2$ is monotonically nonincreasing, and summing up the above inequality, we have
    	\begin{align*}
    		\|F(z_{K})\|^2/2- \|F(z_{k^i+1})\|^2/2 & \ge \sum_{j=1}^i \|F(z_{k^{j}})\|^2/2- \|F(z_{k^j+1})\|^2/2 
    		\ge \frac{i \epsilon}{4} \min \left\{ \Delta_K~,~ \bar\Delta/4 ~,~ \frac{\epsilon}{\nu_2^2} \right\}.
    	\end{align*}
    	This cannot happen for a large enough $i$ since $\|F(z_{k^i+1})\|^2/2\ge 0$ is lower bounded. 
    	
    	Otherwise, if there is no such infinite subsequence $\{k^i\}$, then there exists $K' \ge K$ such that $\rho_k \le 1/2$ for all $k\ge K'$. As a result, the trust-region radius contracts at each iteration after $K'$, thus $\lim_{k \to \infty} \Delta_k = 0$, which contradicts with \eqref{Delta_claim}. Combining the above two cases, we conclude that our original assumption cannot hold and therefore  
    	$\lim\inf_{k \to \infty} \| g_k\| = 0$.
		%%%%%%%%%%%%%%%%%%%%%%%%%%%%%%%%%%%%%%%%%%%%%%%%%%%%%%%%%%%%%%%%%%%%%%%%%%%
		%                proof of  Thm 4.2 part (b)
		%%%%%%%%%%%%%%%%%%%%%%%%%%%%%%%%%%%%%%%%%%%%%%%%%%%%%%%%%%%%%%%%%%%%%%%%%%%
			
		Now we turn to \textbf{(b).} We first present the high-level ideas of the proof. We will show \textbf{(b).} by contradiction. Suppose \textbf{(b).} does not hold, namely, there exists $\epsilon>0$ and an infinite increasing subsequence $\{t_i\}_{i=1}^{\infty}$ of $\{1,2,...\}$ such that  $\|g_{t_i}\| \ge \epsilon$. Then we will show that there exists a constant $C>0$ and an increasing subsequence $\{u_i\}_{i=1}^{\infty}$ of $\{t_i\}_{i=1}^{\infty}$ such that $\|F(z_{u_{i}})\|^2/2 - \|F(z_{u_{i+1}})\|^2/2\ge C$. Thus by the monotonicity of $\|F(z_k)\|^2/2$, we have $\|F(z_{u_{i}})\|^2/2\rightarrow -\infty$ as $i\rightarrow \infty$, which contradicts with the fact that $\|F(z_{u_{i}})\|^2/2\ge 0$. In the rest of this proof we construct the sequence $\{u_i\}_{i=1}^{\infty}$ by induction.
		
		For initialization, we set $u_1 = t_1$. Next, for a given $u_i$, we show how to build $u_{i+1}$.
		%such that $\|g_{t_{j+1}}\| > \epsilon$ and \eqref{contradiction} hold:
		Consider the point $z_{u_i}$ and a close ball $\B(z_{u_i},R)=\{z~|~\|z-z_{u_i}\|\le R\}$ with center $z_{u_i}$ and radius $R$, where
		$$R := \min\{\epsilon/(2\mu) ~,~ R_0\} \ .$$ 
		Notice that for any $z_k\in \B(z_{u_i},R)$, we have from \eqref{g_lip} that
		$$\|  g_{k} -g_{u_i}  \|\leq \ \mu \|z_{k} - z_{u_i}\|  \leq \mu R \leq \epsilon/2 \ ,$$ 
		and thus
		%\beq \label{del_half}
		$$\| g_{k} \|  \geq   \| g_{u_i}\| -\| g_{u_i} - g_{k} \| \geq \epsilon - \epsilon/2=\epsilon/2 \ .  $$
		
		Recall from \textbf{(a).} that there is a subsequence of  $\{\|g_k\|\}_{k=0}^{^{\infty}}$ that converges to 0, thus there exists at least one solution in the sequence $\{z_{k}\}_{k\ge u_i}$ that leaves  $\B(z_{u_i},R)$. 
		% \red{change $l+1$ to $l$.}
		Let $z_{l+1}$ 
		be the first of such iterates.  Then $\|z_k-z_{u_i}\|\le R$ for $k=u_i+1, ..., l$, and it holds that
		\begin{align}
			\begin{split}
				\|F(z_{u_i})\|^2/2 -\|F(z_{l+1})\|^2/2  & = \sum_{k=u_i}^{l} \|F(z_{k})\|^2/2 -\|F(z_{k+1})\|^2/2 \\
				&= \sum_{k=u_i,...,l, z_k\not=z_{k+1}} \|F(z_{k})\|^2/2 -\|F(z_{k+1})\|^2/2 \\
				&\ge \zeta \sum_{k=u_i,...,l, z_k\not=z_{k+1}} m_k(0)-m_k(s_k) \\
				&\geq \frac{\zeta}{2} \sum_{k=u_i,...,l, z_k\not=z_{k+1}} \|g_{k}\|\min \Big\{ \Delta_{k} ~,~ \frac{ \|g_{k}\|}{\nu_2^2} \Big\} \\
				&\ge \frac{\zeta\epsilon}{4} \sum_{k=u_i,...,l, z_k\not=z_{k+1}}\min \Big\{ \Delta_{k} ~,~ \frac{ \epsilon}{2\nu_2^2} \Big\}  \ ,
			\end{split}\label{summation}
		\end{align}
		where the second equality considers only the valid steps (namely ignores the null steps), the first inequality utilizes the criteria of a valid step, the second inequality uses Proposition \ref{cauchy_armijo}, and the third inequality is implied from $\| g_{k} \| \geq \epsilon/2$ since $z_k\in \B(z_{u_i},R)$. 
			
		Next we present a lower bound for \eqref{summation}. There are only two possibilities:
		\begin{quoting}
			\textbf{(i).} Suppose there exists $k$ in the summation in the R.H.S. of \eqref{summation} that $\Delta_{k} > \epsilon/(2\nu_2^2)$, then we  obtain by noticing $\|F(z_{k})\|^2/2$ is monotonically nonincreasing that
			$$
			\|F(z_{u_i})\|^2/2 -\|F(z_{l+1})\|^2/2  \geq  \frac{\zeta\epsilon}{4} \frac{ \epsilon}{2\nu_2^2} \ .
			$$
			
			\textbf{(ii).} Otherwise, $\Delta_{k} \leq \epsilon/(2\nu_2^2)$ for all of $k$ in the summation in the R.H.S. of \eqref{summation}. Notice that $z_{l+1}$ is outside $\B(z_{u_i},R)$, thus $$R\le \|z_{l+1}-z_{u_i}\|\le \sum_{k=u_i}^{l} \|z_{k+1}-z_k\|= \sum_{k=u_i,...,l, z_k\not=z_{k+1}} \|z_{k+1}-z_k\|\le \sum_{k=u_i,...,l, z_k\not=z_{k+1}} \Delta_k\ ,$$
			where the last inequality uses $\|z_{k+1}-z_k\|=\|s_k\|\le \Delta_k$ for a valid step. Therefore, it holds from \eqref{summation} that
			$$
			\|F(z_{u_i})\|^2/2 -\|F(z_{l+1})\|^2/2  \geq \frac{\zeta\epsilon}{4} \sum_{k=u_i,...,l, z_k\not=z_{k+1}} \Delta_{k} \geq  \frac{\zeta\epsilon}{4}R = \frac{\zeta\epsilon}{4} \min\{\epsilon/(2\mu) ~,~ R_0\} \ .
			$$
		\end{quoting}
		Combining (i) and (ii), we arrive at
		$$ \|F(z_{u_i})\|^2/2 - \|F(z_{l+1})\|^2/2   \geq  \frac{\zeta\epsilon}{4}\min\Big\{\epsilon/(2\mu) ~,~ R_0 ~,~ \epsilon/(2\nu_2^2) \Big\} \ .$$
		Now choose $u_{i+1}$ to be the first index in the infinite sequence $\{t_i\}_i^{\infty}$ such that $u_{i+1}\ge l+1$, then such $u_{i+1}$ exists, because $\{t_i\}_i^{\infty}$ has infinite values and

		$$\|F(z_{u_{i}})\|^2/2 - \|F(z_{u_{i+1}})\|^2/2 \geq \|F(z_{u_i})\|^2/2 - \|F(z_{l+1})\|^2/2  \geq \frac{\zeta\epsilon}{4}\min\Big\{\epsilon/(2\mu) ~,~ R_0 ~,~ \epsilon/(2\nu_2^2) \Big\} \ ,$$
		%\eeq
		where the first inequality uses the monotonicity of $\|F(z_{k})\|^2/2$. As a result, let $C=(\zeta\epsilon/4)\min\big\{\epsilon/(2\mu), R_0, \epsilon/(2\nu_2^2) \big\}>0$, then we have  that $\|F(z_{u_{i}})\|^2/2\le \|F(z_{u_{1}})\|^2/2-(i-1) C \rightarrow -\infty $ when $i\rightarrow\infty$,
		which contradicts with the fact that $\|F(z_{u_{i}})\|^2/2\ge 0$. This finishes the proof by contradiction.
		%, according to the bounded monotone convergence theorem the left side of the inequality is converging to  0 while the right side is a constant and therefore \eqref{contradiction} is a contradiction. This also gives $\|g_{t_{j+1}}\| > \delta$ and thus we have a contradiction. This finishes the second step of the proof by contradiction.
		%\normalsize	
	\end{proof}
	%%%%%%%%%%%%%%%%%%%%%%%%%%%%%%%%%%%%%%%%%%%%%%%%%%%%%%%%%%%%%%%%%%%%%%%%%%%%%%%%%%%%%%%%%%%
	%                       proof of Thm 4.3 (part1)
	%%%%%%%%%%%%%%%%%%%%%%%%%%%%%%%%%%%%%%%%%%%%%%%%%%%%%%%%%%%%%%%%%%%%%%%%%%%%%%%%%%%%%%%%%%%
	Finally, we prove Theorem \ref{B_conv}:
	\begin{proof}[Proof of Theorem \ref{B_conv} part (1)]
		Define	 
		$\delta_k := \sup_{0\leq t \leq 1} \|\nabla F(z_k+ts_k)-\nabla F(z^*)\| ~.$ Then it follows from $z_k\rightarrow z^*$, $s_k\rightarrow 0$ and the continuity of $\nabla F(z)$ that $\delta_k \rightarrow 0$. 
		Let $T(z) = F(z) -  \nabla F(z^*)z$, then $\nabla T(z) = \nabla F(z) -  \nabla F(z^*)$. Thus, it holds that
		\begin{align}\label{eq:water}
			\begin{split}
				\|y_k- \nabla F(z^*)s_k\| &= \|F(z_k+s_k) - F(z_k)- \nabla F(z^*)s_k\| = \|T(z_k+s_k) - T(z_k) \| \\
				& \le \sup_{0\leq t \leq 1} \| \nabla T\big(z_k + ts_k\big) \|\|s_k\| = \delta_k \|s_k\| \ .
			\end{split}
		\end{align}
		The rest of the proof is very similar to the proof of Proposition \ref{B_bd}. 
		Applying Lemma \ref{global_Mbar2} with $O = \nabla F(z^*)$ we expand $M_{j+1}$  using \eqref{Ebar_beta}   for $j = k+n+m$ in the following way:
		\begin{align*}
			\|M_{j+1} \| & = \left\|JQ_jJM_jQ_j
			+ \beta_j\frac{\big(y_j- \nabla F(z^*)s_j\big)s_j^T}{s_j^Ts_j} + \beta_j\frac{Js_j\big(y_j-\nabla F(z^*)s_j\big)^TJ}{s_j^Ts_j}Q_j \right\| \\
			& \leq \|JQ_jJ\|\|M_jQ_j\|	+ 2\beta_j\delta_j  \leq   \|M_jQ_j\| + 4\delta_j  \ ,
		\end{align*}
		\normalsize
		where the first inequality utilizes $\| Q_j\| \leq 1$, \eqref{eq:water} and the fact $\|Jq\|=\|q\|$ for any vector $q$ of the appropriate size,
		the second inequality comes from $\|J\|= 1$,  $ \|Q_j \|\leq 1$ and $\beta_j < 2$. Expanding $M_j$ recursively for $n+m-1$ times in the R.H.S. of the inequality $\|M_{j+1} \| \leq \|M_jQ_j\| + 4\delta_j$,  using \eqref{Ebar_beta}  and in the same way as we did for $M_{j+1}$ we obtain:		
		$$	\|M_{k+n+m+1}	\|	 \leq  \|M_{k+1} Q_{k+1}\hdots Q_{k+n+m-1}Q_{k+n+m}\| +  4 \sum_{j=k+1}^{k+m+n}\delta_j \ . $$
		It follows from Lemma \ref{global_Mbar1} that there exists a constant $\theta\in (0,1)$ and index $K$ such that $ \|\prod_{j=k+1}^{k+n+m} Q_j \|\le \theta$ for any $k\ge K$, thus 
		\begin{align*}
			\|M_{k+n+m+1} \| 	 \leq \theta  \|M_{k+1}\| + 4 \sum_{j=k+1}^{k+m+n}\delta_j \ .
		\end{align*}
		We now  apply  Lemma \ref{bound_seq} in the appendix  
		%[Lemma 5.5.]\cite{globalB} (see the appendix) 
		and conclude since $4 \sum_{j=k+1}^{k+m+n}\delta_j \to 0$, then $\|M_k \| \to 0$ as $k \to \infty$.
		%for $k \ge K+n+m+1$.  
		Recalling  $M_k = B_k - \nabla F(z^*)$, the proof is complete.
	\end{proof}
	
	To prove part (2) of Theorem  \ref{B_conv}, we first present two lemmas.
		
	%%%%%%%%%%%%%%%%%%%%%%%%%%%%%%%%%%%%%%%%%%%%%%%%%%%%%%%%%%%%%%%%%%%%%%
	%%%%%%%%%%%%%%%%%%%%%%%%%%%%%%%%%%%%%%%%%%%%%%%%%%%%%%%%%%%%%%%%%%%%%%
	\begin{lemma}\label{BInv_cor}
		Under the assumptions stated in Theorem \ref{B_conv} part (1), it holds that $\|B_k^{-1}\|$ is upper-bounded, that is, there exists a positive value  $\nu_1$ such that
		%	\beq\label{nu}
		$$	\| B_k^{-1} \| \leq \nu_1 \ .$$
		%	\eeq
		\begin{proof}
		Recall from the construction that $B_k$ is invertible and from Assumption \ref{assump_global} that $\nabla F(z^*)$ is nonsingular with $ \| \nabla F^{-1}(z^*) \|\le \gamma$. Notice that Theorem \ref{B_conv} part (1) shows that $\| B_k - \nabla F(z^*) \|\rightarrow 0$, 
		%thus $\|\nabla F^{-1}(z^*) B_k - I\|\rightarrow 0$ as $k\rightarrow 0$, 
		whereby there exists $K$ such that 
		%$\|\nabla F^{-1}(z^*) B_k - I\|\le \frac{1}{2}$ 
		$\|B_k - \nabla F(z^*) \|\le 1/(2\gamma)$ 
		for any $k\ge K$. We can now apply  Banach Perturbation Lemma (Lemma \ref{banach} in the appendix) to the matrices $\nabla F(z^*)$ and $B_k$ (note $\gamma/(2\gamma) = 0.5 <1$) and obtain:
		$$\| B_k^{-1}\| \leq \frac{\gamma}{1-0.5}\le 2\gamma \ ,$$
		and as a result it holds for all $k$:
		$$
		%\|B_k^{-1}\|\le \max\{B_0^{-1}, B_1^{-1}, ..., B_K^{-1}, 2\gamma\} :=\nu_1 \ ,
		\|B_k^{-1}\|\le \max\Big\{ \|B_0^{-1}\|, \|B_1^{-1}\|, ..., \|B_{K-1}^{-1}\|, 2\gamma\Big\} :=\nu_1 \ .
		$$
	    \end{proof}
	\end{lemma}
	%%%%%%%%%%%%%%%%%%%%%%%%%%%%%%%%%%%%%%%%%%%%%%%%%%%%%%%%%%%%%%%%%
	The next Lemma provides a lower bound on the amount of predicted decrease:
	\begin{lemma}\label{pBbound_lem}
		%If $\{z_k\}$ converges to  $z^*$ and $\{s_k\}$ to 0, 
		Under the assumptions stated in Theorem \ref{B_conv} part (1), it holds for any $k$ that
		%\beq \label{pred2}
		\begin{equation}\label{eq:car1}
			m_k(0) - m_k(s_k) \geq \frac{\|s_k^2\|}{2}\min \left\{ \frac{1}{\nu_1^2} ~,~ \frac{1}{\nu_1^4\nu_2^2}\right\} \ .
		\end{equation}
	\end{lemma}
		\begin{proof}
		%\proof{Proof.}
			It follows from \eqref{eq:pBk} and Lemma \ref{BInv_cor} that
			$ \|p^B_k\| \leq \nu_1^2\|g_k\|$. Furthermore, recall that in Algorithm  \ref{tr_alg} we set $s_k=p^B_k$ if $\|p^B_k\|\le \Delta_k$ and otherwise we have $ \|s_k\| = \Delta_k$, thus in either case we have $\|s_k\| \leq \|p^B_k\|$.
			Hence, it holds that
			$ \|g_k\|/\|s_k\|\ge 1/\nu_1^2$. Then, it follows from Proposition \ref{cauchy_armijo} that
			\begin{align*}
				m_k(0) - m_k(s_k) &\geq \frac{\|g_k\|}{2} \min \Big\{ \Delta_k ~,~ \frac{\|g_k\|}{\nu_2^2} \Big\} \\
				&= \|s_k\|\frac{\|g_k\|}{2} \min \Big\{ \frac{\Delta_k}{\|s_k\|} ~,~ \frac{\|g_k\|}{\|s_k\|\nu_2^2} \Big\} \nonumber\\
				&\geq \|s_k\|\frac{\|g_k\|}{2}\min \Big\{ 1 ~,~ \frac{1}{\nu_1^2\nu_2^2} \Big\}\nonumber\\
				&\geq \frac{\|s_k\|^2}{2}\min \Big\{ \frac{1}{\nu_1^2} ~,~ \frac{1}{\nu_1^4\nu_2^2} \Big\} \ ,
			\end{align*}
			%\end{align}
			where
			%comes from applying \eqref{eqB_bd}, 
			the second inequality is from $\Delta_k/\|s_k\|\geq 1$ and $\|g_k\|/\|s_k\| \geq 1/\nu_1^2$, and the third inequality uses $\|g_k\|/\|s_k\| \geq 1/\nu_1^2$ again.
		\end{proof}
		%\endproof
	
	%%%%%%%%%%%%%%%%%%%%%%%%%%%%%%%%%%%%%%%%%%%%%%%%%%%%%%%%%%%%%%%%%%%%
	%                       proof of Thm 4.3 (part2)
	%%%%%%%%%%%%%%%%%%%%%%%%%%%%%%%%%%%%%%%%%%%%%%%%%%%%%%%%%%%%%%%%%%%%
	\begin{proof}[Proof of Theorem \ref{B_conv} part (2)]
		We give proofs for the following two claims:\\
		\textbf{(a). } There exists a constant $K$ such that it holds for all $k\ge K$ that $\|p^B_k\| \leq \Delta_k$, so $s_k = p^B_k$ and $z_{k+1}=z_k+s_k$. This shows that after a finite number of steps, we always take quasi-Newton step, and the quasi-Newton step is a valid step.\\
		\textbf{(b). } The rate of the convergence of Algorithm \ref{tr_alg} is R-superlinear, that is, $ \lim_{k \to \infty} \|z_k-z^*\|^{1/k}=0$.\\
		%%%%%%%%%%%%%%%%%%%%%%%%%%%%%%%%%%%%%%%%%
		%    proof of (a) in part (2)
		%%%%%%%%%%%%%%%%%%%%%%%%%%%%%%%%%%%%%%%%%
		To see \textbf{(a).} we begin by defining 
		%\beq\label{eta_k}
		$$ \eta_k = \sup_{0\leq t \leq 1} \|~\nabla F(z_k+ts_k) -B_k\| \ . $$
		%\eeq
		Since  $z_k\rightarrow z^*$, $s_k\rightarrow 0$   and as we showed in part (1), $ \|B_k -\nabla F(z^*) \|  \to 0$,  as $k \to \infty$,    and  since $\nabla F(z)$ is  continuous around $z^*$, we conclude 
		%\begin{equation}\label{eq:etak0}
		$ \eta_k \rightarrow 0 \ . $
		%\end{equation}
		%by recalling the definition of $\delta_k$ and noticing $\delta_k\rightarrow 0$, we have
			
		Consider now the one-dimensional function $f(t)=\|F(z_k+t s_k)\|^2/2$, and then by second-order mean value theorem we know there exists $0 < t < 1$  such that 
		\begin{align}\label{tay_2nd}
			\begin{split}
				\|F(z_k+s_k)\|^2/2 = & \|F(z_k)\|^2/2 + F(z_k)^T\nabla F(z_k)s_k  +\\
				&   s_k^T\left( \nabla F(z_k+ts_k)^T\nabla F(z_k+ts_k)  \right) s_k/2 + \nabla^2 F(z_k+ts_k)\left(F(z_k+ts_k)/2, s_k, s_k\right) \ ,  %\nonumber 
			\end{split}
		\end{align}
		where $\nabla^2 F(z_k+ts_k)$ is a 3-dimensional tensor and $F(z_k+ts_k)\left(F(z_k+ts_k), s_k, s_k\right)$ refers to the tensor-vector product. Furthermore, notice that $F(z)$ is $\gamma_1$-Lipschitz and $\nabla F(z)$ is $\gamma_2$-Lipschitz, thus 
		%\begin{align}\label{eq:pile}
			$$\nabla^2 F(z_k+ts_k)\Big(F(z_k+ts_k), s_k, s_k\Big)/2 \leq \gamma_2\|F(z_k+ts_k)\|\|s_k\|^2/2  \leq  \gamma_2\gamma_1\left(\|z_k-z_*\|+\|s_k\|\right) \|s_k\|^2/2  \ . $$
		%\end{align}
		Substituting this inequality
		%\eqref{eq:pile} 
		to \eqref{tay_2nd} and recalling $g_k^T = F(z_k)^T\nabla F(z_k)$ we obtain
		$$\|F(z_k+s_k)\|^2/2  \leq  \|F(z_k)\|^2/2 + g_k^Ts_k  + s_k^T\nabla F(z_k+ts_k)^T\nabla F(z_k+ts_k)s_k/2 + \gamma_2\gamma_1\left(\|z_k-z_*\|+\|s_k\|\right) \|s_k\|^2/2.$$
		
	    Denote by $A_k=\nabla F(z_k+ts_k)  -B_k$ and clearly $\|A_k\|\le \eta_k$.  It holds by recalling the definition of $m_k(s_k)$ that 
		\begin{align*} %\label{car2}
			\begin{split}
				&\|F(z_k+s_k)\|^2/2 -m_k(s_k)  \\
				\leq &   s_k^T\nabla F(z_k+ts_k)^T\nabla F(z_k+ts_k)s_k/2 
				+\gamma_2\gamma_1\left(\|z_k-z_*\|+\|s_k\|\right) \|s_k\|^2/2 
				- s_k^TB_k^TB_ks_k/2
				\\
				= & s_k^T \Big( (B_k+ A_k)^T(B_k+ A_k) \Big) s_k/2 
				+\gamma_2\gamma_1\left(\|z_k-z_*\|+\|s_k\|\right) \|s_k\|^2/2 -s_k^TB_k^TB_ks_k/2   \\
				= & s_k^T \Big( B_k^TA_k + A_k^TB_k+ A_k^TA_k \Big) s_k/2 
				+\gamma_2\gamma_1\left(\|z_k-z_*\|+\|s_k\|\right) \|s_k\|^2/2  \\
				\leq & \Big(2\eta_k\nu_2 + \eta_k^2+  \gamma_2\gamma_1\left(\|z_k-z_*\|+\|s_k\|\right)\Big)\|s_k\|^2/2 \ , 
			\end{split}
		\end{align*}
		where the final inequality uses  $\|B_k\| \leq \nu_2$ and $\|A_k\|\le \eta_k$.
		%Recall that by definition $m_k(0) = \|F(z_k)\|^2/2$. Hence from the definition of $\rho_k$ we obtain:
		Recalling the definition of $\rho_k$ and combining \eqref{eq:car1} and this inequality
		%\eqref{car2},
		 we arrive at
		%\begin{align}\label{eq:car3}
		$$	
		    |1-\rho_k| = \frac{|\|F(z_k+s_k)\|^2/2 -m_k(s_k)| }{m_k(0) - m_k(s_k)}   \leq
			\frac{2\eta_k\nu_2 + \eta_k^2+  \gamma_2\gamma_1\left(\|z_k-z_*\|+\|s_k\|\right) }{\min \Big\{ \frac{1}{\nu_1^2} ~,~ \frac{1}{\nu_1^4\nu_2^2} \Big\}} \ .
		$$
		%\end{align}
		We have  $z_k\rightarrow z^*$,  $s_k\rightarrow 0$  and $\eta_k\rightarrow 0$ as  $k\rightarrow\infty$. Thus, in the R.H.S. of %\eqref{eq:car3}
		the above inequality the numerator 
		goes to $0$ and the denominator  is a constant. Therefore, there exists $K_1$ such that 
		%$|1-\rho_k|\le 1-\zeta_1$, 
		$|1-\rho_k|\le 0.5$, 
		thus 
		%$\rho_k\ge\zeta_1$ 
		$\rho_k > 0.5$ 
		for all $k\ge K_1$. This means that for $k\ge K_1$, we always expand the trust-region radius: $\Delta_{k+1} = \min\{2\Delta_k,R_0\}$. As such, there exists $K_2$ such that $\Delta_k=R_0$ for all $k\ge K_2$. Furthermore, it follows from Lemma \ref{BInv_cor} that
		$ \|p^B_k\| \leq \nu_1^2\|g_k\|$, thus $\|p^B_k\|\rightarrow 0$ as $k\rightarrow\infty$, whereby there exists $K_3$ such that $\|p^B_k\|\le R_0$ for all $k\ge K_3$. Let $K=\max\{K_2, K_3\}$, then we have $\|p^B_k\|\le R_0= \Delta_k$ for all $k\ge K$. This finishes the proof  of \textbf{(a).} and we conclude eventually all steps are valid and they are quasi-Newton steps.
		
		%%%%%%%%%%%%%%%%%%%%%%%%%%%%%%%%%%%%%%%%%
		%    proof of (b) of part (2)
		%%%%%%%%%%%%%%%%%%%%%%%%%%%%%%%%%%%%%%%%%			
		To see \textbf{(b).} notice that R-superlinear convergence studies the eventual behavior of the algorithm as the iteration count $k\rightarrow\infty$. It follows from part \textbf{(a).} that eventually (i.e., when $k\ge K$) we always take quasi-Newton step, i.e., $s_k=p_k^B$ and the step is a valid step, i.e., $z_{k+1}=z_k+s_k=z_k+p_k^B$. It then follows from the definition of $p_k^B$ and $g_k$ that $s_k=p_k^B=-B_k^{-1}B_k^{-T}\nabla F(z_k)^TF(z_k)$. 
		Reusing the notation, denote by $A_k =\nabla F(z_k) - B_k $  and clearly  $\|A_k\|\le \eta_k$, further let $N_k = B_{k+1} - B_k $, so, $\|N_k\| \to 0$  as $k\rightarrow \infty$. We have
		\begin{align*}
			\|F(z_{k+1}) \| &= \|F(z_k) + B_{k+1}s_k\|\\
			&=\|F(z_k)  -B_{k+1}B_k^{-1}B_k^{-T}\nabla F(z_k)^TF(z_k)\|\\
			&\leq \|I -B_{k+1}B_k^{-1}B_k^{-T}\nabla F(z_k)^T\|\|F(z_k)\|\\
			&= \left\| I - (B_k+N_k)B_k^{-1}\Big((B_k+A_k)B_k^{-1}\Big)^T \right\|\|F(z_k)\|\\
			&= \| I - (I+N_kB_k^{-1})(I+A_kB_k^{-1})^T \|\|F(z_k)\|\\
			&\leq  \left(\|B_k^{-1}\|\|N_k\| +\|B_k^{-1}\|\|A_k\| + \|N_k\| \|A_k\| \| B_k^{-1} \|^2\right) \|F(z_k)\|\\
			& \leq \nu_1 \left(\|N_k\| + \eta_k + \|N_k\|\eta_k \nu_1 \right) \|F(z_k)\| \ ,
		\end{align*}
		where the first equality comes from the secant condition \eqref{sec}, the first and the second inequalities utilize Cauchy-Schwarz inequality, and the last inequality uses $\|B_k^{-1}\| \le \nu_1$ and  $\|A_k\|\le \eta_k$.
		%from Lemma \ref{BInv_cor}.			
		%Notice that it follows from Theorem  \ref{B_conv} that both  $\{\eta_k\}$ and $\{\|N_k\|\}$ converge to 0 as $k\rightarrow \infty$. 
		Since both  $\eta_k \to 0$ and  $\|N_k\| \to 0$ as $k\rightarrow \infty$, then, for any arbitrary $0<\epsilon<1$, there exists iteration $\bar K \geq K$ such that $\nu_1(\|N_k\| + \eta_k + \|N_k\|\eta_k \nu_1) \leq \epsilon$ for all $k\geq \bar K$.
		Therefore, it holds for $k\geq \bar K$ that
		$  \|F(z_{k}) \| \leq \epsilon^{k-\bar K} \|F(z_{\bar K})\|$, and
		$$\lim_{k \to \infty} \|F(z_{k}) \|^{1/k} \le\lim_{k \to \infty} \left(\epsilon^{k-\bar K} \|F(z_{\bar K})\|\right)^{1/k} = \lim_{k \to \infty}  \epsilon \Big(\frac{\|F(z_{\bar K})\| }{\epsilon^{\bar K}} \Big)^{1/k} = \epsilon\ . $$
		Notice that the above inequality holds for any $0<\epsilon<1$. Together with $  \|F(z_{k}) \|\ge 0$, we conclude that $\lim_{k \to \infty} \|F(z_{k}) \|^{1/k}  =0$. Furthermore, recall that 
		%that $z_k\rightarrow z^*$, 
		$F(z^*)=0$,  $\nabla F(z^*)$ is nonsingular and $\nabla F(z)$ is continuous around $z^*$, 
            thus we can obtain $ \lim_{k \to \infty} \|z_k-z^*\|^{1/k}=0$ as follows.

            By Taylor's expansion, we have
            \[
                F(z_k)=F(z_k)-F(z^*)=\int_0^1 \nabla F(z^*+t(z_k-z^*))dt\cdot(z_k-z^*) \ .
            \]
            Note that for large enough $k$, $\nabla F(z^*+t(z_k-z^*))$ is nonsingular and thus 
            \[
                z_k-z^*=\left( \int_0^1 \nabla F(z^*+t(z_k-z^*))dt \right)^{-1}F(z_k) \ .
            \]
            Recalling the boundedness of $\left\Vert\left( \int_0^1 \nabla F(z^*+t(z_k-z^*))dt \right)^{-1}\right\Vert$ for large enough $k$ we achieve
            \[
                \lim_{k\rightarrow \infty}\| z_k-z^* \|^{1/k} \leq \lim_{k\rightarrow \infty}\left\Vert\left( \int_0^1 \nabla F(z^*+t(z_k-z^*))dt \right)^{-1}\right\Vert^{1/k} \| F(z_k)\|^{1/k} = 0 \ .
            \]
            This finishes the proof.
        
	\end{proof}
\section{Numerical Experiments}\label{num_exp}

{In this section, we present numerical experiments of $J$-symmetric quasi-Newton algorithms and compare them with classical algorithms for minimax problems. We perform the experiments on four sets of minimax problems: quadratic convex-concave minimax problems, two player bilinear zero-sum game, analytic center, and a nonconvex-nonconcave minimax problem. The source code is available at {\url{https://github.com/azamasl/Jsymm}}.

% \subsection{Implementation Details}
We compare the behaviors of the following five algorithms:
\begin{itemize}
	\item EGM: The extra-gradient algorithm \cite{korp,nemirovski2004prox} with fixed stepsize; 
	\item Broyden: Broyden's good method~\cite{broyden1965class,Broyden_single} with a fixed stepsize;
	\item J-symm: $J$-symmetric quasi-Newton Algorithm (Algorithm \ref{alg0}) with a fixed stepsize;
	\item J-symm-LS: $J$-symmetric quasi-Newton Algorithm with line-search (Algorithm \ref{algls});
	\item J-symm-Tr: $J$-symmetric quasi-Newton Algorithm with trust-region (Algorithm \ref{tr_alg}).
\end{itemize}

For all quasi-Newton methods, we initialize the Jacobian estimation $H_0=I$ for J-symm, J-symm-LS, and J-symm-Tr. Notice that $H_0=I$ would often introduce numerical issue for Broyden's method, thus we initialize the Jacobian estimation $H_0$ for Broyden's method as a diagonal matrix with each entry coming from an uniform distribution $U(0,1)$. 

The step-size plays an important rule in the behaviors for EGM, Broyden and J-symm. For the instances in Section \ref{sec:bilinear} and Section \ref{sec:analytic}, we choose the best step-size from $\{0.0008,0.004,0.008,0.04,0.08\}$ in performance. The reason to select a small step-size is because these real instances are often ill-conditioned, and a large stepsize may quickly blow up the solutions due to the bad initial estimation of the Jacobian. For the synthetic quadratic minimax problem in Section \ref{quadratic}, we start with stepsize 0.01, and then stepsize 1 after $\| F(z) \|$ becomes small. For the nonconvex example in Section \ref{sec:nonconvex}, we use a constant step-size 0.01 to showcase the trajectory of the algorithm.

}

%%%%%%%%%%%%%%%%%%%%%%%%%%%%%%%%%%%%%%%%%%%%%%%%%%%%%
%     Bilinear and 3 Strongly CC  experiments
%%%%%%%%%%%%%%%%%%%%%%%%%%%%%%%%%%%%%%%%%%%%%%%%%%%%%
\subsection{Quadratic Minimax Problem}\label{quadratic}
We here consider quadratic convex-concave minimax problems of the form
\beq\label{L_q}%sp_prob
L(x,w)= \dfrac{1}{2}(x-x^*)^TD(x-x^*)+(w-w^*)^TA(x-x^*)-\dfrac{1}{2}(w-w^*)^TC(w-w^*) \ ,
\eeq
where $C$ and $D$ are positive semidefinite matrices. \textcolor{black}{Notice that any function around its optimal solution $(x^*,y^*)$ behaves similarly to a quadratic minimax problem~\eqref{L_q} due to Taylor expansion. Thus, the behaviors of different algorithms on quadratic functions showcase the asymptotical behaviors for general problems.}

In this experiment, we generate synthetic data with different strong convexity value. More specifically, we choose $D\ \in \R^{500\times 500 }$,  $C\ \in \R^{500\times 500 }$ and  $A\in \R^{500\times 500 }$ in the following way. The entries of  $A$  are drawn  randomly from a normal distribution $\N(0, 1/\sqrt{500})$. To generate a random positive definite matrix $D$, we first create  a random matrix $S\in\R^{500\times 500 }$ with entries  drawn  from  $\N(0, 1/\sqrt{500})$, and then we symmetrise the matrix by $S = (S+S^T)/2$. Next, we shift the matrix using a scaled identity matrix to make it positive definite $S=S+(|\lambda_{\min}|+1)I$, where $\lambda_{\min}$ is the minimal eigenvalue of $S$ (which is usually negative), and the identity matrix guarantees the $1$-strong-convexity-strong-concavity of the matrix $S$. Finally, we set $D= \alpha S$, with $\alpha$ taking values from $\{0, 10^{-4}, 10^{-2}, 1\}$. We set the matrix $C$ by the same procedure and with a different random seed. The value $\alpha$ measures the scale ratio between the diagonal terms and the off-diagonal terms, and it turns out to be the critical parameter to characterize the performance of different algorithms. When $\alpha=0$, the problem is a bilinear convex-concave minimax problem. When $\alpha>0$, the problem is strongly-convex-strongly-concave. 
% \red{How do you generate $x^*$, $x^0$?}

% We compare the behaviors of the following five algorithms:
% \begin{itemize}
% 	\item EGM: The extra-gradient algorithm \cite{korp,nemirovski2004prox}  with stepsize $1/||\nabla F(z)||$; 
% 	\item Broyden: Broyden's good method~\cite{broyden1965class,Broyden_single} with a fixed stepsize $0.01$ first and stepsize $1$ after $\|F(z)\|\le 0.1$;
% 	\item J-symm: $J$-symmetric quasi-Newton Algorithm (Algorithm \ref{alg0}) with a fixed stepsize $0.01$ first and stepsize $1$ after $\|F(z)\|\le 0.1$;
% 	\item J-symm-LS: $J$-symmetric quasi-Newton Algorithm with line-search (Algorithm \ref{algls});
% 	\item J-symm-Tr: $J$-symmetric quasi-Newton Algorithm with trust-region (Algorithm \ref{tr_alg}).
% \end{itemize}

% We initialize the Jacobian estimation $H_0=I$ for J-symm, J-symm-LS, and J-symm-Tr. Notice that $H_0=I$ would make the first step of Broyden's method go to infinity, thus we initialize the Jacobian estimation $H_0$ for Broyden's method as a diagonal matrix with each entry coming from uniform distribution $U(0,1)$. For Broyden's method and J-symm, we start with stepsize $0.01$ first, and then $1$ after $\|F(z)\|\le 0.1$. The reason is because a large initial stepsize (i.e., $1$) for both methods quickly blows up the solutions due to the bad initial estimation of the Jacobian.

Figure \ref{fig:cc} plots $\|F(z)\|$ on logarithmic scale versus the number of iterations for the five algorithms and different $\alpha$ values $\alpha\in\{0,10^{-4},10^{-2},1\}$. Some observations in sequence: First, we can clearly see the advantage of J-symm methods compared to Broyden's methods, in particular for those instances with small $\alpha$. Second, as the value of $\alpha$ increases, all five methods have better performance.  Third, J-symm-LS (magenta line) has the best performance in this set of experiments because of its adaptive step-size choice. Finally, J-symm-Tr (blue line) turns out to be too conservative, in particular for the case with small $\alpha$ value.

\begin{figure}[h]
	\centering
	\begin{tabular}{l c c c c}
		\hspace{-1Cm}
		& \includegraphics[width=0.45\textwidth]{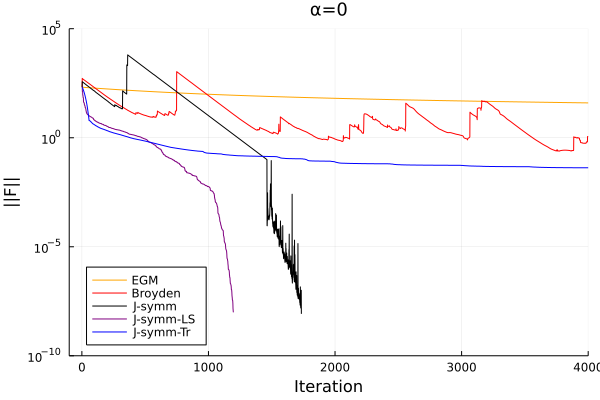}
		\hspace{1Cm}
		& \includegraphics[width=0.45\textwidth]{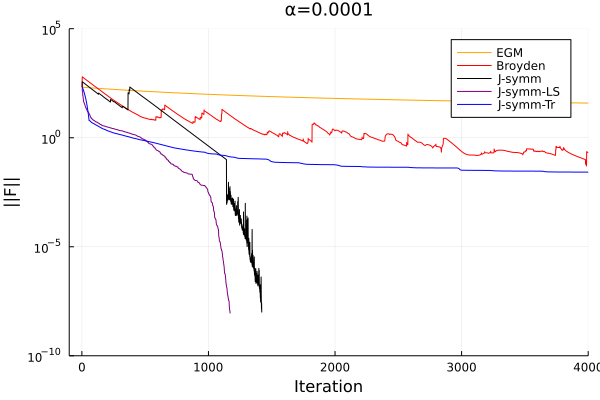}\\
		\hspace{-2Cm} & & & &\\
		\hspace{-1Cm}
		& \includegraphics[width=0.45\textwidth]{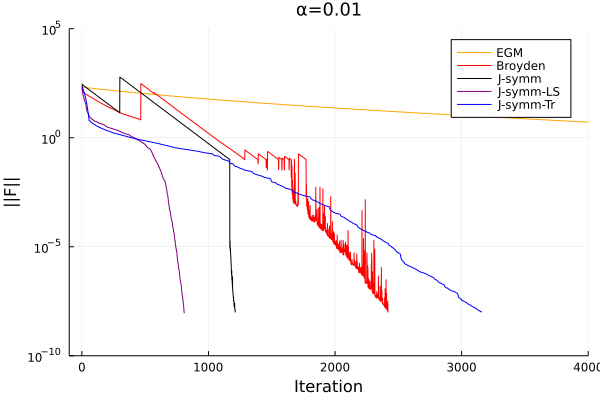}
		\hspace{1Cm}
		& \includegraphics[width=0.45\textwidth]{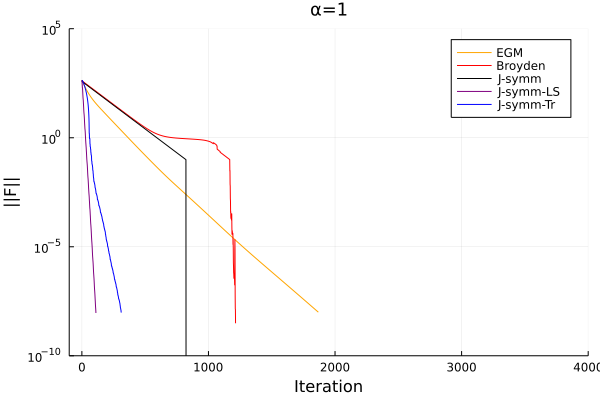}
	\end{tabular}
	\caption{Plots showing $\|F(z)\|$ in log scale versus the number of iterations of EGM (yellow), Broyden's method (red), $J$-symmetric method with fixed stepsize (black), $J$-symmetric method with line-search (magenta) and $J$-symmetric method with trust-region (blue) for solving the quadratic convex-concave problems \eqref{L_q}. The four figures are with $\alpha=0$ (top-left), $\alpha=10^{-4}$ (top-right), $\alpha=10^{-2}$ (bottom-left) and $\alpha=1$ (bottom-right), respectively.}
	\label{fig:cc}
\end{figure}

%%%%%%%%%%%%%%%%%%%%%%%%%%%%%%%%%%%%%%%%%%%%%%%%%%%%%
%     zero-sum Bilinear experiments
%%%%%%%%%%%%%%%%%%%%%%%%%%%%%%%%%%%%%%%%%%%%%%%%%%%%%
{
\subsection{Bilinear Zero-sum Games}\label{sec:bilinear}
Recently, solving bilinear zero-sum games has attracted a great deal of attention in both the optimization community and the machine learning community, as the first step in understanding more complicated applications~\cite{arjovsky2017wasserstein,gidel2018variational,gidel2019negative,mescheder2017numerics,lu2022s}. Surprisingly, the most natural algorithm, gradient descent-ascent (GDA), does not converge and hence many first-order methods (FOMs) tailored for minimax problem have been proposed \cite{du2019linear,liang2019interaction,daskalakis2018training}. In this subsection, we aim at demonstrating that J-symmetric quasi-Newton has competitive performance with other methods on bilinear zero-sum games.
% The bilinear zero-sum game comprises an important class of problems in the minimax optimization. Despite its simple formulation, it captures the difficulty of many minimax problems in machine learning such as generative adversarial networks (GANs) \cite{arjovsky2017wasserstein,gidel2018variational,gidel2019negative,mescheder2017numerics}.

Bilinear zero-sum games can be formulated as the following minimax optimization problem:
\begin{equation}\label{bilinear}
    \min_{x\in\R^n}\max_{y\in\R^m} y^TAx \ .
\end{equation}
The set of all saddle points is 
\begin{equation*}
    \{(x,y): Ax=0,\; A^Ty=0  \} \ .
\end{equation*}

}

{
In this experiment, we use the constraint matrices of the root-node LP relaxation from \texttt{MIPLIB}. The number of iterations of different methods to find a solution with $\|F(z)\|\le 10^{-4}$ are presented Table \ref{table:bilinear}. We can clearly see that J-symm and J-symm-LS significantly outperform Broyden's method and EGM. %The dimension of \texttt{enlight\_hard} is $100\times 200$ while \texttt{22433} is $198\times 429$.
}

% \textcolor{blue}{
% We compare the behaviours of the following five methods as in Subsection \ref{quadratic}:
% \begin{itemize}
% 	\item EGM: The extra-gradient algorithm \cite{korp,nemirovski2004prox}  with stepsize $1/||\nabla F(z)||$; 
% 	\item Broyden: Broyden's good method~\cite{broyden1965class,Broyden_single} with a fixed stepsize;
% 	\item J-symm: $J$-symmetric quasi-Newton Algorithm (Algorithm \ref{alg0}) with a fixed stepsize;
% 	\item J-symm-LS: $J$-symmetric quasi-Newton Algorithm with line-search (Algorithm \ref{algls});
% 	\item J-symm-Tr: $J$-symmetric quasi-Newton Algorithm with trust-region (Algorithm \ref{tr_alg}).
% \end{itemize}
% }

% \textcolor{blue}{
% The stepsizes for Broyden and J-symm are tuned from $\{0.0008,0.004,0.008,0.04,0.08\}$. The initial Jacobian estimation for J-symm, J-symm-LS and J-symm-Tr is set as identity, i.e., $H_0=I$ and we initialize $H_0$ for Broyden's method as a diagonal matrix whose entries are sampled from uniform distribution $U(0,1)$.
% }

{
Figure \ref{fig:bilinear} presents $\Vert F(z)\Vert$ on a logarithmic scale versus the number of iterations for five different algorithms on two real instances, \texttt{enlight\_hard} and \texttt{22433}, to showcase the typical behaviors of different algorithms. As we can see, J-symm-LS (magenta line) exhibits much faster convergence than the other four methods. Furthermore, we can observe the superlinear convergence of Broyden (red line) and J-symm (black) with fixed stepsize. J-symm-Tr (blue line) sometimes can be too conservative.
}

\begin{figure}[h]
	\centering
	\begin{tabular}{l c c c c}
		\hspace{-1Cm}
		& \includegraphics[width=0.45\textwidth]{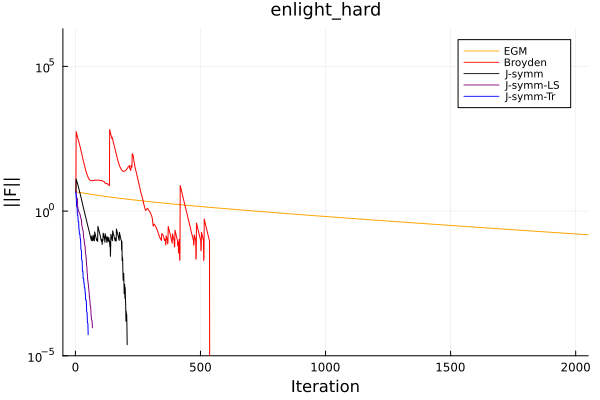}
		\hspace{1Cm}
		& \includegraphics[width=0.45\textwidth]{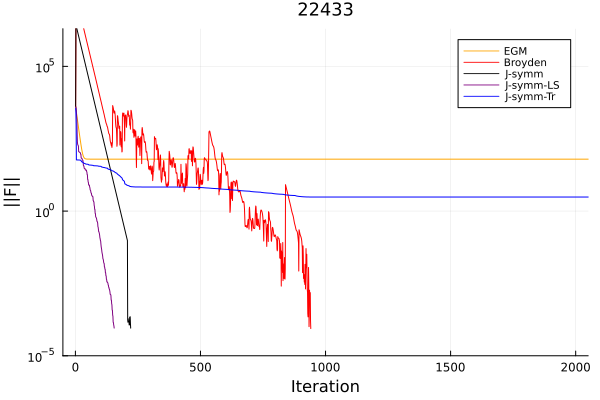}\\
		\hspace{-1Cm}
	\end{tabular}
	\caption{Plots showing $\|F(z)\|$ in log scale versus the number of iterations of EGM (yellow), Broyden's method (red), $J$-symmetric method with fixed stepsize (black), $J$-symmetric method with line-search (magenta) and $J$-symmetric method with trust-region (blue) for solving the bilinear zero-sum problem \eqref{bilinear}. }
	\label{fig:bilinear}
\end{figure}
{\color{blue}
\begin{table}
\centering
\begin{tabular}{ |c|c|c|c|c|c| } 
 \hline
 Instance & EGM & Broyden & J-symm & J-symm-LS & J-symm-Tr \\  \hline
 \texttt{22433} & - & 940 & 220 & 67 & 51 \\ \hline
 \texttt{23588} & - & 773 & 220 & 162 & -  \\
 \hline
 \texttt{assign1-5-8} & - & 1157 & 272 & 169 & 311  \\
 \hline
 \texttt{b-ball} & - & 231 & 87 & 24 & 65  \\
 \hline
 \texttt{enlight8} & - & 277 & 225 & 68 & 63  \\
 \hline
 \texttt{enlight9} & - & 327 & 178 & 77 & 58  \\
 \hline
 \texttt{enlight\_hard} & - & 536 & 206 & 67 & 51  \\
 \hline
 \texttt{gr4x6d} & - & 150 & 119 & 60 & 837  \\
 \hline
 \texttt{neos5} & - & 262 & 88 & 14 & 33\\
 \hline
 \texttt{prod1} & - & 1346 & 590 & 588 & 1356 \\
 \hline
 \texttt{prod2} & - & 1083 & 512 & 373 & - \\ \hline
 \texttt{ran13x13} & - & 736 & 402 & 241 & 390 \\
 \hline
 % \hline
 % Shifted Geometric Mean & 2001 & 522 & 220 & 99 & 220 \\ \hline
\end{tabular}
\caption{Number of iterations for different methods to find an approximate solution with $\|F(z)\|\le 10^{-4}$ for the bilinear zero sum game. ``-'' refers to the algorithm fails to terminate within $2000$ iterations.}
\label{table:bilinear}
\end{table}
}

%%%%%%%%%%%%%%%%%%%%%%%%%%%%%%%%%%%%%%%%%%%%%%%%%%%%%%%%%%%%%%%%
%       Analytic center of a set of inequalities
%%%%%%%%%%%%%%%%%%%%%%%%%%%%%%%%%%%%%%%%%%%%%%%%%%%%%%%%%%%%%%%%%
\subsection{Analytic Center of Polytope}\label{sec:analytic}
{Analytic center is one way to define the geometric center of a polytope, and it has numerous applications, for example, in barrier methods \cite{nesterov1994interior}, cutting plane methods \cite{goffin1993computation,nesterov1995complexity,ye1996complexity,atkinson1995cutting} and MIP solvers \cite{berthold2018four}. }
Consider a polytope given by linear inequalities: 
$$ a_i^Tx \leq b_i \ , ~~~~~i=1,\hdots m \ ,$$
where $x \in \R^n$, $a_i \in \R^{n},$ and $b_i \in \R$ for $i=1,\hdots,m$.   The analytic center of the polytope is the minimizer of the following problem \cite[p. 141]{BV}:
\begin{align}\label{eq:ac}
	&\min_x ~~  -\sum^{m}_{i=1} \log (b_i -a_i^Tx) \ .
	%&\mathrm{subject~to}~~Ax=b,\nonumber   %A^Tw^{(0)}
\end{align}
The classical algorithm for finding the analytic center \eqref{eq:ac} is infeasible-start Newton method~\cite{GSM99}. Here we focus on quasi-Newton methods, which avoid linear equation solving and can be used for larger instances. Notice that it can be nontrivial to identify a feasible solution to \eqref{eq:ac}. We instead consider an equivalent formulation of \eqref{eq:ac}:
\begin{align*}
	\min_{x, y} & ~- \sum^{m}_{i=1} \log y_i \\
	\mathrm{s.t.} & ~y =b-Ax \ ,\nonumber
\end{align*}
where $A=[a_1,\ldots,a_m]^T$ and $b=[b_1,\ldots, b_m]$, and then we dualize the linear constraints to consider the minimax problem:
% Although the problem is unconstrained, solving the problem requires finding a point on the domain, that is the interior of the polyhedron defined by the inequalities.   A simple solution to go around finding a feasible starting point is  to apply infeasible start Newton method, as suggested by \cite{GSM99}. We make the implicit constraint $b-Ax \succ 0$ explicit by introducing a new variable $y$ and solve for:
% \begin{align*}
% 	&\min_{x, y} ~~ - \sum^{m}_{i=1} \log y_i \\
% 	&\mathrm{subject~to}~~y =b-Ax \ ,\nonumber
% \end{align*}
% where $A \in \R^{m\times n}, b\in \R^m$ are defined via $a_i$'s and $b_i$'s, respectively.  %We apply Newton method  to the optimality conditions of t
% The Lagrangian is 
\beq\label{ancent}
\min_{x,y}\max_{w} ~ L(x,y,w) = -\sum^{m}_{i=1} \log y_i + w^T(Ax-b+y) \ .
\eeq  
Now we have a minimax problem of the form \eqref{minimax}, and we can apply our $J$-symmetric methods.

{
In this experiment, we utilize the same \texttt{MIPLIB} instances as in the bilinear zero sum game. We do not perform J-symm-Tr, because of the inherent constraint of $y\ge 0$, which prevents an efficient trust-region solve. The number of iterations of different methods to find a solution with $\|F(z)\|\le 10^{-4}$ are presented Table \ref{table:ac}. Again, we can clearly see that J-symm and J-symm-LS significantly outperform Broyden's method and EGM.

% For quasi-Newton methods, we initialize the Hessian estimation $H_0=I$. We terminate the algorithms when $\|F(z)\|\le 10^{-4}$ (the value is chosen to avoid potential numerical issues caused by the log terms). Table \ref{table:ac} lists the number of iteration required to find the analytic center on 12 instances from \texttt{MIPLIB}. 
Figure \ref{fig:ac_new} presents $\Vert F(z)\Vert$ in log scale versus the number of iterations for five different algorithms on two real instances, \texttt{enlight\_hard} and \texttt{22433}, to showcase the typical behaviors of different algorithms.
% Figure \ref{fig:ac_new} compare the results of four algorithms on two of the instances \texttt{enlight\_hard} and \texttt{22433}. 
%The scale of these two instances are $100\times 200$ for \texttt{enlight\_hard} and $198\times 429$ for \texttt{22433}. 
For these two instances, we can observe that both J-symm-LS (magenta line) and J-symm (black line) exhibit fast convergence and an eventual superlinear rate on both instances. J-symm-LS stays in a plateau for a while to construct meaningful Jacobian estimation, and then enjoys a local superlinear convergence. For Broyden (red line), it converges on the smaller instance \texttt{enlight\_hard} but is numerically unstable on \texttt{22433}. 
}

\begin{figure}[h]
	\centering
	\begin{tabular}{l c c c c}
		\hspace{-1Cm}
		& \includegraphics[width=0.45\textwidth]{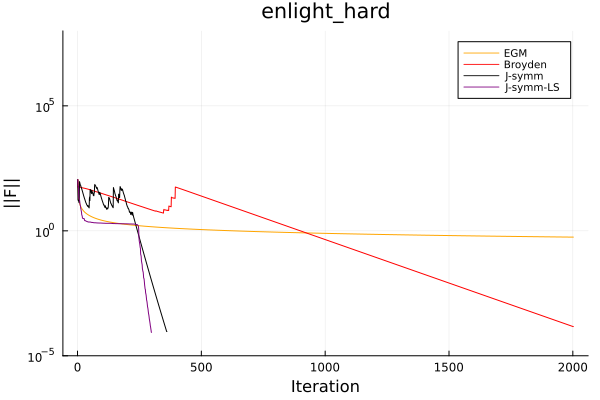}
		\hspace{1Cm}
		& \includegraphics[width=0.45\textwidth]{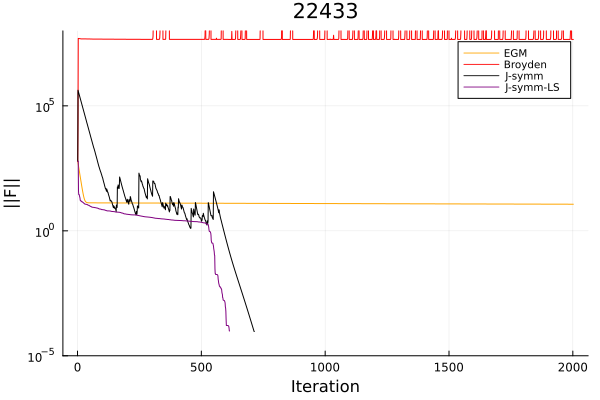}\\
		\hspace{-1Cm}
	\end{tabular}
	\caption{Plots showing $\|F(z)\|$ in log scale versus the number of iterations of Broyden's method (red), $J$-symmetric method with fixed stepsize (black) and $J$-symmetric method with line-search (magenta) for solving the analytic center problem \eqref{ancent}.}
	\label{fig:ac_new}
\end{figure}

\begin{table}
\centering
\begin{tabular}{ |c|c|c|c|c| } 
 \hline
 Instance & EGM & Broyden & J-symm & J-symm-LS \\  \hline
 \texttt{22433} & - & - & 712 & 612  \\ \hline
 \texttt{23588} & - & - & - & 449  \\
 \hline
 \texttt{assign1-5-8} & - & - & - & 1610  \\
 \hline
 \texttt{b-ball} & - & 450 & 374 & 214  \\
 \hline
 \texttt{enlight8} & - & - & 1566 & 409  \\
 \hline
 \texttt{enlight9} & - & 1951 & 1687 & 234  \\
 \hline
 \texttt{enlight\_hard} & - & - & 359 & 297  \\
 \hline
 \texttt{gr4x6d} & - & 1773 & 1889 & 124  \\
 \hline
 \texttt{neos5} & - & - & 1786 & 394 \\
 \hline
 \texttt{prod1} & - & - & - & 923 \\
 \hline
 \texttt{prod2} & - & - & - & 703 \\ \hline
 \texttt{ran13x13} & - & - & - & 395 \\
 \hline
 % \hline
 % Shifted Geometric Mean & 2001 & 1746 & 1318 & 425 \\
 % \hline
\end{tabular}
\caption{Number of iterations for different methods to find an approximated solution with $\|F(z)\|\le 10^{-4}$ for the analytic center problem. ``-'' refers to the algorithm fails to terminate within $2000$ iterations.}
\label{table:ac}
\end{table}

%%%%%%%%%%%%%%%%%%%%%%%%%%%%%%%%%%%%%%%
%%%%%%%%%%%%%%%%%%%%%%%%%%%%%%%%%%%%%%%
\subsection{A Nonconvex-Nonconcave Example}\label{sec:nonconvex}
\begin{figure}[]
	\centering
	\begin{tabular}{l c c c c}
		& & & &\\
		&{\sf{{~~~~A=1}}} 
		& {\sf{{~~~~A=10}}}
		& {\sf{{~~~~A=100}}}
		& {\sf{{~~~~A=1000}}}\\
		& & & &\\
		\rotatebox{90}{\sf{{~~~~~~~EGM}}}
		& \includegraphics[width=0.24\textwidth, trim =.5cm 0.6cm 1.1cm 1.3cm]{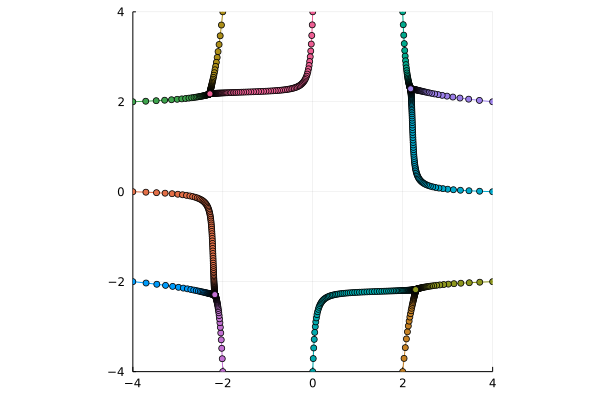}
		& \includegraphics[width=0.24\textwidth, trim =.5cm 0.6cm 1.1cm 1.3cm]{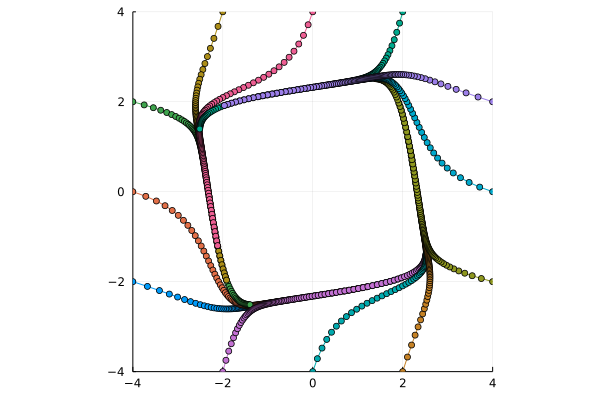}
		& \includegraphics[width=0.24\textwidth, trim =.5cm 0.6cm 1.1cm 1.3cm]{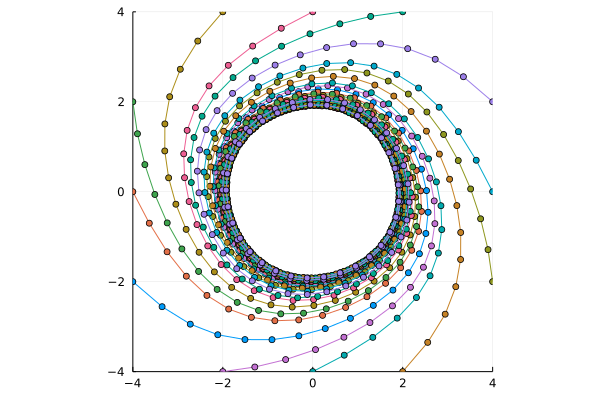}
		& \includegraphics[width=0.24\textwidth, trim =.5cm 0.6cm 1.1cm 1.3cm]{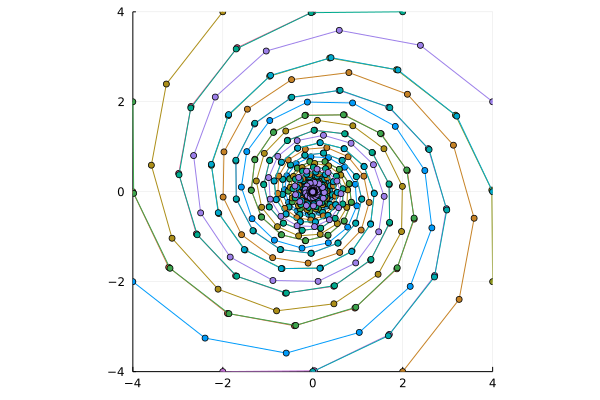}\\	
		& & & &\\
		%%%%%%%%%%%% Broyden
		\rotatebox{90}{\sf{{~~~~~~~Broyden}}}
		& \includegraphics[width=0.24\textwidth, trim =.5cm 0.6cm 1.1cm 1.3cm]{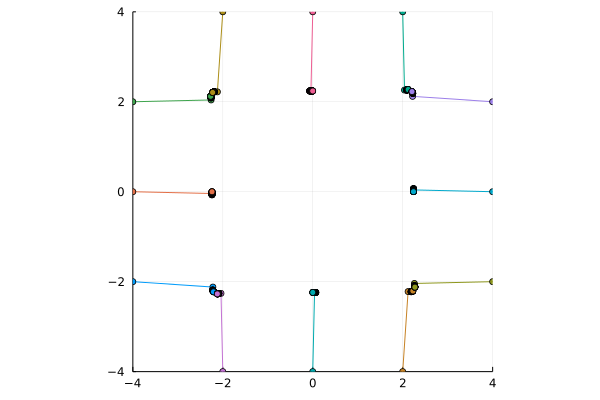}
		& \includegraphics[width=0.24\textwidth, trim =.5cm 0.6cm 1.1cm 1.3cm]{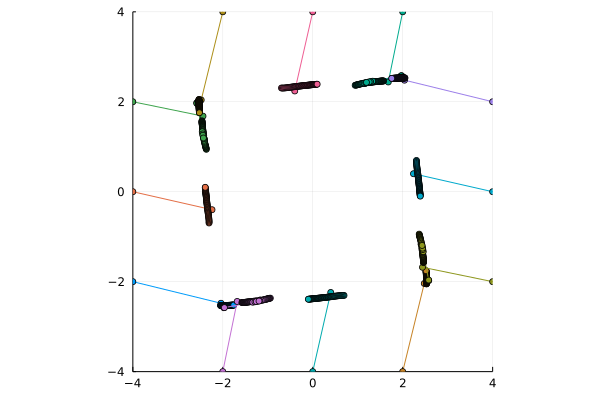}
		& \includegraphics[width=0.24\textwidth, trim =.5cm 0.6cm 1.1cm 1.3cm]{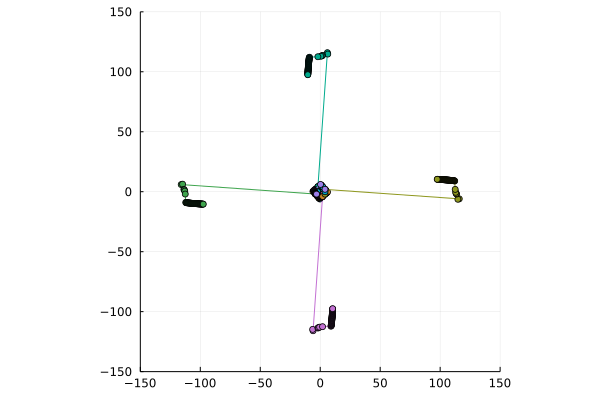}
		& \includegraphics[width=0.24\textwidth, trim =.5cm 0.6cm 1.1cm 1.3cm]{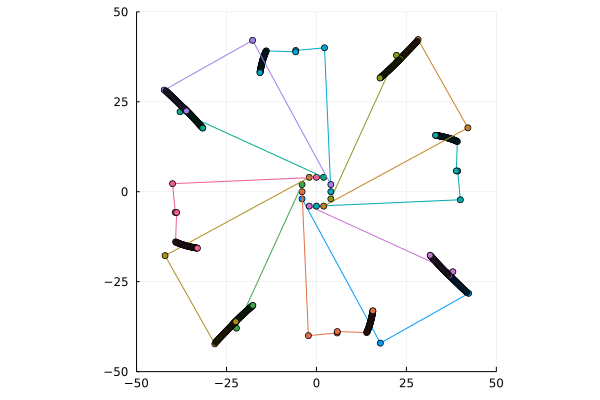}\\	
		& & & &\\
		%%%%%%%%% Fixed J-symm
		\rotatebox{90}{\sf{{~~~~~~~J-symm}}}
		& \includegraphics[width=0.24\textwidth, trim =.5cm 0.6cm 1.1cm 1.3cm]{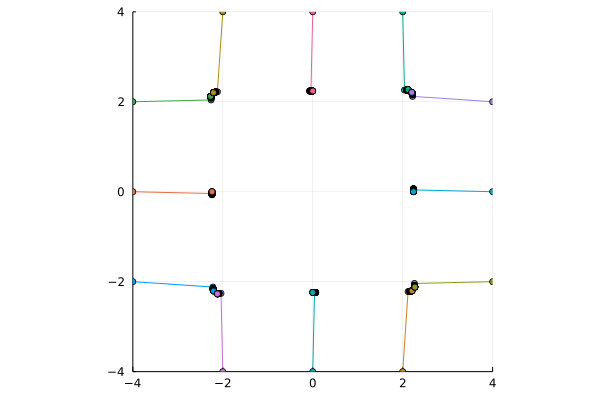}
		& \includegraphics[width=0.24\textwidth, trim =.5cm 0.6cm 1.1cm 1.3cm]{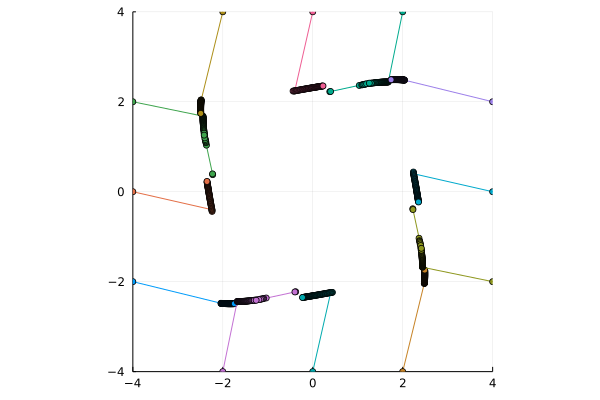}
		& \includegraphics[width=0.24\textwidth, trim =.5cm 0.6cm 1.1cm 1.3cm]{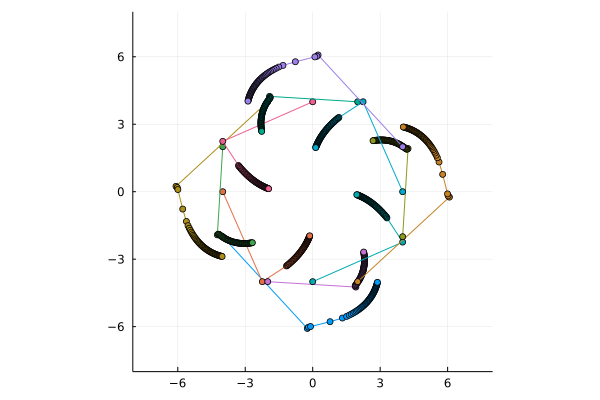}
		& \includegraphics[width=0.24\textwidth, trim =.5cm 0.6cm 1.1cm 1.3cm]{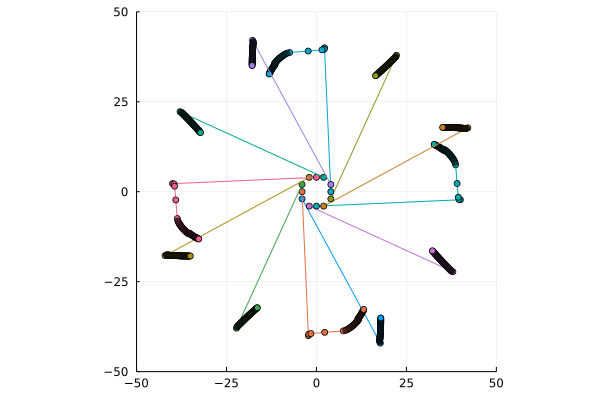}\\	
		& & & &\\
		%%%%%%%%% Ls J-symm
		\rotatebox{90}{\sf{{~~~~~J-symm-LS}}}
		& \includegraphics[width=0.24\textwidth, trim =.5cm 0.6cm 1.1cm 1.3cm]{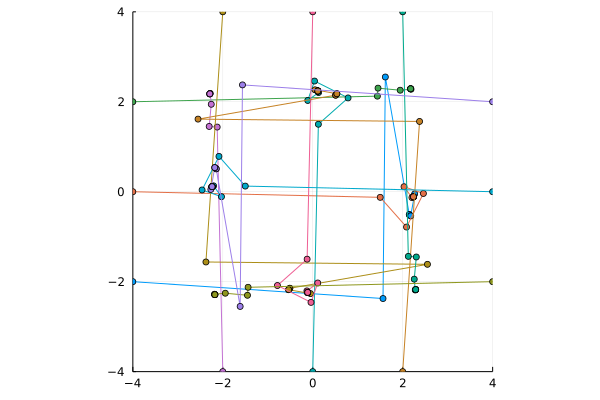}
		& \includegraphics[width=0.24\textwidth, trim =.5cm 0.6cm 1.1cm 1.3cm]{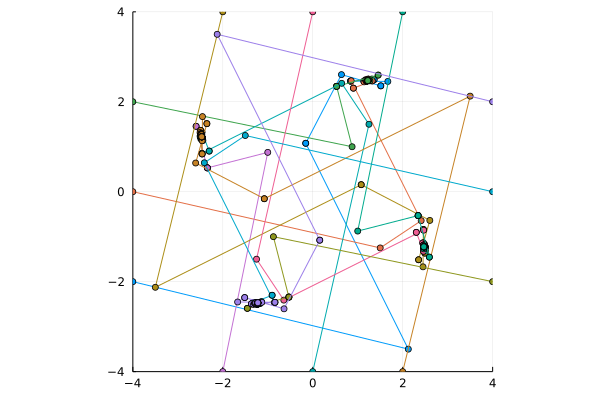}
		& \includegraphics[width=0.24\textwidth, trim =.5cm 0.6cm 1.1cm 1.3cm]{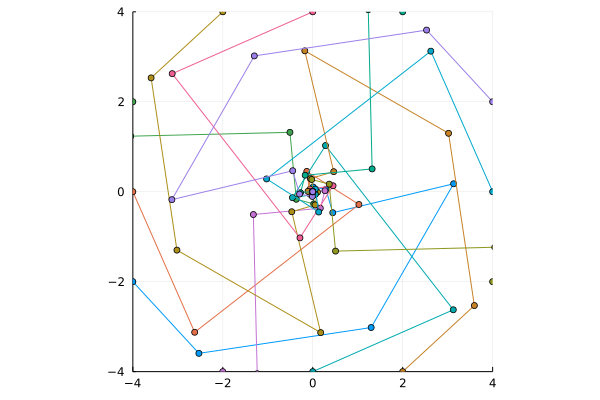}
		& \includegraphics[width=0.24\textwidth, trim =.5cm 0.6cm 1.1cm 1.3cm]{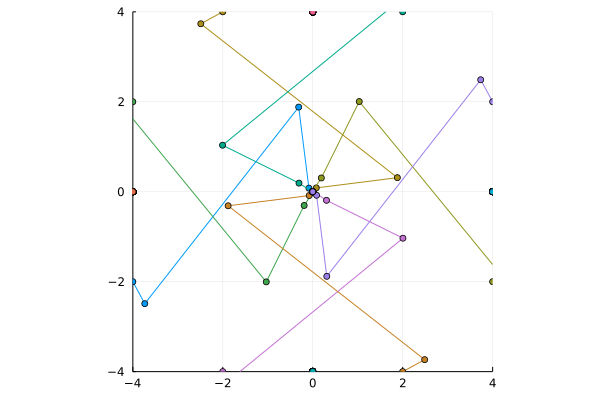}\\	
		& & & &\\
		%%%%%%%%% Tr
		\rotatebox{90}{\sf{{~~~~~J-symm-Tr}}}
		& \includegraphics[width=0.24\textwidth, trim =.5cm 0.6cm 1.1cm 1.3cm]{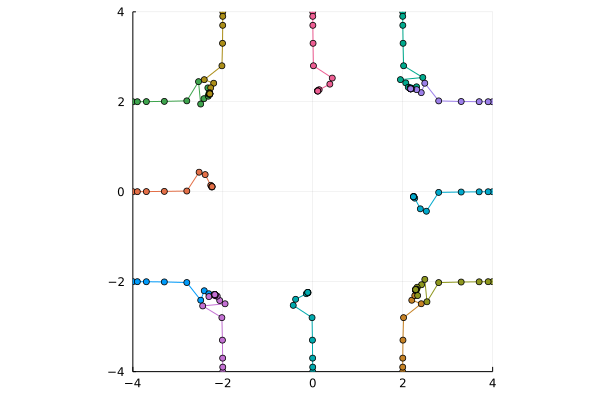}
		& \includegraphics[width=0.24\textwidth, trim =.5cm 0.6cm 1.1cm 1.3cm]{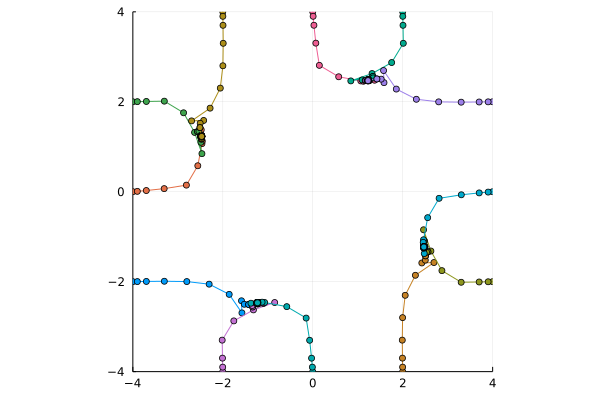}
		& \includegraphics[width=0.24\textwidth, trim =.5cm 0.6cm 1.1cm 1.3cm]{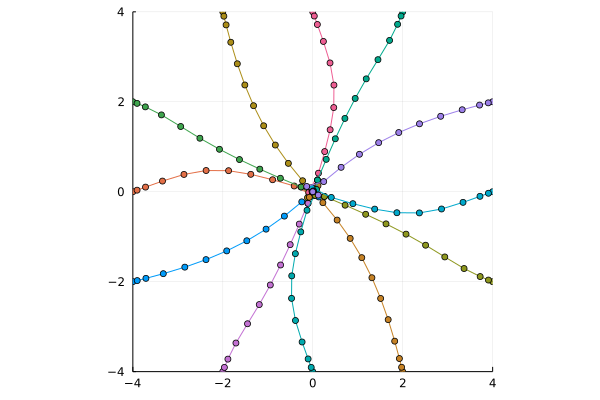}
		& \includegraphics[width=0.24\textwidth, trim =.5cm 0.6cm 1.1cm 1.3cm]{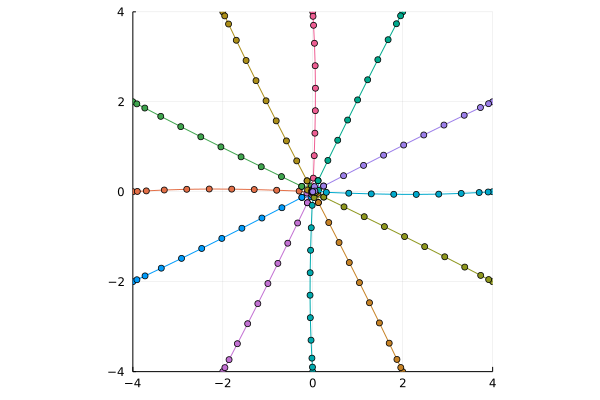}\\	
	\end{tabular}
	\caption{Plots showing the trajectories of five algorithms from twelve different initial solutions.}
	\label{fig:nonconvex}
\end{figure}
Many algorithms for minimax problems assume the problem to be convex-concave. For example, the proximal quasi-Newton methods proposed by Burke and Qian~\cite{BurkePPM1,BurkePPM2} only work for monotone operators; thus, do not work for the nonconvex-nonconcave case. {However, recent emerging applications in machine learning stimulate a surge of interest in nonconvex-nonconcave setting.}
It is a well-known fact that many classical first-order methods fail to converge when applied to nonconvex-nonconcave problems~\cite{grimmer2020landscape}. {On the contrary, our J-symm algorithms and their analysis do not rely on convexity assumptions, implying the potential of J-symm methods on solving nonconvex-nonconcave problems.}

In this section, we examine the behaviors of our quasi-Newton algorithms on a two-dimensional nonconvex-nonconcave example:
\begin{equation}\label{eq:nonconvex}
	\min_x \min_y L(x,y)=(x^2-1)(x^2-9)+xAy-(y^2-1)(y^2-9) \ ,
\end{equation}
where $A$ is a scalar measuring the interaction term in the minimax problem. \eqref{eq:nonconvex} is perhaps the simplest non-trivial nonconvex-nonconcave example, and it has been used in \cite{grimmer2020landscape} to illustrate the landscape of first-order methods for minimax problems.

Figure \ref{fig:nonconvex} presents the trajectories of five algorithms, EGM, Broyden's method, $J$-symmetric quasi-Newton method with fixed stepsize (Algorithm \ref{alg0}), $J$-symmetric quasi-Newton method with line-search (Algorithm \ref{algls}), and $J$-symmetric quasi-Newton method with trust-region (Algorithm \ref{tr_alg}), for solving \eqref{eq:nonconvex}. For each algorithm, we consider four interaction levels $A\in\{1,10,100,1000\}$, and we start with twelve different initial solutions $(-4~,~-2), (-4~,~0), (-4~,~2), (-2~,~-4), (-2~,~4), (0~,~-4), (0~,~4), (2~,~-4), (2~,~4), (4~,~-2), (4~,~0)$ and $(4~,~2)$. {Different color in Figure \ref{fig:nonconvex} represents different initialization.}

%\Asl{I commented the following line out since at first confused me whether you're talking about our paper or \cite{grimmer2020landscape}}
%\cite{grimmer2020landscape} describes the landscape of nonconvex-nonconcave minimax problems in three regions based on the level of interaction $A$. 
When the interaction term $A$ is small (i.e. $A=1$ as shown in the first column), all five methods converge to local minimax solutions. Furthermore, we can clearly see that quasi-Newton methods have a faster convergence compared to first-order methods such as EGM. When the interaction is medium (i.e. $A=10$ as in the second column), EGM converges to an attractive limit circle, while the four quasi-Newton methods converge quickly to some solutions. It turns out that Algorithm \ref{tr_alg} converges to a local minimizer of $\|F(z)\|$, i.e., $g_k=\nabla F(z_k) F(z_k)=0$, which is consistent with our Theorem \ref{glob_conv_thm}. 
% Indeed, the interaction middle case is a very hard case, and many other first-order methods also fall into infinite limit circles. 
As the interaction $A$ increases, we move to the third column (i.e., $A=100$). While EGM still converges to a limit circle,  quasi-Newton methods quickly converge to some solution.
%\Asl{This is not the case for J-symm and Broyden: }\blue{converge to the unique global first-order Nash equilibrium $(0,0)$}. 
In particular, both J-symm-LS and J-symm-Tr converge to the unique global first-order Nash equilibrium $(0,0)$ within a few iterations. J-symm and Broyden's methods both converge to local solutions, but compared to J-symm, Broyden's method is less stable since it moves further away in the beginning. Lastly, when the interaction term is sufficiently large (i.e., $A=1000$ as in the fourth column), EGM converges to the unique stationary point $(0,0)$. Again, J-symm-Tr converges  to $(0,0)$ within a few steps. However, while for some initial solutions, J-symm-LS has rapid convergence to $(0,0)$, for others (namely those with one dimension equal to $0$, such as  $(0,4)$), it converges extremely slowly. This is because the line-search step chooses a very small stepsize. J-symm and Broyden's method both take a large step first, and move back slowly afterwards. The initial large step is because the stepsize choice $0.01$ is initially too large for the case where $A=1000$. The slow convergence is because the step size $0.01$ is small once we construct a reasonable Jacobian.

Overall, in contrast to first-order methods, quasi-Newton methods can avoid the undesirable limit circle for this nonconvex-nonconcave example. The trust-region method J-symm-Tr shows its advantages over the others, and indeed it is the only method with global theoretical guarantees. J-symm-LS performs well in most of the cases, but it may have slow convergence when the interaction term is large. Broyden's method and J-symm have similar behaviors, but J-symm may be more stable in the medium interaction regimes.

% \section*{Acknowledgment}
% The authors would like to thank Jinwen Yang for checking the proofs in the paper.

\bibliography{./refs_all}
\bibliographystyle{amsplain}
\appendix
\section{Existing Definitions and Results Used in the Proofs.}\hfill\\
%\Asl{we agreed to comment trust-region subprolem out.}
% \begin{comment}
% 	\textbf{Trust-region Subproblem \cite{NW06}:}\\
% 	\red{It is a bit unclear why we need this.}
% 	In a trust-region regime the subproblem at iteration $k$ is defined as 
% 	%\beq\label{subprob}
% 	$$ \min_{~~~\| s \| \leq \Delta_k}~ m_k(s) \ , $$
% 	%\eeq
% 	where the model $m_k(s) = f(z_k)+ \nabla f(z_k)^Ts + 0.5s^T A_ks $ is a quadratic approximation of the objective function $f(z)$ at point $z_k$ and $A_k$ is the current approximation of $ \nabla^2 f(z_k)$.  Unlike the subproblem in a line-search method, here $A_k$ is not assumed to be positive definite in general.
% 	%\red{put the subproblem of tr here and introduce J-k}\\
% 	When $A_k$ is positive definite, the unconstrained minimizer of the trust-region subproblem
% 	%\eqref{subprob} 
% 	is its Newton direction, i.e., $-A_k^{-1}\nabla f(z_k)$. But when the constraint is active, even when $A_k$ is positive definite  finding the exact solution involves computing  the Lagrange multiplier $\lambda \geq 0$ and $s^*$ such that the following KKT condition hold: $(A_k +\lambda I)s^*= -\nabla f(z_k)$ and $\lambda(\Delta_k-\|s^*\|)=0$. Since in a purely trust-region method we set $\Delta_k=\|s^*\|$, $\lambda$ could be nonzero. %\cite{tr_step}. 
% 	Solving for exact $\lambda$ involves matrix factorization and therefore expensive, so, often we only solve the subproblem approximately.
% \end{comment}
%%%%%%%%%%%%%%%%%%%%%%%%%
%%%%%%%%%%%%%%%%%%%%%%%%%
% \textbf{Basic linear algebra results.}
{
\begin{lemma}[Sherman-Woodbury Formula]\label{S-W}
    (\cite[page 19]{horn2012matrix}) Suppose $A\in\mathbb R^{n\times n}$ is an invertible matrix and vectors $u,v\in\mathbb R^{n}$. Then $A+uv^T$ is invertible if and only is $1+v^TA^{-1}u\neq 0$. In this case,
    \[
        (A+uv^T)^{-1}=A^{-1}-\frac{A^{-1}uv^TA^{-1}}{1+v^TA^{-1}u} \ .
    \]
\end{lemma}
}

\begin{lemma}[Banach Perturbation Lemma] \label{banach} 
	(\cite[page 45]{ortega79}) Consider square matrices $A,B \in \R^{d\times d}$. Suppose that $A$ is invertible with $\|A^{-1}\|\le a$. If $\| A-B\| \le b$ and $ab < 1$, then $B$ is also invertible and  $$\|B^{-1}\|\le \frac{a}{1-ab }  \ .$$
\end{lemma}
% \textbf{Banach Perturbation Lemma \cite[page 45]{ortega79}:}\\ %\cite{DM_survey}\\

%Banach Perturbation Lemma indicates that for the any given square matrix $Q$ with $\|Q\| < 1$, then $I+Q$ is invertible and   
%\beq\label{banach}
%%\dfrac{ 1}{1+\|Q\|}	 \leq
% \|(Q+I)^{-1}\| \leq \dfrac{ 1}{1-\|Q\|} \ .	
%\eeq
%%%%%%%%%%%%%%%%%%%%%%%%%
\begin{lemma}\label{F2ineq}
	(\cite[Eq. (1.2)]{DM_survey}) 
	Consider square matrices $A,B \in \R^{d\times d}$. Then
	%\beq\label{Fl2}
	$$\| AB \|_F \leq \min \{ ~\| A \|_F\| B \|~,~ \| A \|\| B \|_F~ \} \ .  $$
	%\eeq
\end{lemma}

%%%%%%%%%%%%%%%%%%%%%%%%%
% %This Lemma is very simple to prove so we decided to remove it from the appendix
% \begin{comment}
% 	\textbf{\cite[Lemma 4]{Broyden_single}}:\\
% 	For any given square real matrix $Q$ and vectors $u$ and $v$ of the appropriate size we have:
% 	%\beq\label{fer_exp}
% 	$$\| Q(I -uv^T) \|_F^2 = \| Q \|_F^2 -2 (Qv)^TQu + \| Qu \|^2\| v \|^2 \ . $$
% 	%\eeq
% \end{comment}
%%%%%%%%%%%%%%%%%%%%%%%%%
\begin{definition}[R-superlinear and Q-superlinear Convergence Rates  \cite{globalB}] \label{RQdef}
	We say  the sequence $\{z_k\}$ is converging to $z^*$ R-superlinearly, if 
	%\beq\label{Rsup}
	$$\lim_{k \to \infty} \| z_k -z^*\|^{1/k} =0 \ ,$$
	%\eeq
	and  $\{z_k\}$ is converging to $z^*$ Q-superlinearly, if there exists a sequence  $\{q_k\}$ converging to zero such that
	%\beq\label{Qsup}
	$$ \lim_{k \to \infty} \frac{\| z_{k+1} -z^*\|}{\| z_k -z^*\|} \leq q_k \ . $$
	%\eeq
\end{definition}
% \textbf{R-superlinear and Q-superlinear Convergence Rates  \cite{globalB}:}\\

%%%%%%%%%%%%%%%%%%%%%%%%%
%The following theorem gives the well known Dennis-Moré Q-superlinear characterization identity and we restate it here for the sake of convenience. \\
\begin{theorem}[Dennis-Moré Q-superlinear Characterization Identity] \label{Qcharac}
	(\cite[Theorem 2.2]{charac})
	Let the mapping $F$ be differentiable in the open convex set $\DD$ and assume that for some $z^* \in \DD$, $\nabla F$ is continuous at $z^*$ and $\nabla F(z^*)$ is invertible. Let $\{B_k \}$ be a sequence of invertible matrices and suppose $\{z_k\}$, with $z_{k+1} = z_k -B_k^{-1}F(z_k)$, remains in $\DD$ and converges to $z^*$. Then $\{z_k\}$ converges Q-superlinearly to $z^*$ and $F(z^*) = 0$ iff
	%\beq\label{sup_char}
	$$ \lim_{k\to \infty} \dfrac{\Big\| \big(B_k - \nabla F(z^*)\big)(z_{k+1}-z_k) \Big\|}{\|z_{k+1}-z_k\|} = 0 \ . $$
\end{theorem}
% \textbf{Dennis-Moré Q-superlinear Characterization Identity \cite[Theorem 2.2]{charac}:}\\

%\eeq
%%%%%%%%%%%%%%%%%%%%%%%
% \textbf{\cite[Lemma 7.6.]{DM_survey}:}\\
% Let $u_i \in \R^{m+n}$ for $i=1,2,3,4.$ Then
% \beq\label{det_r2}
% \det(I + u_1u_2^T+ u_3u_4^T) = \Big(  1 + \langle u_1, u_2\rangle \Big)\Big(  1 + \langle u_3, u_4\rangle \Big) -\langle u_1, u_4\rangle\langle u_2, u_3\rangle \ .
% \eeq
%%%%%%%%%%%%%%%%%%%%%
% \textbf{Uniform Linear Independence}:\\%\cite{ortega79}
\begin{definition}[Uniform Linear Independence]\label{ulidef}
	(\cite[Definition 5.1.]{globalB})
	A sequence of unit vectors $\{u_j\}$ in $\R^{n+m}$ is uniformly linearly independent if there is  $\beta >0$,  $k_0\geq 0$ and $t\geq n+m$, such that for $k\geq k_0$ and $\|x\|=1$, we have:
	$$ \max \Big\{ \big|\langle x~,~ u_j \rangle\big|: ~~j=k+1,\hdots, k+t \Big\} \geq \beta \ .$$
\end{definition}

%%%%%%%%%%%%%%%%%%%%%%%
% \textbf{}:\\
\begin{theorem}\label{uliequval}
	(\cite[Theorem 5.3.]{globalB})
	Let $\{ u_k\}$ be a sequence of unit vectors in  $\R^{n+m}$. Then the following options are equivalent.
	\begin{itemize}
		\item The sequence $\{ u_k\}$ is uniformly linearly independent. 
		\item For any $\hat \beta \in [0,1) $ there is a constant $\theta \in (0,1)$ such that if $|\beta_j-1| \leq \hat \beta$ then:
		$$ \Big\|\prod_{j=k+1}^{k+t} \big( I -\beta_j u_ju_j^T \big)\Big\| \leq \theta, ~~\mathrm{for}~~k\geq k_0  \mathrm{~~and~~} t\geq n+m\ . $$
	\end{itemize}
\end{theorem}

%%%%%%%%%%%%%%%%%%%%%%%
% \textbf{}:\\
\begin{lemma}\label{bound_seq}
	(\cite[Lemma 5.5.]{globalB})
	Let  $\{ \phi_k\}$ and  $\{ \delta_k\}$ be sequences of nonnegative numbers such that $\phi_{k+t} \leq \theta \phi_k + \delta_k$ for some fixed integer $t \geq 1$ and $\theta \in (0,1)$. If  $\{ \delta_k\}$ is bounded then  $\{ \phi_k\}$ is also bounded, and if in addition,  $\{ \delta_k\}$  converges to zero, then  $\{ \phi_k\}$ converges to zero.\\
\end{lemma}

\end{document}